\documentclass[hyperref,12pt]{article}
\usepackage{graphicx} 
\usepackage[toc]{appendix}
\usepackage[autostyle]{csquotes} 
\usepackage{float} 
\usepackage{amssymb}
\usepackage{amsmath}
\usepackage{latexsym}
\usepackage{mathrsfs} 
\usepackage{bm} 
\usepackage{extarrows}
\usepackage[hang,flushmargin]{footmisc} 
\usepackage[noend]{algpseudocode}
\usepackage{fancyhdr}
\usepackage{pdfpages}
\usepackage{tikz}
\usepackage{tikz-cd}
\usepackage[all]{xy}
\usepackage{dsfont}
\usepackage{tipa}
\usepackage{changepage}
\usepackage{pifont}
\usepackage{imakeidx}
\usepackage{multicol}
\usepackage[colorlinks,linkcolor=red,anchorcolor=blue,citecolor=blue]{hyperref}
\usepackage{titletoc}
\usepackage{titlesec}
\usepackage{docmute}
\usepackage[font=footnotesize,skip=0pt,textfont=rm,labelfont=rm]{caption,subcaption}
\usepackage{booktabs} 
\usepackage{tocloft}

\usepackage{amsthm}
\newtheorem{thm}{Theorem}
\newtheorem{cor}[thm]{Corollary}
\newtheorem{prop}[thm]{Proposition}

\newtheorem{lem}[thm]{Lemma}
{
	\theoremstyle{definition}
	\newtheorem{defn}[thm]{Definition}

	\newtheorem{question}[thm]{Question}
	
	\theoremstyle{remark}
	\newtheorem{remark}[thm]{Remark}
}
\numberwithin{thm}{section}
\numberwithin{equation}{section} 

\usepackage{enumitem}

\graphicspath{{figure/}}                                 
\usepackage[a4paper]{geometry}                           
\usepackage{lmodern, mathtools}
\usepackage{old-arrows} 

\makeatletter
\newcommand{\xdbheadrightarrow}[2][]{%
	\ext@arrow 0099\xdbheadfill@{#1}{#2}}%
\newcommand{\xdbheadfill@}{%
	\arrowfill@\relbar\relbar{\mathrel{\vphantom{\rightarrow}\smash{\twoheadrightarrow}}}}
\makeatother

\usepackage{pst-plot}

\newcommand{\qp}{\mathbb{Q}_p}
\newcommand{\gal}{\mathrm{Gal}}

\newcommand{\modet}{\mathrm{Mod}^{\mathrm{\acute{e}t}}}

\newcommand\closeopen[2]{\ensuremath{[#1,#2)}}
\newcommand\openclose[2]{\ensuremath{(#1,#2]}}

\usepackage[backend=biber, style=alphabetic, maxalphanames=5, maxnames=99, useprefix=true]{biblatex}
\title{The overconvergence of multivariable $(\varphi_q,\mathcal{O}_K^{\times})$-modules at the perfectoid level}

\author{Changjiang DU}

\begin{document}
	\maketitle
	
	\begin{abstract}
		Let $K$ be a finite unramified extension of $\qp$, and $E$ a finite extension of $K$ with ring of integers $\mathcal{O}_E$. We define the overconvergence of  multivariable $(\varphi_q,\mathcal{O}_K^{\times})$-modules over $A_{\mathrm{mv},E}$ and explore some basic properties. We prove the overconvergence at the perfectoid level using the geometry of relative Fargues-Fontaine curve.
	\end{abstract}
	
	\tableofcontents
	
	\section{Introduction}
	
	Let $K$ be a finite unramified extension of $\qp$ with ring of integers $\mathcal{O}_K$ and residue field $\mathbb{F}_q$, and let $G_K\coloneq \gal(\overline{K}/K)$ be the absolute Galois group. Let $E$ be a finite extension of $K$. To study continuous $E$-representations of $G_K$, Fontaine introduced $(\varphi,\Gamma)$-modules for certain extension $K_{\infty}/K$, and showed that the category of $(\varphi,\Gamma)$-modules is equivalent to the category of $E$-representations of $G_K$ (\cite{Fontaine1990}). A common choice for $K_{\infty}$ is the cyclotomic extension; the resulting $(\varphi,\Gamma)$-modules are then called \textit{cyclotomic} $(\varphi,\Gamma)$-modules.  Cherbonnier and Colmez proved that cyclotomic $(\varphi,\Gamma)$-modules are always overconvergent; that is, every cyclotomic $(\varphi,\Gamma)$-module descends to a finite free module over the Robba ring $\mathcal{R}_E$, where $\mathcal{R}_E$ is the ring of Laurent series in one variable with coefficients in $E$ which converge in an open annulus of outer radius $1$ and inner radius $r$ for some $r<1$ (\cite{cherbonnier1998representations}). This fundamental result allows us to relate cyclotomic $(\varphi,\Gamma)$-modules and $p$-adic differential equations, and to recover many invariants in $p$-adic Hodge theory using cyclotomic $(\varphi,\Gamma)$-modules, see for example \cite{colmez1999representations}, \cite{berger2002representations}, and \cite{berger2004limites}. 
	
	One can also choose $K_{\infty}$ as the Lubin-Tate extension associated to some uniformizer of $K$ and obtain \textit{Lubin-Tate} $(\varphi,\Gamma)$-modules (see for example \cite{kisin2009galois} and \cite{schneider2017galois}). However, {Lubin-Tate} $(\varphi,\Gamma)$-modules are usually not overconvergent (\cite[Remark 5.21]{fourquaux2014triangulable}). More precisely, with a fixed Lubin-Tate formal group over $\mathcal{O}_K$ which induces a Lubin-Tate character $G_K\twoheadrightarrow \mathcal{O}_K^{\times}$, Berger defined $K$-analytic Galois representations and proved that every overconvergent {Lubin-Tate} $(\varphi,\Gamma)$-module corresponds precisely to a quotient of the tensor product of a $K$-analytic Galois representation and a Galois representation which factors through the Lubin-Tate character (\cite{berger2016multivariable}, \cite{berger2017iwasawa}). 
	
	Let $\mathcal{O}_E$ be the ring of integers of $E$ with uniformizer $\pi$ and residue field $\mathbb{F}$, and fix an embedding $\sigma_{0}: K\hookrightarrow E$. In \cite{du2025multivariablevarphiqmathcaloktimesmodulesassociatedpadic}, we introduced a new kind of $(\varphi,\Gamma)$-modules, called \textit{multivariable} $(\varphi_q,\mathcal{O}_K^{\times})$-modules since the coefficient ring $A_{\mathrm{mv},E}$ has $f\coloneq [K:\qp]$ variables, and constructed a fully faithful exact functor $D_{A_{\mathrm{mv},E}}^{(i)}$ sending continuous $\mathcal{O}_E$-representations of $G_K$ to multivariable $(\varphi_q,\mathcal{O}_K^{\times})$-modules over $A_{\mathrm{mv},E}$ for $0\leq i\leq f-1$, which is a $p$-adic analogue of the functor $D_A^{(i)}$ in \cite{breuil2023multivariable}. This kind of multivariable $(\varphi_q,\mathcal{O}_K^{\times})$-modules may be related to $p$-adic representations of $\mathrm{GL}_2(K)$, see \cite{breuil2023multivariable} and \cite{wang2024lubintatemultivariablevarphimathcaloktimesmodulesdimension}.
	
	In this paper, we consider the subring $A_{\mathrm{mv},E}^{\dagger}$ of $A_{\mathrm{mv},E}$ consisting of overconvergent elements (for a precise definition, see (\ref{defn of dagger})). Let $\modet_{(\varphi_q,\mathcal{O}_K^{\times})}\left(A_{\mathrm{mv},E}\right)$ (resp.\ $\modet_{(\varphi_q,\mathcal{O}_K^{\times})}\left(A_{\mathrm{mv},E}^{\dagger}\right)$) be the category of étale $(\varphi_q,\mathcal{O}_K^{\times})$-modules over $A_{\mathrm{mv},E}$ (resp.\ over $A_{\mathrm{mv},E}^{\dagger}$), and define
	\[j^{\dagger,*}\coloneq A_{\mathrm{mv},E}\otimes_{A_{\mathrm{mv},E}^{\dagger}}-: \modet_{(\varphi_q,\mathcal{O}_K^{\times})}\left(A_{\mathrm{mv},E}^{\dagger}\right)\to \modet_{(\varphi_q,\mathcal{O}_K^{\times})}\left(A_{\mathrm{mv},E}\right).\]
	\begin{defn}
		We say that a multivariable $(\varphi_q,\mathcal{O}_K^{\times})$-module $M$ over $A_{\mathrm{mv},E}$ is \textit{overconvergent}, if $M$ lies in the essential image of $j^{\dagger,*}$.
	\end{defn}
	We establish the following result:
	\begin{thm}[Theorem \ref{general facts}]\label{non-perfectoid case}
		The base change functor $j^{\dagger,*}$ is exact and admits a right adjoint functor $j_{*}^{\dagger}$ such that the natural transformation $\mathrm{id}\to j_{*}^{\dagger}\circ j^{\dagger,*}$ is an isomorphism. In particular, $j^{\dagger,*}$ is fully faithful.
	\end{thm}
	
	The construction of $j^{\dagger}_*$ follows Cherbonnier's thesis (\cite{cherb1996}): for an étale $(\varphi_q,\mathcal{O}_K^{\times})$-module $M$ over $A_{\mathrm{mv},E}$, $j^{\dagger}_*M$ is the union of $\varphi_q$-stable finitely generated $A_{\mathrm{mv},E}^{\dagger}$-submodules of $M$. 
	
	Before we proceed to the perfectoid level, we briefly recall the construction of $D_{A_{\mathrm{mv},E}}^{(i)}$ and the definition of perfectoid Robba rings. Recall from \cite[$\S$2.2]{breuil2023multivariable} that
	\[A_{\mathrm{mv},E}/\pi\cong A=\mathbb{F}\left(\negthinspace\left(Y_{\sigma_0}\right)\negthinspace\right)\left\langle\left(\dfrac{Y_{\sigma_{i}}}{Y_{\sigma_0}}\right)^{\pm 1}: 1\leq i\leq f-1\right\rangle,\]
	and there is a continuous $\mathbb{F}$-linear $(\varphi_q,\mathcal{O}_K^{\times})$-action on $A$.
	Let 
	\[A_{\infty}\coloneq \mathbb{F}\left(\negthinspace\left(Y_{\sigma_0}^{1/p^{\infty}}\right)\negthinspace\right)\left\langle\left(\dfrac{Y_{\sigma_{i}}}{Y_{\sigma_0}}\right)^{\pm 1/p^{\infty}}: 1\leq i\leq f-1\right\rangle\]
	be the completed perfection of the ring $A$, and let
	\[A_{\infty}'\coloneq \mathbb{F}\left(\negthinspace\left(T_{\mathrm{LT},0}^{1/p^{\infty}}\right)\negthinspace\right)\left\langle\left(\dfrac{T_{\mathrm{LT},{i}}}{T_{\mathrm{LT},0}^{p^{i}}}\right)^{\pm 1/p^{\infty}}: 1\leq i\leq f-1\right\rangle.\]
	There is a continuous $\mathbb{F}$-linear $(\varphi_q,(\mathcal{O}_K^{\times})^{f})$-action on $A_{\infty}'$, and there is a morphism $m:\mathrm{Spa}(A_{\infty}',(A_{\infty}')^{\circ})\to \mathrm{Spa}(A_{\infty},A_{\infty}^{\circ})$ which
	is a pro-étale $\Delta_1$-torsor(\cite[Proposition 2.4.4]{breuil2023multivariable}), where
	\[\Delta_1=\{(a_0,\dots,a_{f-1})\in(\mathcal{O}_K^{\times})^{f}: a_0a_1\cdots a_{f-1}=1\}.\]
	For $0\leq i\leq f-1$, there is a $(\varphi_q,(\mathcal{O}_K^{\times})^{f})$-equivariant map
	\[\mathrm{pr}_i: \mathbb{F}\left(\negthinspace\left(T_{\mathrm{LT}}^{1/p^{\infty}}\right)\negthinspace\right)\hookrightarrow A_{\infty}': T_{\mathrm{LT}}\mapsto T_{\mathrm{LT},i}\]
	where $\mathbb{F}\left(\negthinspace\left(T_{\mathrm{LT}}^{1/p^{\infty}}\right)\negthinspace\right)$ is the completed perfection of the residue field of the coefficient ring of Lubin-Tate $(\varphi_q,\mathcal{O}_K^{\times})$-modules.
	These maps lift to a diagram of $p$-adically complete rings:
	\begin{align*}
		A_{\mathrm{mv},E}\hookrightarrow W_{\mathcal{O}_E}(A_{\infty})\xhookrightarrow{m}W_{\mathcal{O}_E}(A_{\infty}')\xhookleftarrow{\mathrm{pr}_i}W_{\mathcal{O}_E}\left(\mathbb{F}\left(\negthinspace\left(T_{\mathrm{LT}}^{1/p^{\infty}}\right)\negthinspace\right)\right).
	\end{align*}
	where $W_{\mathcal{O}_E}(-)\coloneq \mathcal{O}_E\otimes_{W(\mathbb{F})}W(-)$ sends perfect $\mathbb{F}$-algebras to rings of ramified Witt vectors.
	If $\rho$ is a continuous $\mathcal{O}_E$-representation of $G_K$, the theory of Lubin-Tate $(\varphi_q,\mathcal{O}_K^{\times})$-modules gives us a $(\varphi_q,\mathcal{O}_{K}^{\times})$-module $D_{\mathrm{LT}}(\rho)$ over $W_{\mathcal{O}_E}\left(\mathbb{F}\left(\negthinspace\left(T_{\mathrm{LT}}^{1/p^{\infty}}\right)\negthinspace\right)\right)$, then the $\Delta_1$-invariant  $\left(W_{\mathcal{O}_E}(A_{\infty}')\otimes_{\mathrm{pr}_i}D_{\mathrm{LT}}(\rho)\right)^{\Delta_1}$
	is a $(\varphi_q,\mathcal{O}_{K}^{\times})$-module over $W_{\mathcal{O}_E}(A_{\infty})$, which canonically descends to an étale $(\varphi_q,\mathcal{O}_{K}^{\times})$-module $D_{A_{\mathrm{mv},E}}^{(i)}(\rho)$ over $A_{\mathrm{mv},E}$ (for details, see \cite{du2025multivariablevarphiqmathcaloktimesmodulesassociatedpadic}).
	
	For a perfectoid Huber pair $(R,R^{+})$ over $\mathbb{F}$ with a pseudo-uniformizer $\varpi\in R^{\circ\circ}$ (\cite[Definition 6.1.1]{WS2020berkeley}), we consider the affinoid analytic adic space
	\[\mathcal{Y}_{(R,R^{+}),[0,r]}\coloneq \left\{x\in \mathrm{Spa}\left(W_{\mathcal{O}_E}(R^{+}),W_{\mathcal{O}_E}(R^{+})\right): |\pi|_x^r\leq |[\varpi]|_x\neq 0 \right\},\]
	where $r\in\mathbb{Q}_{>0}$. If $1/r\in\mathbb{Z}[1/p]$, we have
	\[\mathcal{O}\left(\mathcal{Y}_{(R,R^{+}),[0,r]}\right)=W_{\mathcal{O}_E}(R^{+})\left\langle\frac{\pi}{[\varpi]^{1/r}}\right\rangle\left[\frac{1}{[\varpi]}\right]\subset W_{\mathcal{O}_E}(R).\]
	In general, for a perfectoid $\mathbb{F}$-algebra $R$ with a  pseudo-uniformizer $\varpi$ and $r\in\mathbb{Q}_{>0}$, we define the integral perfectoid Robba ring $\widetilde{\mathcal{R}}_{R}^{\mathrm{int}}$ as follows:
	\[B_{R,[0,r]}\coloneq \mathcal{O}\left(\mathcal{Y}_{(R,R^{\circ}),[0,r]}\right),\quad \widetilde{\mathcal{R}}_{R}^{\mathrm{int}}\coloneq \varinjlim_{r\to 0^{+}}B_{R,[0,r]}\subset W_{\mathcal{O}_E}(R).\]
	Note that the $\pi$-adic completions of $B_{R,[0,r]}$ and $\widetilde{\mathcal{R}}_{R}^{\mathrm{int}}$ are $W_{\mathcal{O}_E}(R)$. Using  
	the fact that $\pi$ is contained in the Jacobson radical of $\widetilde{\mathcal{R}}_{R}^{\mathrm{int}}$, we can show that:
	\begin{prop}[Corollary \ref{fin proj over int Robba is free}]\label{fin proj over int is free}
		If any finite projective module over $R$ is free, then every finite projective module over $\widetilde{\mathcal{R}}_{R}^{\mathrm{int}}$ is free.
	\end{prop}
	
	For a general perfectoid space $X$ over $\mathbb{F}$, we can construct $\mathcal{Y}_{X,\closeopen{0}{\infty}}$ by \cite[Proposition 11.2.1]{WS2020berkeley} ($\mathcal{Y}_{X,\closeopen{0}{\infty}}$ is denoted by ``$X\times\mathrm{Spa}\mathbb{Z}_p$'' there). If we fix an affinoid perfectoid space $S=\mathrm{Spa}(R_0,R_0^{+})$ over $\mathbb{F}$ with a pseudo-uniformizer $\varpi\in R_0^{+}$, then we can define $\mathcal{Y}_{X,[0,r]}\coloneq \mathcal{Y}_{X,\closeopen{0}{\infty}}\times_{\mathcal{Y}_{S,\closeopen{0}{\infty}}}\mathcal{Y}_{S,[0,r]}$ for any perfectoid space $X$ over $S$ and $r\in\mathbb{Q}_{>0}$. For details, see for example \cite[Chapter II]{fargues2024geometrizationlocallanglandscorrespondence} and \cite[$\S$11]{WS2020berkeley}. 
	\begin{prop}[{\cite[Proposition II.2.1]{fargues2024geometrizationlocallanglandscorrespondence}, \cite[Proposition 19.5.3]{WS2020berkeley}}]\label{relative FF curve is a stack}
		Let $\mathbf{Perf}_{S}$ be the category of perfectoid spaces over $S$, where $S$ is an affinoid perfectoid space, and let $r$ be a positive rational number.
		\begin{enumerate}
			\item [(i)] 
			The functor $\mathcal{O}_{[0,r]}: X\mapsto \mathcal{O}\left(\mathcal{Y}_{X,[0,r]}\right)$ is a $v$-sheaf on $\mathbf{Perf}_{S}$, and $H_v^k(S,\mathcal{O}_{[0,r]})=0$ for $k\geq 1$.
			\item [(ii)] The functor sending $X$ to the groupoid of vector bundles over $\mathcal{Y}_{X,[0,r]}$ is a $v$-stack.
		\end{enumerate}
	\end{prop}
	
	We take $Y_{\sigma_{0}}$ as the fixed pseudo-uniformizer of $A_{\infty}\subset A_{\infty}'$, and take $T_{\mathrm{LT}}$ as the fixed pseudo-uniformizer of $\mathbb{F}(\negthinspace(T_{\mathrm{LT}}^{1/p^{\infty}})\negthinspace)$ so that the rings $B_{A_{\infty},[0,r]}$, $B_{A_{\infty}',[0,r]}$ and $B_{\mathbb{F}(\negthinspace(T_{\mathrm{LT}}^{1/p^{\infty}})\negthinspace),[0,r]}$ are clear. Then we obtain a commutative diagram for $s\in\mathbb{Z}_{\geq 1}$:
	\begin{equation*}
		\begin{tikzcd}
			A_{\mathrm{mv},E}^{\dagger,s^{-}}\left[\frac{1}{Y_{\sigma_{0}}}\right]\arrow[d,hook]\arrow[r,hook] & B_{A_{\infty},[0,1/s]}\arrow[d,hook]\arrow[r,hook,"m"]&B_{A_{\infty}',[0,1/s]}\arrow[d,hook]&B_{\mathbb{F}(\negthinspace(T_{\mathrm{LT}}^{1/p^{\infty}})\negthinspace),[0,\frac{p-1}{(q-1)p^{i}s}]}\arrow[d,hook]\arrow[l,hook',"\mathrm{pr}_i"']\\
			A_{\mathrm{mv},E}^{\dagger}\arrow[d,hook]\arrow[r,hook] &\widetilde{\mathcal{R}}_{A_{\infty}}^{\mathrm{int}} \arrow[d,hook]\arrow[r,hook,"m"]& \widetilde{\mathcal{R}}_{A_{\infty}'}^{\mathrm{int}}\arrow[d,hook]&\widetilde{\mathcal{R}}_{\mathbb{F}(\negthinspace(T_{\mathrm{LT}}^{1/p^{\infty}})\negthinspace)}^{\mathrm{int}} \arrow[d,hook]\arrow[l,hook',"\mathrm{pr}_i"']\\
			A_{\mathrm{mv},E}\arrow[r,hook]& W_{\mathcal{O}_E}(A_{\infty})\arrow[r,hook,"m"]&W_{\mathcal{O}_E}(A_{\infty}')&W_{\mathcal{O}_E}(\mathbb{F}(\negthinspace(T_{\mathrm{LT}}^{1/p^{\infty}})\negthinspace))\arrow[l,hook',"\mathrm{pr}_i"'].
		\end{tikzcd}
	\end{equation*}
	For the definition of $A_{\mathrm{mv},E}^{\dagger,s^{-}}$, see (\ref{defn of s^{-1}}). Now, let $\rho$ be a finite free continuous $\mathcal{O}_E$-representation of $G_K$ of rank $d$, Proposition \ref{relative FF curve is a stack} implies that for any $0<r\in\mathbb{Z}[1/p]$, we can associate to $\rho$ a $B_{\mathbb{F}(\negthinspace(T_{\mathrm{LT}}^{1/p^{\infty}})\negthinspace),[0,r]}$-module $D_{\mathrm{LT}}^{[0,r]}(\rho)$ locally free of rank $d$. Since $B_{\mathbb{F}(\negthinspace(T_{\mathrm{LT}}^{1/p^{\infty}})\negthinspace),[0,r]}$ is a PID (\cite[Corollary 2.10]{kedlaya2016noetherian}), $D_{\mathrm{LT}}^{[0,r]}(\rho)$ is actually a free $B_{\mathbb{F}(\negthinspace(T_{\mathrm{LT}}^{1/p^{\infty}})\negthinspace),[0,r]}$-module of rank $d$.
	Applying Proposition \ref{relative FF curve is a stack} and Proposition \ref{fin proj over int is free}, we have the following perfectoid overconvergence result:
	\begin{thm}[Theorem \ref{overconvergence of mv module at the perfectoid level}, Proposition \ref{for r large is free}]\label{D_{A_infty} is finite projective}
		Let $\rho$ be a finite free continuous $\mathcal{O}_E$-representation of $G_K$ of rank $d$. Then for any $r\in \mathbb{Q}_{>0}$, 
		\[D_{A_{\infty},i}^{[0,r]}(\rho)\coloneq\left(B_{A_{\infty}',[0,r]}\otimes_{B_{\mathbb{F}(\negthinspace(T_{\mathrm{LT}}^{1/p^{\infty}})\negthinspace),[0,\frac{p-1}{(q-1)p^{i}}r]},\mathrm{pr}_i}D_{\mathrm{LT}}^{[0,\frac{p-1}{(q-1)p^{i}}r]}(\rho)\right)^{\Delta_1}\]
		is a locally free $B_{A_{\infty},[0,r]}$-module of rank $d$. Moreover, $D_{A_{\infty},i}^{[0,r]}(\rho)$ is a free $B_{A_{\infty},[0,r]}$-module of rank $d$ for $r$ sufficiently small.
	\end{thm}
	
	Our first proof of Theorem~\ref{D_{A_infty} is finite projective} was based on a complicated computation using the generalized Colmez-Sen-Tate method developed in \cite{colmez1998theorie}, \cite{cherbonnier1998representations} and \cite{berger2008familles}. This previous method only ensures that Theorem \ref{D_{A_infty} is finite projective} holds for sufficiently small $r$, whereas the new method shows that the overconvergent radius at the perfectoid level can be arbitrary. 
	
	There is yet another proof of Theorem \ref{D_{A_infty} is finite projective}. Using the same method as in \cite[Theorem 2.4.5]{kedlaya2015new} and \cite[Proposition 4.8]{de2019induction}, one can show that the base change functor from the category of finite projective étale $\varphi_q$-modules over $\widetilde{\mathcal{R}}^{\mathrm{int}}_{A_{\infty}}$ to the category of finite projective étale $\varphi_q$-modules over $W_{\mathcal{O}_E}(A_{\infty})$ is an equivalence of categories. This is a special case of \cite[Theorem 4.5.7]{kedlaya2019relative}. From this equivalence, one also easily deduces the perfectoid overconvergence.
	
	A natural open question is whether every multivariable $(\varphi_q,\mathcal{O}_K^{\times})$-module $D_{A_{\mathrm{mv},E}}^{(i)}(\rho)$ associated to a Galois representation $\rho$ is overconvergent. However, we currently do not know how to descend $\widetilde{\mathcal{R}}_{A_{\infty}}^{\mathrm{int}}\otimes D_{A_{\infty},i}^{[0,r]}(\rho)$ to a module over $A_{\mathrm{mv},E}^{\dagger}$. For $K=\qp$, this was proved by Cherbonnier and Colmez in \cite{cherbonnier1998representations} using Tate's normalized trace. However, it seems very hard to extend this method to the general case, since $\mathcal{O}_K^{\times}$ is not a one-dimensional $p$-adic Lie group and the coefficient ring $A_{\mathrm{mv},E}^{\dagger}$ has $f$ variables. 
	
	Another possible approach is to consider locally analytic vectors of $\widetilde{\mathcal{R}}_{A_{\infty}}^{\mathrm{int}}\otimes D_{A_{\infty},i}^{[0,r]}(\rho)$ for the action of $\mathcal{O}_K^{\times}$. This approach was used by Berger in \cite{berger2016multivariable} to prove the overconvergence of $K$-analytic Galois representations. The main difficulty with this approach is that it is difficult to compute locally analytic vectors of $\widetilde{\mathcal{R}}_{A_{\infty}}^{\mathrm{int}}$ (see, however, Lemma \ref{subring of qp-locally analytic vectors} for a partial result). We also note that it is extremely difficult to do concrete computations, even for the case $\rho$ is a character, since the map $A_{\infty}\hookrightarrow A_{\infty}'$ is quite complicated. At present, the only overconvergent examples we know are unramified characters of $G_K$.
	
	Since $A_{\mathrm{mv},E}^{\dagger}$ and $A_{\mathrm{mv},E}$ admit an endomorphism $\varphi$ such that $\varphi_q=\varphi^f$, by replacing $\varphi_q$ everywhere with $\varphi$, we can also define \textit{overconvergent} $(\varphi,\mathcal{O}_K^{\times})$-modules. A natural question is then whether $D^{\otimes}_{\mathrm{mv},E}(\rho)\coloneq \bigotimes^{f-1}_{i=0}D^{(i)}_{A_{\mathrm{mv},E}}(\rho)$ is an overconvergent $(\varphi,\mathcal{O}_K^{\times})$-module, where $\rho$ is a general Galois representation. Note that this is weaker than the overconvergence of each $D_{A_{\mathrm{mv},E}}^{(i)}(\rho)$. 
	
	We now describe the content of each section. In $\S$\ref{section 2}, we introduce the definition and investigate basic properties of overconvergent multivariable $(\varphi_q,\mathcal{O}_K^{\times})$-modules over $A_{\mathrm{mv},E}$, and we prove that the category of overconvergent $(\varphi_q,\mathcal{O}_K^{\times})$-modules is a full subcategory of the category of étale $(\varphi_q,\mathcal{O}_K^{\times})$-modules. In $\S$\ref{section 3}, we show the overconvergence at the perfectoid level, that is every finite free $\mathcal{O}_E$-representation $\rho$ of $G_K$ descends to a finite projective module $D_{A_{\infty},i}^{[0,r]}(\rho)$ over $B_{A_{\infty},[0,r]}$. The proof uses the geometry of the relative Fargues-Fontaine curve. Finally, in $\S$\ref{appendix}, we prove some technical lemmas which are needed in the proof of Theorem \ref{general facts}.
	
	We now introduce the main general notation (many have already been introduced, and more specific notation will be introduced throughout). 
	Let $K$ be a finite unramified extension of $\qp$ of degree $f$ with ring of integers $\mathcal{O}_K$ and residue field $\mathbb{F}_q$. We denote by $G_K$ the absolute Galois group of $K$. Let $E$ be a finite extension of $K$ with ring of integers $\mathcal{O}_E$, residue field $\mathbb{F}$ and uniformizer $\pi$. We fix a normalized $p$-adic norm of $E$ such that $|\pi|=p^{-1}$. Fix an embedding $\sigma_{0}: K\hookrightarrow E$ and set $\sigma_{i}\coloneq \sigma_{0}\circ\mathrm{Frob}_p^{i}$ for $0\leq i\leq f-1$, each $\sigma_{i}$ also induces embeddings $\mathbb{F}_q\hookrightarrow \mathbb{F}$ and $\mathcal{O}_K\hookrightarrow \mathcal{O}_E$, which are still denoted by $\sigma_{i}$. For a perfect $\mathbb{F}_p$-algebra $R$, we denote by $W(R)$ the ring of Witt vectors. For $r\in R$, we denote by $[r]\in W(R)$ the Teichmüller lift of $r$. If in addition $R$ contains $\mathbb{F}$, we define the ring of ramified Witt vectors $W_{\mathcal{O}_E}(R)\coloneq \mathcal{O}_E\otimes_{W(\mathbb{F})}W(R)$. If $R$ is a Huber ring, we denote by $R^{\circ}$ the subring consisting of power-bounded elements, and by $R^{\circ\circ}$ the ideal of $R^{\circ}$ consisting of topological nilpotent elements and we write $\mathrm{Spa}R$ for $\mathrm{Spa}(R,R^{\circ})$. If $(R,R^{+})$ is a perfectoid Huber pair, we denote its tilt by $(R^{\flat},R^{\flat+})$. If $X$ is a perfectoid space, we denote its tilt by $X^{\flat}$.
	
	\textbf{Acknowledgements}: We thank Christophe Breuil and Laurent Berger for their invaluable guidance and many helpful discussions. We thank Christophe Breuil and Laurent Berger for their careful reading of earlier drafts of this paper. We also express our gratitude to Yuanyang Jiang, Hyungseop Kim and Yicheng Zhou for many helpful discussions and suggestions, and to Johannes Anschütz and Qixiang Wang for answering questions.
	
	This work was supported by \'Ecole Doctorale de Mathématiques Hadamard.
	\section{Overconvergent $(\varphi_q,\mathcal{O}_K^{\times})$-modules over $A_{\mathrm{mv},E}$}\label{section 2}
	In this section, we introduce the coefficient ring $A_{\mathrm{mv},E}^{\dagger}$ of overconvergent elements of $A_{\mathrm{mv},E}$ and define overconvergent multivariable $(\varphi_q,\mathcal{O}_K^{\times})$-modules over $A_{\mathrm{mv},E}$. We prove that the category of overconvergent $(\varphi_q,\mathcal{O}_K^{\times})$-modules is a full subcategory of the category of étale $(\varphi_q,\mathcal{O}_K^{\times})$-modules.
	\subsection{The coefficient ring $A_{\mathrm{mv},E}^{\dagger}$}
	
	Recall that $\mathcal{O}_E[\negthinspace[\mathcal{O}_K]\negthinspace]=\mathcal{O}_E[\negthinspace[Y_{\sigma_0},\dots,Y_{\sigma_{f-1}}]\negthinspace]$, where 
	\begin{align}\label{choice of variables}
		Y_{\sigma_{i}}\coloneq \sum_{\lambda\in\mathbb{F}_q^{\times}}\sigma_{i}([\lambda])\cdot [\lambda]\in\mathcal{O}_E[\negthinspace[\mathcal{O}_K]\negthinspace]
	\end{align}
	for $K\neq \mathbb{Q}_2$. 
	For $K=\mathbb{Q}_2$, since $[1]$ is not topologically nilpotent, we choose $Y_{\sigma_{0}}\coloneq[1]-1$. 
	We define a Noetherian domain as in \cite{du2025multivariablevarphiqmathcaloktimesmodulesassociatedpadic}:
	\[A_{\mathrm{mv},E}\coloneq\left\{\sum_{\underline{n}\in \mathbb{Z}^{f-1}} f_{\underline{n}}(Y_{\sigma_0})\prod_{i=1}^{f-1}\left(\dfrac{Y_{\sigma_i}}{Y_{\sigma_0}}\right)^{n_i}: f_{\underline{n}}(Y_{\sigma_0})\in \left(\mathcal{O}_E[\negthinspace[Y_{\sigma_0}]\negthinspace]\left[\frac{1}{Y_{\sigma_0}}\right]\right)^{\wedge}, \lim\limits_{|\underline{n}|\to\infty}f_{\underline{n}}(Y_{\sigma_0})=0\right\},\]
	where $|\underline{n}|\coloneq |n_1|+\dots+|n_{f-1}|$, $\left(\mathcal{O}_E[\negthinspace[Y_{\sigma_0}]\negthinspace]\left[\frac{1}{Y_{\sigma_0}}\right]\right)^{\wedge}$ is the $p$-adic completion of $\mathcal{O}_E[\negthinspace[Y_{\sigma_0}]\negthinspace]\left[\frac{1}{Y_{\sigma_0}}\right]$, and the topology of $\left(\mathcal{O}_E[\negthinspace[Y_{\sigma_0}]\negthinspace]\left[\frac{1}{Y_{\sigma_0}}\right]\right)^{\wedge}$ is the $(\pi,Y_{\sigma_0})$-adic topology. We fix a normalized $p$-adic norm of $E$ such that $|\pi|=p^{-1}$.
	
	\begin{defn}
		For a real number $0<\rho<1$, we define
		\begin{align}
			A_E^{(\rho,1)}\coloneq\left\{\sum_{n\in\mathbb{Z}}a_nY_{\sigma_0}^n: a_n\in\mathcal{O}_E, |a_n|\rho^n\leq 1\right\}.
		\end{align}
	\end{defn}
	
	The following lemma is straightforward to verify:
	\begin{lem}
		If $\rho=p^{-\frac{1}{s}}$ for some $s\in\mathbb{Z}_{\geq 1}$, then 
		\[A_E^{(\rho,1)}\cong \dfrac{\mathcal{O}_E[\negthinspace[Y_{\sigma_0},X]\negthinspace]}{(Y_{\sigma_0}^sX-\pi)}.\]
		In particular, $A_E^{(\rho,1)}$ is a Noetherian domain.
	\end{lem}

	\begin{defn}
		Let $s$ be a positive integer, and endow $A_E^{(p^{-1/s},1)}$ with the $(Y_{\sigma_0},\frac{\pi}{Y_{\sigma_0}^s})$-adic topology. We define
		\begin{align}
			A_{\mathrm{mv},E}^{\dagger,s}\coloneq\left\{\sum_{\underline{n}\in \mathbb{Z}^{f-1}} f_{\underline{n}}(Y_{\sigma_0})\prod_{i=1}^{f-1}\left(\dfrac{Y_{\sigma_i}}{Y_{\sigma_0}}\right)^{n_i}: f_{\underline{n}}(Y_{\sigma_0})\in A_E^{(p^{-1/s},1)}, \lim\limits_{|\underline{n}|\to\infty}f_{\underline{n}}(Y_{\sigma_0})=0\right\},
		\end{align}
		and
		\begin{align}\label{defn of dagger}
			A_{\mathrm{mv},E}^{\dagger,\infty}\coloneq \bigcup_{s\geq 1}A_{\mathrm{mv},E}^{\dagger,s},\quad A_{\mathrm{mv},E}^{\dagger}\coloneq A_{\mathrm{mv},E}^{\dagger,\infty}\left[\dfrac{1}{Y_{\sigma_0}}\right].
		\end{align}
		These rings are contained in $A_{\mathrm{mv},E}$, hence are integral domains.
	\end{defn}
	
	For $x\in A_{\mathrm{mv},E}^{\dagger}$, there exists $n\geq 0$ such that $Y_{\sigma_{0}}^nx=y\in A_{\mathrm{mv},E}^{\dagger,\infty}$, hence $\pi x=\frac{\pi}{Y_{\sigma_{0}}^n}y\in A_{\mathrm{mv},E}^{\dagger,\infty}$. Hence
	\begin{align}\label{pi dagger=infty}
		\pi A_{\mathrm{mv},E}^{\dagger}\subset A_{\mathrm{mv},E}^{\dagger,\infty}.
	\end{align}
	The definition of $A_{\mathrm{mv},E}^{\dagger,s}$ implies the following. 
	
	\begin{lem}\label{A dagger s is noetherian}
		Let $s$ be a positive integer. There is an isomorphism
		\[A_{\mathrm{mv},E}^{\dagger,s}\cong\varprojlim_{n}\dfrac{A_E^{(p^{-1/s},1)}\left[\left(\frac{Y_{\sigma_i}}{Y_{\sigma_0}}\right)^{\pm 1}: 1\leq i\leq f\right]}{\left(Y_{\sigma_0},\frac{\pi}{Y_{\sigma_0}^s}\right)^n},\]
		that is, $A_{\mathrm{mv},E}^{\dagger,s}$ is the $\left(Y_{\sigma_0},\frac{\pi}{Y_{\sigma_0}^s}\right)$-adic completion of $A_E^{(p^{-1/s},1)}\left[\left(\frac{Y_{\sigma_i}}{Y_{\sigma_0}}\right)^{\pm 1}: 1\leq i\leq f\right]$. As a consequence, $A_{\mathrm{mv},E}^{\dagger,s}$ is a Noetherian domain.
	\end{lem}
	
	Similarly, let $s$ be a positive integer, we define a subring of $A_E^{(p^{-1/s},1)}$:
		\begin{align}
			A_E^{\closeopen{p^{-1/s}}{1}}\coloneq\left\{\sum_{n\in\mathbb{Z}} a_nY_{\sigma_{0}}^n\in A_E^{(p^{-1/s},1)}: \lim\limits_{|n|\to\infty}|a_n|p^{-n/s}=0\right\},
		\end{align}
		with a multiplicative norm $\lVert \sum_{n\in\mathbb{Z}} a_nY_{\sigma_{0}}^n\rVert_{s}\coloneq \sup_n |a_n|p^{-n/s}$. Note that $A_E^{\closeopen{p^{-1/s}}{1}}$ is the $\pi$-adic completion of $\mathcal{O}_E[\negthinspace[Y_{\sigma_{0}}]\negthinspace]\left[\frac{\pi}{Y_{\sigma_{0}}^s}\right]$. Then we can define the subring of $A_{\mathrm{mv},E}^{\dagger,s}$:
		\begin{align}\label{defn of s^{-1}}
			A_{\mathrm{mv},E}^{\dagger,s^{-}}\coloneq\left\{\sum_{\underline{n}\in \mathbb{Z}^{f-1}} f_{\underline{n}}(Y_{\sigma_0})\prod_{i=1}^{f-1}\left(\dfrac{Y_{\sigma_i}}{Y_{\sigma_0}}\right)^{n_i}: f_{\underline{n}}(Y_{\sigma_0})\in A_E^{\closeopen{p^{-1/s}}{1}}, \lim\limits_{|\underline{n}|\to\infty}\lVert f_{\underline{n}}(Y_{\sigma_0})\rVert_{s}=0\right\}
		\end{align}
		with a norm given by
		\begin{align}\label{defn of || ||_s}
			\left\lVert\sum_{\underline{n}} f_{\underline{n}}\prod_{i=1}^{f-1}\left(\dfrac{Y_{\sigma_i}}{Y_{\sigma_0}}\right)^{n_i}\right\rVert_s\coloneq \sup_{\underline{n}}\lVert f_{\underline{n}}\rVert_{s}.
		\end{align}
		It is straightforward to check that $\lVert\cdot\rVert_s$ is a multiplicative norm of $A_{\mathrm{mv},E}^{\dagger,s^{-}}$ satisfying the strong triangle inequality, that $A_{\mathrm{mv},E}^{\dagger,s^{-}}$ is complete for $\lVert\cdot\rVert_s$, and 
		\begin{align}
			\begin{cases}
				Y_{\sigma_0}A_{\mathrm{mv},E}^{\dagger,s^{-}}=\left\{f\in A_{\mathrm{mv},E}^{\dagger,s^{-}}: \lVert f\rVert_{s}<1\right\}=\left\{f\in A_{\mathrm{mv},E}^{\dagger,s^{-}}: \lVert f\rVert_{s}\leq p^{-1/s}\right\},\\
				\dfrac{A_{\mathrm{mv},E}^{\dagger,s^{-}}}{Y_{\sigma_0}A_{\mathrm{mv},E}^{\dagger,s^{-}}}\cong \mathbb{F}\left[\frac{\pi}{Y_{\sigma_0}^s},\left(\frac{Y_{\sigma_{i}}}{Y_{\sigma_0}}\right)^{\pm 1}: 1\leq i\leq f-1\right].
			\end{cases}
		\end{align}
		Hence $A_{\mathrm{mv},E}^{\dagger,s^{-}}$ is also a Noetherian domain. Moreover, $\lVert\cdot\rVert_s$ extends uniquely to a multiplicative norm on $A_{\mathrm{mv},E}^{\dagger,s^{-}}\left[\frac{1}{Y_{\sigma_{0}}}\right]$, which is still denoted by $\lVert\cdot\rVert_{s}$. Note that for $s_1< s_2$ we have
		\[A_{\mathrm{mv},E}^{\dagger,s_1^{-}}\subset A_{\mathrm{mv},E}^{\dagger,s_1}\subset A_{\mathrm{mv},E}^{\dagger,s_2^{-}}\subset A_{\mathrm{mv},E}^{\dagger,s_2},\]
		thus
		\[A_{\mathrm{mv},E}^{\dagger,\infty}= \bigcup_{s\geq 1}A_{\mathrm{mv},E}^{\dagger,s}= \bigcup_{s\geq 1}A_{\mathrm{mv},E}^{\dagger,s^{-}}.\]
	
	\subsection{The endomorphism $\varphi_q$ and the $\mathcal{O}_K^{\times}$-action}
	
	Recall that the multiplication by $p$ (resp.\ $\mathcal{O}_K^{\times}$) on $\mathcal{O}_K$ induces an endomorphism $\varphi$ of $ \mathcal{O}_K[\negthinspace[\mathcal{O}_K]\negthinspace]$ (resp.\ an $\mathcal{O}_K^{\times}$-action on $\mathcal{O}_K[\negthinspace[\mathcal{O}_K]\negthinspace]$). Note that there is a commutative diagram of rings
	\begin{equation}
		\begin{tikzcd}
			\mathcal{O}_K[\negthinspace[\mathcal{O}_K]\negthinspace]\arrow[rr,"\varphi"]\arrow[rd,two heads]&  &\mathcal{O}_K[\negthinspace[\mathcal{O}_K]\negthinspace]\arrow[ld,two heads]\\
			&\mathcal{O}_K & 
		\end{tikzcd}
	\end{equation}
	where the two maps $\mathcal{O}_K[\negthinspace[\mathcal{O}_K]\negthinspace]\twoheadrightarrow \mathcal{O}_K$ are given by $Y_{\sigma_i}\mapsto 0$, thus the constant term of $\varphi(Y_{\sigma_i})$ is always $0$ for any $0\leq i\leq f-1$. On the other hand, recall that 
	\[\varphi(Y_{\sigma_i})\equiv Y_{\sigma_{i-1}}^p \mod p,\]
	hence if we let $\mathfrak{m}\coloneq Y_{\sigma_0}\mathcal{O}_E[\negthinspace[\mathcal{O}_K]\negthinspace]+\cdots+Y_{\sigma_{f-1}}\mathcal{O}_E[\negthinspace[\mathcal{O}_K]\negthinspace]$, then
	\begin{align}\label{endomorphism phi_q}
		\varphi(Y_{\sigma_i})\in  Y_{\sigma_{i-1}}^p+p\cdot\mathfrak{m},\quad 0\leq i\leq f-1.
	\end{align}
	Similarly, the action of $\mathcal{O}_K^{\times}$ also satisfies
	\begin{align}\label{O_k^*-action}
		a(Y_{\sigma_i})\in  \sigma_i(a)Y_{\sigma_{i}}+p\cdot\mathfrak{m}+\mathfrak{m}^p,\quad a\in\mathcal{O}_K^{\times}, 0\leq i\leq f-1.
	\end{align}
	In fact, by the discussion in \cite[$\S$2.2]{breuil2023multivariable}, (\ref{O_k^*-action}) can be refined as follows:
	\begin{align}\label{refined O_k^*-action}
		a(Y_{\sigma_i})\in  Y_{\sigma_{i}}+p\cdot\mathfrak{m}+\mathfrak{m}^{p^n},\quad a\in1+p^n\mathcal{O}_K, 0\leq i\leq f-1, n\geq 1.
	\end{align}
	We will use (\ref{refined O_k^*-action}) to prove Lemma \ref{subring of qp-locally analytic vectors}.
	
	\begin{lem}
		Let $s\geq 1$ be an integer. Then the endomorphism $\varphi$ of $ A_{\mathrm{mv},E}$ satisfies
		\[\varphi\left(A_{\mathrm{mv},E}^{\dagger,s}\right)\subseteq A_{\mathrm{mv},E}^{\dagger,ps},\]
		and $A_{\mathrm{mv},E}^{\dagger,s}$ is stable under the $\mathcal{O}_K^{\times}$-action. As a consequence, $A_{\mathrm{mv},E}^{\dagger,\infty}$ and $A_{\mathrm{mv},E}^{\dagger}$ are stable under the endomorphism $\varphi$ and the $\mathcal{O}_K^{\times}$-action.
	\end{lem}
	\begin{proof}
		For $0\leq i\leq f-1$, we write $\varphi(Y_{\sigma_i})=Y_{\sigma_{i-1}}^p+pQ_i$ where $Q_i\in\mathfrak{m}$. In particular, $Q_i$ is topologically nilpotent in $A_{\mathrm{mv},E}^{\dagger,s}$ for any $s$. Then
		\begin{equation}
			\begin{aligned}
				\varphi\left(\dfrac{Y_{\sigma_i}}{Y_{\sigma_j}}\right)&=\dfrac{Y_{\sigma_{i-1}}^p+pQ_i}{Y_{\sigma_{j-1}}^p+pQ_j}\\
				&=\left(\dfrac{Y_{\sigma_{i-1}}}{Y_{\sigma_{j-1}}}\right)^p\cdot\dfrac{1+Q_i\frac{p}{Y_{\sigma_{i-1}}^p}}{1+Q_j\frac{p}{Y_{\sigma_{j-1}}^p}}\in A_{\mathrm{mv},E}^{\dagger,p}\subset A_{\mathrm{mv},E}^{\dagger,ps},
			\end{aligned}
		\end{equation}
		and
		\begin{equation}
			\begin{aligned}
				\varphi\left(\dfrac{\pi}{Y_{\sigma_0}^s}\right)&=\dfrac{\pi}{\left(Y_{\sigma_{f-1}}^p+pQ_0\right)^s}\\
				&=\dfrac{\pi}{Y_{\sigma_{f-1}}^{ps}}\cdot \left(1+Q_0\dfrac{p}{Y_{\sigma_{f-1}}^p}\right)^{-s}\in A_{\mathrm{mv},E}^{\dagger,ps}.
			\end{aligned}
		\end{equation}
		Hence
		\[\varphi\left(A_{\mathrm{mv},E}^{\dagger,s}\right)\subseteq A_{\mathrm{mv},E}^{\dagger,ps}.\]
		The proof for the $\mathcal{O}_K^{\times}$-action is similar, and we omit the details.
	\end{proof}
	
	Similarly, one can check that
	\begin{lem}\label{||_s and phi,Gamma}
		Let $s\geq 1$ be an integer. Then the endomorphism $\varphi$ of $ A_{\mathrm{mv},E}$ satisfies
		\[\varphi\left(A_{\mathrm{mv},E}^{\dagger,s^{-}}\right)\subseteq A_{\mathrm{mv},E}^{\dagger,ps^{-}},\]
		and $A_{\mathrm{mv},E}^{\dagger,s^{-}}$ is stable under the $\mathcal{O}_K^{\times}$-action. Moreover, for $x\in A_{\mathrm{mv},E}^{\dagger,s^{-}}$ and $\gamma\in \mathcal{O}_K^{\times}$, $\lVert\varphi(x)\rVert _{ps}=\lVert x\rVert_{s}$ and $\lVert\gamma(x)\rVert _{s}=\lVert x\rVert_{s}$.
	\end{lem}
	
	Let $A^{\circ}$ be the subring of $A=A_{\mathrm{mv},E}/\pi$ consisting of power bounded elements, then 
	\[A^{\circ}=\mathbb{F}\left[\negthinspace\left[Y_{\sigma_0}\right]\negthinspace\right]\left\langle\left(\dfrac{Y_{\sigma_{i}}}{Y_{\sigma_0}}\right)^{\pm 1}: 1\leq i\leq f-1\right\rangle.\]
	Recall that $\varphi: A\to A$ satisfies $\varphi(Y_{\sigma_{i}})=Y_{\sigma_{i-1}}^p$, thus
	\begin{align}\label{A is free over phi(A)}
		A^{\circ}=\bigoplus\limits_{0\leq n_0,\dots,n_{f-1}\leq p-1}Y_{\sigma_{0}}^{n_0}\prod_{i=1}^{f-1}\left(\dfrac{Y_{\sigma_{i}}}{Y_{\sigma_{0}}}\right)^{n_i}\varphi(A^{\circ}).
	\end{align}
	
	\begin{lem}\label{general form of basis}
		Let $a_1,\dots,a_{q}\in A_{\mathrm{mv},E}^{\dagger,1}$ such that $A^{\circ}=\bigoplus_{i=1}^q \overline{a}_i\varphi(A^{\circ})$, where $\overline{a}_i\in  A^{\circ}$ is the image of $a_i\in A_{\mathrm{mv},E}^{\dagger,1}\subset A_{\mathrm{mv},E}$ in $A=A_{\mathrm{mv},E}/\pi$ (the existence of such $a_1,\dots,a_{q}$ is ensured by (\ref{A is free over phi(A)})). Then we have
		\begin{align}\label{a basis}
			A_{\mathrm{mv},E}=\bigoplus_{i=1}^qa_i \varphi(A_{\mathrm{mv},E}),\quad A_{\mathrm{mv},E}^{\dagger,ps}=\bigoplus_{i=1}^qa_i \varphi(A_{\mathrm{mv},E}^{\dagger,s}),\quad s\geq 1.
		\end{align}
		In particular, $\varphi: A_{\mathrm{mv},E}\to A_{\mathrm{mv},E}$ and $\varphi: A_{\mathrm{mv},E}^{\dagger,s}\to A_{\mathrm{mv},E}^{\dagger,ps}$ are faithfully flat and of finite presentation.
	\end{lem}
	\begin{proof}
		The assumption $A^{\circ}=\bigoplus_{i=1}^q \overline{a}_i\varphi(A^{\circ})$ implies that
		\[A=A^{\circ}\left[\frac{1}{Y_{\sigma_{0}}^p}\right]=\bigoplus_{i=1}^q \overline{a}_i\varphi\left(A^{\circ}\left[\frac{1}{Y_{\sigma_{1}}}\right]\right)=\bigoplus_{i=1}^q \overline{a}_i\varphi(A).\]
		Then the first equality of (\ref{a basis}) follows from the fact that $A_{\mathrm{mv},E}$ is $\pi$-adically complete and $A=A_{\mathrm{mv},E}/\pi$. For the second equality of (\ref{a basis}), using (\ref{endomorphism phi_q}), we have
		\[\varphi\left(\dfrac{\pi}{Y_{\sigma_0}^s}\right)=\dfrac{\pi}{Y_{\sigma_{f-1}}^{ps}}\left(\dfrac{Y_{\sigma_{f-1}}^p}{\varphi(Y_{\sigma_0})}\right)^s\in \dfrac{\pi}{Y_{\sigma_{f-1}}^{ps}}\cdot\left(1+\mathfrak{m}\cdot A_{\mathrm{mv},E}^{\dagger,p}\right)\subset \dfrac{\pi}{Y_{\sigma_{f-1}}^{ps}}\cdot \left(A_{\mathrm{mv},E}^{\dagger,ps}\right)^{\times},\]
		hence
		\begin{align}\label{phi(pi/T^s)}
			\varphi\left(\dfrac{\pi}{Y_{\sigma_0}^s}\right)A_{\mathrm{mv},E}^{\dagger,ps}=\dfrac{\pi}{Y_{\sigma_{f-1}}^{ps}}A_{\mathrm{mv},E}^{\dagger,ps}=\dfrac{\pi}{Y_{\sigma_{0}}^{ps}}A_{\mathrm{mv},E}^{\dagger,ps}.
		\end{align}
		Note that
		\begin{align}\label{s/pi/T^s=ps/pi/T^{ps}}
			\dfrac{A_{\mathrm{mv},E}^{\dagger,s}}{\frac{\pi}{Y_{\sigma_0}^s}A_{\mathrm{mv},E}^{\dagger,s}}= A^{\circ}= \dfrac{A_{\mathrm{mv},E}^{\dagger,ps}}{\frac{\pi}{Y_{\sigma_{0}}^{ps}}A_{\mathrm{mv},E}^{\dagger,ps}}.
		\end{align}
		Combining (\ref{phi(pi/T^s)}) and (\ref{s/pi/T^s=ps/pi/T^{ps}}), we obtain
		\[\dfrac{A_{\mathrm{mv},E}^{\dagger,ps}}{\varphi\left(\frac{\pi}{Y_{\sigma_0}^s}\right)A_{\mathrm{mv},E}^{\dagger,ps}}=A^{\circ}=\bigoplus_{i=1}^q\overline{a}_i\varphi(A^{\circ})=\bigoplus_{i=1}^q \overline{a}_i\varphi\left(\dfrac{A_{\mathrm{mv},E}^{\dagger,s}}{\frac{\pi}{Y_{\sigma_0}^{s}}A_{\mathrm{mv},E}^{\dagger,s}}\right)=\dfrac{\bigoplus_{i=1}^qa_i \varphi(A_{\mathrm{mv},E}^{\dagger,s})}{\varphi\left(\frac{\pi}{Y_{\sigma_0}^s}\right)\left(\bigoplus_{i=1}^qa_i \varphi(A_{\mathrm{mv},E}^{\dagger,s})\right)}.\]
		Therefore, for any $x\in A_{\mathrm{mv},E}^{\dagger,ps}$, there exists $y_1^{(0)},\dots,y_q^{(0)}\in A_{\mathrm{mv},E}^{\dagger,s}$ and $x_1\in A_{\mathrm{mv},E}^{\dagger,ps}$ such that
		\[x-\sum_{i=1}^qa_i\varphi(y_i^{(0)})=\varphi\left(\frac{\pi}{Y_{\sigma_{0}}^s}\right)x_1.\]
		Starting again with $x_1$ instead of $x$, by induction on $n$, we can prove that for every $n\geq 0$, there exists $y_1^{(n)},\dots,y_q^{(n)}\in A_{\mathrm{mv},E}^{\dagger,s}$ and $x_{n+1}\in A_{\mathrm{mv},E}^{\dagger,ps}$ such that
		\[x-\sum_{i=1}^qa_i\varphi\left(\sum_{j=0}^{n}\dfrac{\pi^j}{Y_{\sigma_{0}}^{sj}}y_i^{(j)}\right)=\varphi\left(\frac{\pi^{n+1}}{Y_{\sigma_{0}}^{s(n+1)}}\right)x_{n+1}.\]
		Let $y_i\coloneq \sum_{n=0}^{\infty}\frac{\pi^n}{Y_{\sigma_{0}}^{sn}}y_i^{(n)}$ which is an element of $ A_{\mathrm{mv},E}^{\dagger,s}$ by Lemma \ref{A dagger s is noetherian}, then $x=\sum_{i=1}^qa_i\varphi(y_i)$, i.e.
		\[A_{\mathrm{mv},E}^{\dagger,ps}=\sum_{i=1}^qa_i \varphi(A_{\mathrm{mv},E}^{\dagger,s}).\]
		Since $A_{\mathrm{mv},E}=\bigoplus_{i=1}^qa_i \varphi(A_{\mathrm{mv},E})$, we deduce that $\sum_{i=1}^qa_i \varphi(A_{\mathrm{mv},E}^{\dagger,s})=\bigoplus_{i=1}^qa_i \varphi(A_{\mathrm{mv},E}^{\dagger,s})$, thus
		\[A_{\mathrm{mv},E}^{\dagger,ps}=\bigoplus_{i=1}^qa_i \varphi(A_{\mathrm{mv},E}^{\dagger,s}),\quad s\geq 1.\]
		Thus $\varphi: A_{\mathrm{mv},E}^{\dagger,s}\to A_{\mathrm{mv},E}^{\dagger,ps}$ is finite, flat and of finite presentation. This implies that $\varphi: \mathrm{Spec}A_{\mathrm{mv},E}^{\dagger,ps}\to \mathrm{Spec}A_{\mathrm{mv},E}^{\dagger,s}$ is open and closed (see for example \cite[\href{https://stacks.math.columbia.edu/tag/00I1}{Tag 00I1}]{stacks-project}). Note that $A_{\mathrm{mv},E}^{\dagger,s}$ is an integral domain, thus $\mathrm{Spec}A_{\mathrm{mv},E}^{\dagger,s}$ is connected, which implies that $\varphi\left(\mathrm{Spec}A_{\mathrm{mv},E}^{\dagger,ps}\right)= \mathrm{Spec}A_{\mathrm{mv},E}^{\dagger,s}$.
		Hence $\varphi: A_{\mathrm{mv},E}^{\dagger,s}\to A_{\mathrm{mv},E}^{\dagger,ps}$ is faithfully flat by \cite[\href{https://stacks.math.columbia.edu/tag/00HQ}{Tag 00HQ}]{stacks-project}. The proof of the faithfully flatness of $\varphi: A_{\mathrm{mv},E}\to A_{\mathrm{mv},E}$ is similar, and we omit the details.
	\end{proof}
	
	\begin{cor}\label{phi_q is flat}
		We have 
		\[A_{\mathrm{mv},E}^{\dagger}=\bigoplus_{n_i\in \{0,\dots,p-1\}}Y_{\sigma_{0}}^{n_0}\prod_{i=1}^{f-1}\left(\dfrac{Y_{\sigma_{i}}}{Y_{\sigma_{0}}}\right)^{n_i}\varphi(A_{\mathrm{mv},E}^{\dagger}).\]
		In particular, $\varphi: A_{\mathrm{mv},E}^{\dagger}\to A_{\mathrm{mv},E}^{\dagger}$ and $\varphi_q:A_{\mathrm{mv},E}^{\dagger}\to A_{\mathrm{mv},E}^{\dagger}$ are of finite presentation and faithfully flat, where $\varphi_q\coloneq \varphi^f$.
	\end{cor}
	\begin{proof}
		We take $\left\{a_1,\dots,a_q\right\}=\left\{Y_{\sigma_{0}}^{n_0}\prod_{i=1}^{f-1}\left(\dfrac{Y_{\sigma_{i}}}{Y_{\sigma_{0}}}\right)^{n_i}: 0\leq n_0,\dots,n_{f-1}\leq f-1\right\}$ and apply (\ref{A is free over phi(A)}) and Lemma \ref{general form of basis}, thus
		\begin{align}
			A_{\mathrm{mv},E}^{\dagger,qs}=\bigoplus_{n_i\in \{0,\dots,p-1\}}Y_{\sigma_{0}}^{n_0}\prod_{i=1}^{f-1}\left(\dfrac{Y_{\sigma_{i}}}{Y_{\sigma_{0}}}\right)^{n_i}\varphi(A_{\mathrm{mv},E}^{\dagger,s}),\quad s\geq 1.
		\end{align}
		Hence by the definition of $A_{\mathrm{mv},E}^{\dagger,\infty}$, we have
		\begin{align}
			A_{\mathrm{mv},E}^{\dagger,\infty}=\bigoplus_{n_i\in \{0,\dots,p-1\}}Y_{\sigma_{0}}^{n_0}\prod_{i=1}^{f-1}\left(\dfrac{Y_{\sigma_{i}}}{Y_{\sigma_{0}}}\right)^{n_i}\varphi(A_{\mathrm{mv},E}^{\dagger,\infty}).
		\end{align}
		For $x\in A_{\mathrm{mv},E}^{\dagger}=A_{\mathrm{mv},E}^{\dagger,\infty}\left[\frac{1}{Y_{\sigma_{0}}}\right]$, there exists $n\geq 0$ such that $Y_{\sigma_{0}}^{pn}x\in A_{\mathrm{mv},E}^{\dagger,\infty}$
		. Thus
		\[\varphi(Y_{\sigma_{1}}^n)x=\dfrac{\varphi(Y_{\sigma_{1}}^n)}{Y_{\sigma_{0}}^{pn}}Y_{\sigma_{0}}^{pn}x\in Y_{\sigma_{0}}^{pn}x\cdot\left(A_{\mathrm{mv},E}^{\dagger,p}\right)^{\times}\subset A_{\mathrm{mv},E}^{\dagger,\infty},\]
		thus 
		\[\varphi(Y_{\sigma_{1}}^n)x=\sum_{\underline{n}\in \{0,\dots,p-1\}^{f-1}}Y_{\sigma_{0}}^{n_0}\prod_{i=1}^{f-1}\left(\dfrac{Y_{\sigma_{i}}}{Y_{\sigma_{0}}}\right)^{n_i}\varphi(f_{\underline{n}}),\quad f_{\underline{n}}\in A_{\mathrm{mv},E}^{\dagger,\infty},\]
		thus
		\[x=\sum_{\underline{n}\in \{0,\dots,p-1\}^{f-1}}Y_{\sigma_{0}}^{n_0}\prod_{i=1}^{f-1}\left(\dfrac{Y_{\sigma_{i}}}{Y_{\sigma_{0}}}\right)^{n_i}\varphi(Y_{\sigma_{1}}^{-n}f_{\underline{n}}),\quad Y_{\sigma_{1}}^{-n}f_{\underline{n}}\in A_{\mathrm{mv},E}^{\dagger}.\]
		Hence we deduce the equality in the statement.
		This implies that $\varphi: A_{\mathrm{mv},E}^{\dagger}\to A_{\mathrm{mv},E}^{\dagger}$ is of finite presentation and flat and $A_{\mathrm{mv},E}^{\dagger}$ is a finite $\varphi(A_{\mathrm{mv},E}^{\dagger})$-module. Then we use the same argument as in the proof of Lemma \ref{general form of basis} to conclude.
	\end{proof}

	\subsection{The structure of étale $\varphi_q$-modules over $A_{\mathrm{mv},E}^{\dagger}$}
	Let $\varphi_q\coloneq \varphi^f$.
	\begin{defn}
		Let $R$ be one of the rings $\{A_{\mathrm{mv},E}, A_{\mathrm{mv},E}/{\pi}^n,A_{\mathrm{mv},E}^{\dagger}\}$ where $n\geq 1$. An \textit{étale} $\varphi_q$-\textit{module} over $R$ is a finite $R$-module $M$ equipped with a semi-linear endomorphism $\varphi_q$ such that the map $R\otimes_{\varphi_q,R}M\to M, x\otimes m\mapsto \varphi_q(x)m$
		is an isomorphism. We say that $M$ is an \textit{étale} $(\varphi_q,\mathcal{O}_K^{\times})$-\textit{module}, if $M$ is an étale $\varphi_q$-module equipped with a semi-linear action of $\mathcal{O}_K^{\times}$ commuting with $\varphi_q$. We denote by $\modet_{(\varphi_q,\mathcal{O}_K^{\times})}(R)$ the category of étale $(\varphi_q,\mathcal{O}_K^{\times})$-modules over $R$.
	\end{defn}

	\begin{lem}\label{pA dagger is topologically nilpotent}
		The ideal $\pi A_{\mathrm{mv},E}^{\dagger}$ is contained in the Jacobson radical of $A_{\mathrm{mv},E}^{\dagger}$.
	\end{lem}
	\begin{proof}
		By \cite[\href{https://stacks.math.columbia.edu/tag/0AME}{Tag 0AME}]{stacks-project}, we need to show that $1+\pi x$ is invertible in $A_{\mathrm{mv},E}^{\dagger}$ for any $x\in A_{\mathrm{mv},E}^{\dagger}$. Recall that
		\[A_{\mathrm{mv},E}^{\dagger}=\bigcup_{s\geq 1}A_{\mathrm{mv},E}^{\dagger,s}\left[\dfrac{1}{Y_{\sigma_0}}\right],\]
		hence there exists $s_0,n\geq 1$, $x'\in A_{\mathrm{mv},E}^{\dagger,s_0}$ such that $x=\frac{x'}{Y_{\sigma_0}^{n}}$.
		Then $\pi x=\frac{\pi}{Y_{\sigma_0}^{n}}x'\in A_{\mathrm{mv},E}^{\dagger,s}$ is topologically nilpotent for $s\geq \max\{s_0,n\}+1$ by Lemma \ref{A dagger s is noetherian}. Thus $(1+\pi x)^{-1}=1+\sum_{i=1}^{\infty}(-\pi x)^{i}$ converges in $ A_{\mathrm{mv},E}^{\dagger,s}$ for $s\geq \max\{s_0,n\}+1$, hence $1+\pi x$ is invertible in $A_{\mathrm{mv},E}^{\dagger}$.
	\end{proof}
	
	\begin{remark}
		Using Newton's method, one can check that $\left( A_{\mathrm{mv},E}^{\dagger},\pi A_{\mathrm{mv},E}^{\dagger}\right)$ is a henselian pair. We do not need this fact.
	\end{remark}
	
	\begin{cor}\label{finite projective implies free}
		If $M$ is a finite projective $A_{\mathrm{mv},E}^{\dagger}$-module, then $M$ is a free $A_{\mathrm{mv},E}^{\dagger}$-module.
	\end{cor}
	\begin{proof}
		Since $M$ is finite projective over $A_{\mathrm{mv},E}^{\dagger}$, $M/\pi M$ is finite projective over $A_{\mathrm{mv},E}^{\dagger}/\pi=A$. By \cite[Satz 3, p. 131]{lutkebohmert1977vektorraumbundel}, $M/\pi M$ is a finite free $A$-module. Let $e_1,\dots,e_r\in M$ such that the image of $e_1,\dots,e_r$ in $M/\pi M$ forms an $A$-basis of $M/\pi M$. By Lemma \ref{pA dagger is topologically nilpotent} and Nakayama's lemma, $M=\sum_{i=1}^rA_{\mathrm{mv},E}^{\dagger}e_i$. Suppose that $\sum_{i=1}^ra_ie_i=0$ for some $a_i\in A_{\mathrm{mv},E}^{\dagger}$, then by reducing modulo $\pi$, we see that $\pi\mid a_i$ for every $i$. Thus $\pi\cdot\sum_{i=1}^r\frac{a_i}{\pi}e_i=0$, and since $M$ is projective over $A_{\mathrm{mv},E}^{\dagger}$, we have $\sum_{i=1}^r\frac{a_i}{\pi}e_i=0$, hence $\pi\mid \frac{a_i}{\pi}$, i.e. $\pi^2\mid a_i$ for $i=1,\dots,r$. By induction on $n$, we have $\pi^n\mid a_i$ for every $n\geq 1$, hence $a_i=0$ for $i=1,\dots,r$.
	\end{proof}
	
	\begin{cor}\label{finite etale over A^{dagger} is free}
		Let $M$ be a $\pi$-torsion free étale $\varphi_q$-module over $A_{\mathrm{mv},E}^{\dagger}$. Then $M$ is free.
	\end{cor}
	\begin{proof}
		Using the isomorphism $A_{\mathrm{mv},E}^{\dagger}\otimes_{\varphi_q,A_{\mathrm{mv},E}^{\dagger}}M\xrightarrow{\sim }M$ and the natural isomorphism 
		\[\dfrac{A_{\mathrm{mv},E}^{\dagger}\otimes_{\varphi_q,A_{\mathrm{mv},E}^{\dagger}}M}{\pi\cdot \left(A_{\mathrm{mv},E}^{\dagger}\otimes_{\varphi_q,A_{\mathrm{mv},E}^{\dagger}}M\right)}\cong \dfrac{A_{\mathrm{mv},E}^{\dagger}}{\pi A_{\mathrm{mv},E}^{\dagger}}\otimes_{\varphi_q,A_{\mathrm{mv},E}^{\dagger}}\dfrac{M}{\pi M}\cong A\otimes_{\varphi_q,A}M/\pi M,\]
		we deduce that $A\otimes_{\varphi_q,A}M/\pi M\xrightarrow{\sim}M/\pi M$,
		hence $M/\pi M$ is a finitely generated étale $\varphi_q$-module over $A$. By \cite[Corollary 4.9]{du2025multivariablevarphiqmathcaloktimesmodulesassociatedpadic}, $M/\pi M$ is a finite free $A$-module. Then we can use the same argument as in the proof of Corollary \ref{finite projective implies free} to deduce that $M$ is a finite free $A_{\mathrm{mv},E}^{\dagger}$-module.
	\end{proof}

	\begin{prop}\label{structure of étale varphi_q-modules of finite type}
		Let $M$ be an étale $\varphi_q$-module over $A_{\mathrm{mv},E}^{\dagger}$. Then there exist $1\leq n_1<\cdots< n_r\leq \infty$, $d_1,\dots,d_r\geq 1$ such that
		\[M\cong \bigoplus_{i=1}^r\left(A_{\mathrm{mv},E}^{\dagger}/\pi^{n_i}\right)^{\oplus d_i},\]
		here $A_{\mathrm{mv},E}^{\dagger}/\pi^{\infty}\coloneq A_{\mathrm{mv},E}^{\dagger}$.
	\end{prop}
	\begin{proof}
		Let $M[\pi^n]\coloneq\left\{x\in M: \pi^nx=0\right\}$ for $n\geq 1$, and $M[\pi^{\infty}]\coloneq \bigcup_{n=1}M[\pi^n]$. There is a short exact sequence
		\[0\to M[\pi^n]\to M\xrightarrow{\times \pi^n}M.\]
		Since $\varphi_q$ is flat (Corollary \ref{phi_q is flat}), it induces a short exact sequence
		\[0\to A_{\mathrm{mv},E}^{\dagger}\otimes_{\varphi_q,A_{\mathrm{mv},E}^{\dagger}}M[\pi^n]\to A_{\mathrm{mv},E}^{\dagger}\otimes_{\varphi_q,A_{\mathrm{mv},E}^{\dagger}}M\xrightarrow{\times \pi^n}A_{\mathrm{mv},E}^{\dagger}\otimes_{\varphi_q,A_{\mathrm{mv},E}^{\dagger}}M,\]
		hence
		\begin{align}\label{p^n-torsion elements}
			A_{\mathrm{mv},E}^{\dagger}\otimes_{\varphi_q,A_{\mathrm{mv},E}^{\dagger}}M[\pi^n]=\left(A_{\mathrm{mv},E}^{\dagger}\otimes_{\varphi_q,A_{\mathrm{mv},E}^{\dagger}}M\right)[\pi^n],\quad n\geq 1.
		\end{align}
		Consider the following commutative diagram with exact rows
		\begin{equation*}
			\begin{tikzcd}
				A_{\mathrm{mv},E}^{\dagger}\otimes_{\varphi_q,A_{\mathrm{mv},E}^{\dagger}}M[\pi^{\infty}]\arrow[d]\arrow[r,hook] &A_{\mathrm{mv},E}^{\dagger}\otimes_{\varphi_q,A_{\mathrm{mv},E}^{\dagger}}M \arrow[d,"\wr"]\arrow[r,two heads]&A_{\mathrm{mv},E}^{\dagger}\otimes_{\varphi_q,A_{\mathrm{mv},E}^{\dagger}}M/M[\pi^{\infty}]\arrow[d]\\
				M[\pi^{\infty}]\arrow[r,hook] &M\arrow[r,two heads]&M/M[\pi^{\infty}].
			\end{tikzcd}
		\end{equation*}
		The left vertical map is an isomorphism by (\ref{p^n-torsion elements}), hence the right vertical map is also an isomorphism, thus $M/M[\pi^{\infty}]$ is an étale $\varphi_q$-module over $A_{\mathrm{mv},E}^{\dagger}$ without $\pi$-torsion. Note that $M/M[\pi^{\infty}]$ is of finite type since it is a quotient of $M$, thus $M/M[\pi^{\infty}]$ is a free $A_{\mathrm{mv},E}^{\dagger}$-module of finite rank by Corollary \ref{finite etale over A^{dagger} is free}. In particular, the short exact sequence of $A_{\mathrm{mv},E}^{\dagger}$-modules
		\[0\to M[\pi^{\infty}]\to M\to M/M[\pi^{\infty}]\to 0\]
		splits, hence $M\cong M/M[\pi^{\infty}]\oplus M[\pi^{\infty}] $ as $A_{\mathrm{mv},E}^{\dagger}$-modules, and $M[\pi^{\infty}]$ is also a quotient of $M$, hence is of finite type, hence $M[\pi^{\infty}]=M[\pi^n]$ for some $n\geq 1$. Thus $M[\pi^{\infty}]$ is an étale $\varphi_q$-module over $A_{\mathrm{mv},E}^{\dagger}/\pi^n=A_{\mathrm{mv},E}/\pi^n$. Then we apply \cite[Proposition 4.10 ]{du2025multivariablevarphiqmathcaloktimesmodulesassociatedpadic} to conclude.
	\end{proof}
	
	By \cite[ Lemma 4.11]{du2025multivariablevarphiqmathcaloktimesmodulesassociatedpadic}, $\modet_{(\varphi_q,\mathcal{O}_K^{\times})}(A_{\mathrm{mv},E})$ and $\modet_{(\varphi_q,\mathcal{O}_K^{\times})}(A_{\mathrm{mv},E}/\pi^n)$ are abelian categories. This is also true for $A_{\mathrm{mv},E}^{\dagger}$.
	
	\begin{lem}\label{A^{dagger} is abelian}
		The category $\modet_{(\varphi_q,\mathcal{O}_K^{\times})}(A_{\mathrm{mv},E}^{\dagger})$ is an abelian category.
	\end{lem}
	\begin{proof}
		Let $f: M\to N$ be a morphism of étale $(\varphi_q,\mathcal{O}_K^{\times})$-modules over $A_{\mathrm{mv},E}^{\dagger}$, we need to show that $\ker f$ and $\mathrm{coker}f$ are also étale $(\varphi_q,\mathcal{O}_K^{\times})$-modules. First, Proposition \ref{structure of étale varphi_q-modules of finite type} implies that $M,N$ are of finite presentation over $A_{\mathrm{mv},E}^{\dagger}$, hence $\ker f$ and $\mathrm{coker}f$ are of finite type over $A_{\mathrm{mv},E}^{\dagger}$. The endomorphism $\varphi_q$ and the $\mathcal{O}_K^{\times}$-action on $\ker f$ and $\mathrm{coker}f$ are clear, hence we only need to show the étaleness of $\ker f$ and $\mathrm{coker}f$. Since $\varphi_q$ is flat (Corollary \ref{phi_q is flat}), we have a commutative diagram with exact rows
		\begin{equation}
			\begin{tikzcd}
				0\arrow[r]&A_{\mathrm{mv},E}^{\dagger}\otimes_{\varphi_q}\ker f\arrow[r]\arrow[d]&A_{\mathrm{mv},E}^{\dagger}\otimes_{\varphi_q}M\arrow[r] \arrow[d,"\wr"]&A_{\mathrm{mv},E}^{\dagger}\otimes_{\varphi_q}N \arrow[d,"\wr"]\\
				0\arrow[r]&\ker f\arrow[r]&M\arrow[r,"f"]&N,
			\end{tikzcd}
		\end{equation}
		then the étaleness of $\ker f$ follows. The proof of the étaleness of $\mathrm{coker}f$ is similar, and we omit the details.
	\end{proof}
	
	\subsection{Overconvergent multivariable $(\varphi_q,\mathcal{O}_K^{\times})$-modules}

	
	
	The $(\varphi_q,\mathcal{O}_K^{\times})$-equivariant inclusion $A_{\mathrm{mv},E}^{\dagger}\hookrightarrow A_{\mathrm{mv},E}$ induces a functor
	\begin{equation}\label{defn of j^*}
		\begin{aligned}
			j^{\dagger,*}: \modet_{(\varphi_q,\mathcal{O}_K^{\times})}(A_{\mathrm{mv},E}^{\dagger}) &\to 
			\modet_{(\varphi_q,\mathcal{O}_K^{\times})} (A_{\mathrm{mv},E})\\
			M&\mapsto A_{\mathrm{mv},E}\otimes_{A_{\mathrm{mv},E}^{\dagger}}M.
		\end{aligned}
	\end{equation}
	By Proposition \ref{inclusion for dagger is flat}, $j^{\dagger,*}$ is an exact functor. In this section, we will construct a right adjoint functor $j_*^{\dagger}$ of $j^{\dagger,*}$ and investigate some basic properties of $j^{\dagger,*}$. In particular, $j_*^{\dagger}$ is left exact. The construction of $j_*^{\dagger}$ follows Cherbonnier's thesis (\cite{cherb1996}).
	
	\begin{defn}\label{defn of F_dagger}
		Let $M$ be an étale $(\varphi_q,\mathcal{O}_K^{\times})$-module over $A_{\mathrm{mv},E}$, $s\geq 1$. We define
		\begin{align*}
			F_{\dagger}(M)&\coloneq\left\{N: N\ \text{is a finitely generated }A_{\mathrm{mv},E}^{\dagger}\text{-submodule of}\ M, \varphi_q(N)\subseteq N\right\},\\
			F_{\dagger,s}(M)&\coloneq\left\{N: N\ \text{is a finitely generated }A_{\mathrm{mv},E}^{\dagger}\text{-submodule of}\ M, \varphi_q(N)\subseteq A_{\mathrm{mv},E}^{\dagger,qs}N\right\},\\
			F_{\dagger,\infty}(M)&\coloneq\left\{N: N\ \text{is a finitely generated }A_{\mathrm{mv},E}^{\dagger}\text{-submodule of}\ M, \varphi_q(N)\subseteq N\right\}.
		\end{align*}
		Then we define
		\begin{align}
			j^{\dagger}_{*}(M)\coloneq \bigcup_{N\in F_{\dagger}(M)}N,\quad j^{\dagger,s}_{*}(M)\coloneq \bigcup_{N\in F_{\dagger,s}(M)}N,\quad j^{\dagger,\infty}_{*}(M)\coloneq\bigcup_{N\in F_{\dagger,\infty}(M)}N.
		\end{align}
	\end{defn}
	
	\begin{lem}\label{relation between j,j_s,j_infty}
		Let $M$ be an étale $(\varphi_q,\mathcal{O}_K^{\times})$-module over $A_{\mathrm{mv},E}$. Then
		\begin{align}\label{relation between s,infty,dagger}
			j_*^{\dagger}(M)=j_*^{\dagger,\infty}(M)\left[\dfrac{1}{Y_{\sigma_0}}\right],\quad j_*^{\dagger,\infty}(M)=\bigcup_{s\geq 1}j_*^{\dagger,s}(M).
		\end{align}
	\end{lem}
	\begin{proof}
		For every $N\in F_{\dagger}(M)$, assume that $N=\sum_{i=1}^r A_{\mathrm{mv},E}^{\dagger}e_i$, and $\varphi_q(e_i)=\sum_{j}a_{ij}e_j$ for some $a_{ij}\in A_{\mathrm{mv},E}^{\dagger}$ (these $a_{ij}$ may not be unique). Then for $n\geq 1$,
		\[\varphi_q\left(Y_{\sigma_0}^ne_i\right)=\sum_{j=1}^r a_{ij}\dfrac{\varphi_q(Y_{\sigma_0}^n)}{Y_{\sigma_0}^n}Y_{\sigma_0}^ne_j.\]
		By (\ref{endomorphism phi_q}), $\dfrac{\varphi_q(Y_{\sigma_0}^n)}{Y_{\sigma_0}^n}\in Y_{\sigma_0}^{(q-1)n}+pA_{\mathrm{mv},E}^{\dagger,1}$. Since $pA_{\mathrm{mv},E}^{\dagger}\subseteq A_{\mathrm{mv},E}^{\dagger,\infty}$ by (\ref{pi dagger=infty}), we have $a_{ij}\dfrac{\varphi_q(Y_{\sigma_0}^{n_{ij}})}{Y_{\sigma_0}^{n_{ij}}}\in A_{\mathrm{mv},E}^{\dagger,\infty}$
		for $n_{ij}\gg 0$. Thus replacing $e_0,\dots,e_r$ by $Y_{\sigma_0}^{n_0}e_0,\dots,Y_{\sigma_0}^{n_0}e_r$ with $n_0\coloneq \max_{i,j}n_{ij}$, we may assume that $\varphi_q(e_i)=\sum_{j=1}^r a_{ij}e_j$ for some $a_{ij}\in A_{\mathrm{mv},E}^{\dagger,\infty}$. Then $N_{\infty}\coloneq\sum_{i=1}^r A_{\mathrm{mv},E}^{\dagger,\infty}e_i$ is an element of $F_{\dagger,\infty}(M)$ and $N_{\infty}\left[\frac{1}{Y_{\sigma_0}}\right]=N$.
		Thus
		\[j_*^{\dagger}(M)=j_*^{\dagger,\infty}(M)\left[\dfrac{1}{Y_{\sigma_0}}\right].\]
		For any $N\in F_{\dagger,\infty}(M)$, assume that $N=\sum_{i=1}^r A_{\mathrm{mv},E}^{\dagger,\infty}e_i$, and $\varphi_q(e_i)=\sum_{j}a_{ij}e_j$ for some $a_{ij}\in A_{\mathrm{mv},E}^{\dagger,\infty}$. Since $A_{\mathrm{mv},E}^{\dagger,\infty}=\bigcup_{s\geq 1}A_{\mathrm{mv},E}^{\dagger,s}$, there exists $s(N)\geq 1$ such that $a_{i,j}\in A_{\mathrm{mv},E}^{\dagger,s(N)}\subseteq A_{\mathrm{mv},E}^{\dagger,s}$ for any $1\leq i,j\leq r$, $s\geq s(N)$. Then $\sum_{i=1}^rA_{\mathrm{mv},E}^{\dagger,s}e_i\in F_{\dagger,s}(M)$ and
		\[N=\bigcup_{s\geq s(N)}\sum_{i=1}^rA_{\mathrm{mv},E}^{\dagger,s}e_i.\]
		This finishes the proof.
	\end{proof}
	
	
	Note that for $N\in F_{\dagger}(M)$ and $a\in\mathcal{O}_K^{\times}$, since $\varphi_q$ commutes with the action of $\mathcal{O}_K^{\times}$, $a(N)$ is also stable under $\varphi_q$. Hence $a(N)\in F_{\dagger}(M)$. This implies that $j^{\dagger}_{*}(M)$ is stable under the $\mathcal{O}_K^{\times}$-action. For the same reason, $j^{\dagger,s}_{*}(M)$ and $j^{\dagger,\infty}_{*}(M)$ are also stable under the $\mathcal{O}_K^{\times}$-action. The following lemma is important in the proof of the finiteness of $j^{\dagger}_{*}(M)$:
	
	\begin{lem}\label{1/p j^s cap M is contained in T^-s j^s}
		Let $M$ be a finite free étale $(\varphi_q,\mathcal{O}_K^{\times})$-module over $A_{\mathrm{mv},E}$, $s\geq 1$. Then
		\[\left(\dfrac{1}{\pi}j_{*}^{\dagger,s}(M)\right)\cap M = \dfrac{1}{Y_{\sigma_0}^s}j_{*}^{\dagger,s}(M).\]
	\end{lem}
	\begin{proof}
		Let $N\in F_{\dagger,s}(M)$ (Definition \ref{defn of F_dagger}), and let $N_1\coloneq \left(\frac{Y_{\sigma_0}^{s}}{\pi}N\right) \cap M$. Note that every $x\in N$ can be written as $\frac{\pi}{Y_{\sigma_0}^{s}}\cdot \frac{Y_{\sigma_0}^{s}}{\pi}x$ with $\frac{Y_{\sigma_0}^{s}}{\pi}x\in \frac{Y_{\sigma_0}^{s}}{\pi}N$, hence $N\subseteq N_1$.
		Since $N_1\subseteq  \frac{Y_{\sigma_0}^{s}}{\pi}N$ and $A_{\mathrm{mv},E}^{\dagger,s}$ is Noetherian (Lemma \ref{A dagger s is noetherian}), $N_1$ is a finite $A_{\mathrm{mv},E}^{\dagger,s}$-module.
		By (\ref{endomorphism phi_q}),
		\[\dfrac{\varphi_q(Y_{\sigma_0}^s)}{Y_{\sigma_0}^{qs}}\in \left(1+\dfrac{p\mathfrak{m}}{Y_{\sigma_0}^q}\right)^s\subset A_{\mathrm{mv},E}^{\dagger,q}\subseteq A_{\mathrm{mv},E}^{\dagger,qs}.\]
		And note that
		$A_{\mathrm{mv},E}^{\dagger,qs}\subseteq  \frac{\pi}{Y_{\sigma_0}^{qs}}A_{\mathrm{mv},E}^{\dagger,qs}+A_{\mathrm{mv},E}^{\dagger,s}$ by Lemma \ref{A dagger s is noetherian}, thus we have
		\begin{align}\label{varphi_q(T^s)qs=pi qs+ss}
			\varphi_q(Y_{\sigma_0}^s)A_{\mathrm{mv},E}^{\dagger,qs}\subseteq Y_{\sigma_0}^{qs}A_{\mathrm{mv},E}^{\dagger,qs}\subseteq \pi A_{\mathrm{mv},E}^{\dagger,qs}+Y_{\sigma_0}^{s}A_{\mathrm{mv},E}^{\dagger,s}.
		\end{align}
		Therefore,
		\begin{align*}
			\varphi_q(N_1)&\subseteq \left(\varphi_q\left(\dfrac{Y_{\sigma_0}^s}{\pi}\right)A_{\mathrm{mv},E}^{\dagger,qs}N\right)\cap M\\
			&\subseteq \left(A_{\mathrm{mv},E}^{\dagger,qs}N+\dfrac{Y_{\sigma_0}^s}{\pi}N\right)\cap M\\ &=A_{\mathrm{mv},E}^{\dagger,qs}N+\left(\dfrac{Y_{\sigma_0}^s}{\pi}N\cap M\right)\subseteq A_{\mathrm{mv},E}^{\dagger,qs}N_1,
		\end{align*}
		where the first inclusion follows from $N\in F_{\dagger,s}(M)$ and $\varphi_q(M)\subseteq M$, the second inclusion follows from (\ref{varphi_q(T^s)qs=pi qs+ss}), and the last equality follows from $A_{\mathrm{mv},E}^{\dagger,qs}N\subseteq M$, and the last inclusion is due to $N\subseteq N_1$. Therefore, $N_1\in F_{\dagger,s}(M)$. This implies that
		\begin{align*}
			\left(\dfrac{Y_{\sigma_0}^{s}}{\pi}j_{*}^{\dagger,s}(M)\right) \cap M \subseteq j_{*}^{\dagger,s}(M).
		\end{align*}
		Since every $x\in j_{*}^{\dagger,s}(M)$ can be written as $\frac{\pi}{Y_{\sigma_0}^{s}}\cdot \frac{Y_{\sigma_0}^{s}}{\pi}x$ where $\frac{Y_{\sigma_0}^{s}}{\pi}x\in \frac{Y_{\sigma_0}^{s}}{\pi}j_{*}^{\dagger,s}(M)$, hence $j_{*}^{\dagger,s}(M)\subseteq \left(\frac{Y_{\sigma_0}^{s}}{\pi}j_{*}^{\dagger,s}(M)\right) \cap M$, then the required equality follows.
	\end{proof}

	\begin{cor}\label{j^dagger cap pM =p j^dagger}
		Let $M$ be a finite free étale $(\varphi_q,\mathcal{O}_K^{\times})$-module over $A_{\mathrm{mv},E}$. Then
		\[ j_{*}^{\dagger}(M)\cap \pi M=\pi j_{*}^{\dagger}(M).\]
	\end{cor}
	\begin{proof}
		The inclusion $\pi j_{*}^{\dagger}(M)\subseteq j_{*}^{\dagger}(M)\cap \pi M$ is obvious. Conversely, let $x\in M$ with $\pi x\in j_{*}^{\dagger}(M)$, we need to show that $x\in j_{*}^{\dagger}(M)$. By (\ref{relation between s,infty,dagger}), there exists $n\geq 0$ such that $Y_{\sigma_0}^n\pi x\in j_{*}^{\dagger,\infty}(M)=\bigcup_{s\geq 1} j_{*}^{\dagger,s}(M)$. Thus there exists $s\geq 1$ such that $\pi Y_{\sigma_0}^nx\in j_{*}^{\dagger,s}(M)$, hence $Y_{\sigma_0}^nx\in \left(\frac{1}{\pi}j_{*}^{\dagger,s}(M)\right)\cap M\subseteq \frac{1}{Y_{\sigma_0}^s}j_{*}^{\dagger,s}(M)$ (Lemma \ref{1/p j^s cap M is contained in T^-s j^s}),
		thus $Y_{\sigma_0}^{s+n}x\in j_{*}^{\dagger,s}(M)$, thus $x\in j_{*}^{\dagger,s}(M)\left[\frac{1}{Y_{\sigma_0}}\right]\subset j_{*}^{\dagger}(M)$.
	\end{proof}
	
	\begin{lem}\label{mod p generators are enough}
		Let $M$ be a finite free étale $(\varphi_q,\mathcal{O}_K^{\times})$-module over $A_{\mathrm{mv},E}$, $s\geq 1$. If the image of $e_1,\dots,e_r\in j_{*}^{\dagger,s}(M)$ generates $\frac{j_{*}^{\dagger,s}(M)}{j_{*}^{\dagger,s}(M)\cap \pi M}$ as an $A^{\circ}$-module, then $j_{*}^{\dagger,s}(M)=\sum_{i=1}^rA_{\mathrm{mv},E}^{\dagger,s}e_i$.
	\end{lem}
	\begin{proof}
		By Lemma \ref{1/p j^s cap M is contained in T^-s j^s}, 
		\[j_{*}^{\dagger,s}(M)\cap \pi M=\dfrac{\pi}{Y_{\sigma_0}^s}j_{*}^{\dagger,s}(M).\]
		Thus for $x\in j_{*}^{\dagger,s}(M)$, there exists $y_0\in \sum_{i=1}^rA_{\mathrm{mv},E}^{\dagger,s}e_i$ such that $x-y_0\in \dfrac{\pi}{Y_{\sigma_0}^s}j_{*}^{\dagger,s}(M)$. By induction on $n\geq 0$, we can prove that there exists $y_n\in \sum_{i=1}^rA_{\mathrm{mv},E}^{\dagger,s}e_i$ such that
		\[x-\sum_{j=1}^{n}\left(\dfrac{\pi}{Y_{\sigma_0}^s}\right)^jy_j \in \left(\dfrac{\pi}{Y_{\sigma_0}^s}\right)^{n+1}j_{*}^{\dagger,s}(M), \quad n\geq 0.\]
		For each $y_n$, we write $y_n=\sum_{i=1}a^{(n)}_ie_i$, $a^{(n)}_i\in A_{\mathrm{mv},E}^{\dagger,s}$ (these $a^{(n)}_i$ may not be unique), then we define
		\[a_i\coloneq \sum_{n=0}^{\infty}a^{(n)}_i\left(\dfrac{\pi}{Y_{\sigma_0}^s}\right)^n,\quad i=1,\dots,r.\]
		Note that $A_{\mathrm{mv},E}^{\dagger,s}$ is complete for the $\dfrac{\pi}{Y_{\sigma_0}^s}$-adic topology (Lemma \ref{A dagger s is noetherian}), hence $a_i\in A_{\mathrm{mv},E}^{\dagger,s}$ for $i=1,\dots,r$. Moreover, we can prove that
		\[x-\sum_{i=1}^ra_ie_i \in \bigcap_{n=0}^{\infty}\left(\dfrac{\pi^n}{Y_{\sigma_0}^{ns}} j_{*}^{\dagger,s}(M)\right)\subseteq \bigcap_{n=0}^{\infty}\left(\dfrac{\pi^n}{Y_{\sigma_0}^{ns}}\right) M=\bigcap_{n=0}^{\infty}\pi^nM=0,\]
		thus $x\in \sum_{i=1}^rA_{\mathrm{mv},E}^{\dagger,s}e_i$.
	\end{proof}
	
	Now we show the finiteness of $j^{\dagger}_{*}(M)$.
	\begin{prop}\label{finiteness of dagger, s, infty}
		If $M$ is a finite free étale $(\varphi_q,\mathcal{O}_K^{\times})$-module over $A_{\mathrm{mv},E}$, then there exists $s_0$ such that, for $s\geq s_0$, $j^{\dagger,s}_{*}(M)$ is of finite type over $A_{\mathrm{mv},E}^{\dagger,s}$ and 
		\[A_{\mathrm{mv},E}^{\dagger,\infty }\cdot j^{\dagger,s}_{*}(M)= j^{\dagger,\infty}_{*}(M),\quad A_{\mathrm{mv},E}^{\dagger}\cdot j^{\dagger,s}_{*}(M)= j^{\dagger}_{*}(M).\]
		In particular, $j^{\dagger,\infty}_{*}(M)$ and $j^{\dagger}_{*}(M)$ are of finite type.
	\end{prop}
	\begin{proof}
		By Lemma \ref{A dagger s is noetherian}, one can easily check that
		\begin{align}
			\dfrac{A_{\mathrm{mv},E}^{\dagger,s}}{A_{\mathrm{mv},E}^{\dagger,s}\cap \pi A_{\mathrm{mv},E}}=\dfrac{A_{\mathrm{mv},E}^{\dagger,\infty}}{A_{\mathrm{mv},E}^{\dagger,\infty}\cap \pi A_{\mathrm{mv},E}}=A^{\circ}.
		\end{align}
		For $N\in F_{\dagger,\infty}(M)$ (Definition \ref{defn of F_dagger}), $\frac{N}{N\cap \pi M}$ is a finite $A^{\circ}$-module stable under $\varphi_q$, hence $\frac{N}{N\cap \pi M}\subseteq \left(M/\pi M\right)^{\circ}$ (for the definition of $\left(M/\pi M\right)^{\circ}$, see (\ref{defn of M^circ})). Therefore, 
		\[\dfrac{j_{*}^{\dagger,\infty}(M)}{j_{*}^{\dagger,\infty}(M)\cap \pi M}=\bigcup_{N\in F_{\dagger,\infty}(M)}\dfrac{N}{N\cap \pi M}\subseteq \left(M/\pi M\right)^{\circ}.\]
		By Lemma \ref{A^{circ}-submodule is finite}, $\left(M/\pi M\right)^{\circ}$ is a finite $A^{\circ}$-module. As $A^{\circ}$ is Noetherian, it follows that $\frac{j_{*}^{\dagger,\infty}(M)}{j_{*}^{\dagger,\infty}(M)\cap \pi M}$ is also a finite $A^{\circ}$-module. Take $e_1,\dots,e_r\in j_{*}^{\dagger,\infty}(M)$ such that the image of $e_1,\dots,e_r$ generates $\frac{j_{*}^{\dagger,\infty}(M)}{j_{*}^{\dagger,\infty}(M)\cap \pi M}$, there exists $s_0$ such that $e_i \in j_*^{\dagger,s}(M)$ for every $s\geq s_0$, $i=1,\dots,r$, hence we have
		\[\sum_{i=1}^r A^{\circ}\overline{e}_i\subseteq \dfrac{j_{*}^{\dagger,s}(M)}{j_{*}^{\dagger,s}(M)\cap \pi M}\subseteq \dfrac{j_{*}^{\dagger,\infty}(M)}{j_{*}^{\dagger,\infty}(M)\cap \pi M} = \sum_{i=1}^r A^{\circ}\overline{e}_i,\quad s\geq s_0,\]
		which implies
		\[\dfrac{j_{*}^{\dagger,s}(M)}{j_{*}^{\dagger,s}(M)\cap \pi M}= \dfrac{j_{*}^{\dagger,\infty}(M)}{j_{*}^{\dagger,\infty}(M)\cap \pi M} = \sum_{i=1}^r A^{\circ}\overline{e}_i,\quad s\geq s_0.\]
		Hence $j_*^{\dagger,s}(M)=\sum_{i=1}^r A_{\mathrm{mv},E}^{\dagger,s}e_i$ for every $s\geq s_0$ by Lemma \ref{mod p generators are enough}. Therefore, $j_*^{\dagger,s}(M)$ is of finite type for $s\geq s_0$ and $A_{\mathrm{mv},E}^{\dagger,s}\cdot j_*^{\dagger,s_0}(M)=j_*^{\dagger,s}(M)$. Then we apply (\ref{relation between s,infty,dagger}) to conclude.
	\end{proof}
	
	Now we can state and prove the main result of this section:
	
	\begin{thm}\label{general facts}
		\begin{enumerate}[label=(\roman*),ref={\theprop.(\roman*)}]
			\item \label{part 1}
			If $M$ is an étale $(\varphi_q,\mathcal{O}_K^{\times})$-module over $A_{\mathrm{mv},E}$, killed by ${\pi}^n$ for some $n\geq 1$, then
			\[j^{\dagger}_{*}(M)=M.\]
			\item \label{part 2}
			If $M$ is an étale $(\varphi_q,\mathcal{O}_K^{\times})$-module over $A_{\mathrm{mv},E}$ free of rank $d$, then there exists $s_0$ such that, for $s\geq s_0$, $j^{\dagger,s}_{*}(M)$ is of finite type over $A_{\mathrm{mv},E}^{\dagger,s}$ and there is an equality
			\[A_{\mathrm{mv},E}^{\dagger}\cdot j^{\dagger,s}_{*}(M)= j^{\dagger}_{*}(M).\]
			Moreover, $j^{\dagger}_{*}(M)$ is an étale $(\varphi_q,\mathcal{O}_K^{\times})$-module over $A_{\mathrm{mv},E}^{\dagger}$, free of rank $r\leq d$.
			\item \label{part 3}
			If $M$ is an étale $(\varphi_q,\mathcal{O}_K^{\times})$-module over $A_{\mathrm{mv},E}$, let $M_{\mathrm{tor}}$ be the submodule of $M$ consisting of $\pi$-power torsion elements, then there is a short exact sequence
			\[0\to j^{\dagger}_{*}(M_{\mathrm{tor}})\to j^{\dagger}_{*}(M)\to j^{\dagger}_{*}(M/M_{\mathrm{tor}})\to 0.\]
			\item \label{part 4}
			If $M$ is an étale $(\varphi_q,\mathcal{O}_K^{\times})$-module over $A_{\mathrm{mv},E}$, then $j^{\dagger}_{*}(M)$ is an étale $(\varphi_q,\mathcal{O}_K^{\times})$-module over $A_{\mathrm{mv},E}^{\dagger}$ of finite presentation.
			\item \label{part 5}The functor $j^{\dagger}_{*}$ is right adjoint to $j^{\dagger,*}$ and the natural transformation $\mathrm{id}\to j_{*}^{\dagger}\circ j^{\dagger,*}$ is an isomorphism. 
			In particular, $j^{\dagger,*}$ is fully faithful and $j^{\dagger}_{*}$ is left exact.
			\item \label{part 6}If $M$ is an étale $(\varphi_q,\mathcal{O}_K^{\times})$-module over $A_{\mathrm{mv},E}$, then the natural map $j^{\dagger,*}j_{*}^{\dagger}M\to M$
			is injective.
		\end{enumerate}
	\end{thm}
	\begin{proof}
		\begin{enumerate}
			\item [(i)]
			It is a direct consequence of the fact $A_{\mathrm{mv},E}/\pi^{n}=A_{\mathrm{mv},E}^{\dagger}/\pi^{n}$, $n\geq 1$.
			\item [(ii)]
			The first part is Proposition \ref{finiteness of dagger, s, infty}. By Corollary \ref{j^dagger cap pM =p j^dagger}, we have $\frac{j_*^{\dagger}(M)}{\pi j_*^{\dagger}(M)}=\frac{j_*^{\dagger}(M)}{j_*^{\dagger}(M)\cap \pi M}$, which is therefore an $A$-submodule of $M/\pi M$, stable under $\varphi_q$ and the $\mathcal{O}_K^{\times}$-action. Hence $\frac{j_*^{\dagger}(M)}{\pi j_*^{\dagger}(M)}$ is an étale $(\varphi_q,\mathcal{O}_K^{\times})$-module over $A$ by Corollary \ref{submodule of etale module over A is etale}. The étaleness of $\frac{j_*^{\dagger}(M)}{\pi j_*^{\dagger}(M)}$ implies that
			\[j_*^{\dagger}(M)=A_{\mathrm{mv},E}^{\dagger}\cdot\varphi_q\left(j_*^{\dagger}(M)\right)+\pi j_*^{\dagger}(M).\]
			Hence by Lemma \ref{pA dagger is topologically nilpotent} and Nakayama's lemma, $j_*^{\dagger}(M)=A_{\mathrm{mv},E}^{\dagger}\cdot\varphi_q\left(j_*^{\dagger}(M)\right)$, i.e. $j_*^{\dagger}(M)$ is an  étale $(\varphi_q,\mathcal{O}_K^{\times})$-module over $A_{\mathrm{mv},E}^{\dagger}$, thus is free of rank $r=\mathrm{rank}_{A}\frac{j_*^{\dagger}(M)}{\pi j_*^{\dagger}(M)}\leq \mathrm{rank}_{A}M/\pi M=d$ by Corollary \ref{finite etale over A^{dagger} is free}.
			\item [(iii)]
			Note that we have a commutative diagram of complexes
			\begin{equation*}
				\begin{tikzcd}
					0\arrow[r] &j_*^{\dagger}(M_{\mathrm{tor}})\arrow[r]\arrow[d,equal]& j_*^{\dagger}(M)\arrow[d,hook]\arrow[r]&j_*^{\dagger}(M/M_{\mathrm{tor}})\arrow[d,hook]\arrow[r]&0\\
					0\arrow[r]& M_{\mathrm{tor}}\arrow[r]&M\arrow[r]&M/M_{\mathrm{tor}}\arrow[r]&0.
				\end{tikzcd}
			\end{equation*}
			The bottom complex of the diagram is exact, hence the exactness of the top complex at $j_*^{\dagger}(M_{\mathrm{tor}})$ and $j_*^{\dagger}(M)$ is clear. Since $j_*^{\dagger}(M/M_{\mathrm{tor}})$ is free of finite rank by (ii), we can take $e_1,\dots,e_r\in M/M_{\mathrm{tor}}$ such that
			\[j_*^{\dagger}(M/M_{\mathrm{tor}})=\oplus_{i=1} A_{\mathrm{mv},E}^{\dagger}e_i.\]
			Let $x_i\in M$ be a lift of $e_i\in M/M_{\mathrm{tor}}$ for $i=1,\dots,r$, then consider $N\coloneq \sum_{i=1}^{r}A_{\mathrm{mv},E}^{\dagger}x_i+M_{\mathrm{tor}}\subset M$ which is a finitely generated $A_{\mathrm{mv},E}^{\dagger}$-module and stable under $\varphi_q$. Thus $N\subseteq j_*^{\dagger}(M)$, which implies that the natural map $j_*^{\dagger}(M)\to j_*^{\dagger}(M/M_{\mathrm{tor}})$ is surjective.
			\item [(iv)]
			The claim that $j_*^{\dagger}(M)$ is of finite presentation follows from the short exact sequence of (iii), the fact that $M_{\mathrm{tor}}\cong\bigoplus_{i=1}^r\left(A_{\mathrm{mv},E}/\pi^{n_i}\right)^{\oplus d_i}$ for some $n_i,d_i$ (\cite[Proposition 4.10]{du2025multivariablevarphiqmathcaloktimesmodulesassociatedpadic}) and \cite[\href{https://stacks.math.columbia.edu/tag/0519}{Tag 07RB}]{stacks-project}. As for the proof of Lemma \ref{A^{dagger} is abelian}, the étaleness of $j_*^{\dagger}(M)$ follows from the étaleness of $j_*^{\dagger}(M_{\mathrm{tor}})=M_{\mathrm{tor}}$ and $j_*^{\dagger}(M/M_{\mathrm{tor}})$, the flatness of $\varphi_q: A_{\mathrm{mv},E}^{\dagger}\to A_{\mathrm{mv},E}^{\dagger}$ (Corollary \ref{phi_q is flat}) and five lemma.
			\item [(v)]
			The claim that $j^{\dagger}_{*}$ is right adjoint to $j^{\dagger,*}$ is easy to verify, hence we only prove that the natural map
			\begin{align}\label{natural map id to j_*j^*}
				M\to j_{*}^{\dagger}(j^{\dagger,*} (M))
			\end{align}
			is an isomorphism for $M\in \modet_{(\varphi_q,\mathcal{O}_K^{\times})}(A_{\mathrm{mv},E}^{\dagger})$. First we assume that $M$ is $\pi$-torsion free, hence $M$ is a free $A_{\mathrm{mv},E}^{\dagger}$-module (Corollary \ref{finite etale over A^{dagger} is free}). Then the map (\ref{natural map id to j_*j^*}) is injective.
			By Corollary \ref{j^dagger cap pM =p j^dagger}, $j_{*}^{\dagger}(j^{\dagger,*} (M)\cap \pi j^{\dagger,*} (M))=\pi j_{*}^{\dagger}(j^{\dagger,*} (M))$.
			Moreover, since $A_{\mathrm{mv},E}^{\dagger}\cap \pi A_{\mathrm{mv},E}=\pi A_{\mathrm{mv},E}^{\dagger}$, we have
			$M\cap \pi j^{\dagger,*} (M)={\pi} M$.
			Thus
			\[\dfrac{M}{{\pi}M}=\dfrac{M}{M\cap {\pi}j^{\dagger,*} (M)}\subseteq \dfrac{j_{*}^{\dagger}(j^{\dagger,*} (M))}{j_{*}^{\dagger}(j^{\dagger,*} (M)\cap {\pi}j^{\dagger,*} (M))}=\dfrac{j_{*}^{\dagger}(j^{\dagger,*} (M))}{{\pi}j_{*}^{\dagger}(j^{\dagger,*} (M))}\subseteq \dfrac{j^{\dagger,*} (M)}{{\pi}j^{\dagger,*} (M)}=\dfrac{M}{{\pi}M},\]
			where the last equality follows from the definition of $j^{\dagger,*}$. Therefore,
			\[\dfrac{M}{{\pi}M}=\dfrac{j_{*}^{\dagger}(j^{\dagger,*} (M))}{{\pi}j_{*}^{\dagger}(j^{\dagger,*} (M))},\]
			which implies
			\[j_{*}^{\dagger}(j^{\dagger,*} (M))=M+{\pi}j_{*}^{\dagger}(j^{\dagger,*} (M)).\]
			By Nakayama's lemma and Lemma \ref{pA dagger is topologically nilpotent}, we deduce $M=j_{*}^{\dagger}(j^{\dagger,*} (M))$. For general $M$, by Proposition \ref{structure of étale varphi_q-modules of finite type}, for $n\gg 1$, $M[{\pi}^n]\coloneq \left\{x\in M:{\pi}^nx=0\right\}$ is a finitely generated étale $(\varphi_q,\mathcal{O}_K^{\times})$-module over $A_{\mathrm{mv},E}^{\dagger}/{\pi}^n$ and $M/M[{\pi}^n]$ is an étale $(\varphi_q,\mathcal{O}_K^{\times})$-module over $A_{\mathrm{mv},E}^{\dagger}$ free of finite rank. Therefore, by Proposition \ref{inclusion for dagger is flat}, we have a short exact sequence of étale $(\varphi_q,\mathcal{O}_K^{\times})$-modules over $A_{\mathrm{mv},E}$:
			\[0\to j^{\dagger,*}(M[{\pi}^n])\to j^{\dagger,*}(M) \to j^{\dagger,*}(M/M[{\pi}^n])\to 0.\]
			Then we can apply (iii) to deduce that we have a short exact sequence of $A_{\mathrm{mv},E}^{\dagger}$-modules
			\[0\to j_{*}^{\dagger}(j^{\dagger,*} (M[{\pi}^n]))\to j_{*}^{\dagger}(j^{\dagger,*} (M))\to j_{*}^{\dagger}(j^{\dagger,*} (M/M[\pi^n]))\to 0.\]
			Moreover, $j^{\dagger,*}(M[{\pi}^n])=M[{\pi}^n]$, hence $j_{*}^{\dagger}(j^{\dagger,*} (M[{\pi}^n]))=M[{\pi}^n]$. Thus we obtain a commutative diagram with exact rows where the right vertical isomorphism follows from the free case:
			\begin{equation*}
				\begin{tikzcd}
					0\arrow[r]&M[{\pi}^n]\arrow[d,equal]\arrow[r]& M\arrow[d]\arrow[r]&M/M[{\pi}^n]\arrow[d,"\wr"]\arrow[r]&0\\
					0\arrow[r]& M[{\pi}^n]\arrow[r]& j_{*}^{\dagger}(j^{\dagger,*} (M))\arrow[r]& j_{*}^{\dagger}(j^{\dagger,*} (M/M[\pi^n]))\arrow[r]&0.
				\end{tikzcd}
			\end{equation*}
			Thus the natural map (\ref{natural map id to j_*j^*}) is an isomorphism. Then we deduce the fully faithfulness of $j^{\dagger,*}$ by \cite[\href{https://stacks.math.columbia.edu/tag/07RB}{Tag 07RB}]{stacks-project}.
			\item [(vi)]
			If $M$ is torsion, then $j^{\dagger,*}j_{*}^{\dagger}M=M$. If $M$ is free, we have
			\[\dfrac{j^{\dagger,*}j_{*}^{\dagger}M}{{\pi}j^{\dagger,*}j_{*}^{\dagger}M}=\dfrac{j_{*}^{\dagger}M}{{\pi}j_{*}^{\dagger}M}\subseteq \dfrac{M}{{\pi}M},\]
			thus if $x\in \ker \left(j^{\dagger,*}j_{*}^{\dagger}M\to M\right)$, then $x\in {\pi}j^{\dagger,*}j_{*}^{\dagger}M$. Say $x={\pi}x_1$ for some $x_1\in j^{\dagger,*}j_{*}^{\dagger}M$, since $M$ is free, we have $x_1\in \ker \left(j^{\dagger,*}j_{*}^{\dagger}M\to M\right)$, hence $x_1\in {\pi}j^{\dagger,*}j_{*}^{\dagger}M$, hence $x={\pi}x_1\in {\pi}^2j^{\dagger,*}j_{*}^{\dagger}M$. Repeating this argument, we deduce that $x\in \bigcap_{n=1}^{\infty}{\pi}^nj^{\dagger,*}j_{*}^{\dagger}M=0$, hence $\ker \left(j^{\dagger,*}j_{*}^{\dagger}M\to M\right)=0$.
			
			For general $M$, we have a commutative diagram with exact rows:
			\begin{equation*}
				\begin{tikzcd}
					0\arrow[r]& j^{\dagger,*}j_{*}^{\dagger}(M_{\mathrm{tor}})\arrow[r]\arrow[d,equal]& j^{\dagger,*}j_{*}^{\dagger}(M)\arrow[r]\arrow[d]& j^{\dagger,*}j_{*}^{\dagger}(M/M_{\mathrm{tor}})\arrow[r]\arrow[d,hook]& 0\\
					0\arrow[r]& M_{\mathrm{tor}}\arrow[r]& M\arrow[r]& M/M_{\mathrm{tor}}\arrow[r]& 0.
				\end{tikzcd}
			\end{equation*}
			Then the claim follows.
		\end{enumerate}
	\end{proof}
	
	\begin{defn}
		Let $M$ be an étale $(\varphi_q,\mathcal{O}_K^{\times})$-module over $A_{\mathrm{mv},E}$. We say that $M$ is \textit{overconvergent}, if $M$ lies in the essential image of $j^{\dagger,*}$.
	\end{defn}
	
	\begin{lem}\label{equivalent condition on overconvergence}
		Let $M$ be an étale $(\varphi_q,\mathcal{O}_K^{\times})$-module over $A_{\mathrm{mv},E}$, let $M_{\mathrm{tor}}$ be the submodule consisting of $\pi$-power torsion elements. The following are equivalent:
		\begin{enumerate}[label=(\roman*)]
			\item  $M$ is overconvergent.
			\item  $M/M_{\mathrm{tor}}$ is overconvergent.
			\item  The natural map $j^{\dagger,*}j_{*}^{\dagger}M\to M$ is surjective (hence is an isomorphism by Theorem \ref{part 6}).
			\item $\mathrm{rank}_{A_{\mathrm{mv}^{\dagger}}}j_*^{\dagger}\left(M/M_{\mathrm{tor}}\right)=\mathrm{rank}_{A_{\mathrm{mv},E}}M/M_{\mathrm{tor}}$.
			\item  $\mathrm{rank}_{A}M/\pi M=\mathrm{rank}_{A}j_*^{\dagger}M/{\pi}j_*^{\dagger}M$.
		\end{enumerate}
	\end{lem}
	\begin{proof}
		The equivalence among (i), (ii) and (iii) is easy to prove. Also, it is obvious that (i) implies (iv) and (v), hence it remains to show that (iv) implies (ii) and that (v) implies (iv). Assuming (iv), let $N\coloneq M/M_{\mathrm{tor}}$, then
		\[\mathrm{rank}_{A}N/{\pi}N=\mathrm{rank}_{A_{\mathrm{mv},E}}N =\mathrm{rank}_{A_{\mathrm{mv},E}^{\dagger}}j_*^{\dagger}N=\mathrm{rank}_{A}j_*^{\dagger}N/\pi j_*^{\dagger}N,\]
		where the first equality follows from \cite[Proposition 4.10]{du2025multivariablevarphiqmathcaloktimesmodulesassociatedpadic} and the last equality from Proposition \ref{structure of étale varphi_q-modules of finite type}. By Corollary \ref{j^dagger cap pM =p j^dagger} and Corollary \ref{submodule of etale module over A is etale}, we see that $j_*^{\dagger}N/{\pi}j_*^{\dagger}N=j_*^{\dagger}N/(j_*^{\dagger}N\cap\pi N)\subseteq N/{\pi}N$ are étale $(\varphi_q,\mathcal{O}_K^{\times})$-modules over $A$. By Lemma \ref{equal rank implies equality}, it follows that $j_*^{\dagger}N/{\pi}j_*^{\dagger}N= N/{\pi}N$. By the definition of $j^{\dagger,*}$ (\ref{defn of j^*}) and  Proposition \ref{structure of étale varphi_q-modules of finite type}, we have
		\[j^{\dagger,*}j_*^{\dagger}N/{\pi}j^{\dagger,*}j_*^{\dagger}N=j_*^{\dagger}N/{\pi}j_*^{\dagger}N=N/\pi N.\] 
		Thus by Theorem \ref{part 6}, we obtain $N=j^{\dagger,*}j_*^{\dagger}N+{\pi}N$, then we apply Nakayama's lemma to conclude.
		
		Suppose (v) holds. Since $M\cong M_{\mathrm{tor}}\oplus M/M_{\mathrm{tor}}$ as $A_{\mathrm{mv},E}$-modules (\cite[Proposition 4.10]{du2025multivariablevarphiqmathcaloktimesmodulesassociatedpadic}) and $j_*^{\dagger}M\cong M_{\mathrm{tor}}\oplus j_*^{\dagger}\left(M/M_{\mathrm{tor}}\right)$ as $A_{\mathrm{mv},E}^{\dagger}$-modules (Theorem \ref{general facts} (i), (ii), (iii) and Proposition \ref{structure of étale varphi_q-modules of finite type}), we deduce that $\mathrm{rank}_A\frac{M/M_{\mathrm{tor}}}{\pi(M/M_{\mathrm{tor}}) }=\mathrm{rank}_A \frac{j_*^{\dagger}\left(M/M_{\mathrm{tor}}\right)}{\pi j_*^{\dagger}\left(M/M_{\mathrm{tor}}\right)}$. Therefore, 
		\[\mathrm{rank}_{A_{\mathrm{mv},E}}M/M_{\mathrm{tor}}=\mathrm{rank}_A\frac{M/M_{\mathrm{tor}}}{\pi(M/M_{\mathrm{tor}}) }=\mathrm{rank}_A \frac{j_*^{\dagger}\left(M/M_{\mathrm{tor}}\right)}{\pi j_*^{\dagger}\left(M/M_{\mathrm{tor}}\right)}=\mathrm{rank}_{A_{\mathrm{mv},E}^{\dagger}}j_*^{\dagger}\left(M/M_{\mathrm{tor}}\right),\]
		where the first equality uses \cite[Proposition 4.10]{du2025multivariablevarphiqmathcaloktimesmodulesassociatedpadic} and the last equality uses Proposition \ref{structure of étale varphi_q-modules of finite type}. This finishes the proof.
	\end{proof}
	
	\begin{lem}\label{subquo of oc is oc}
		Let 
		\[0\to M_1\to M\to M_2\to 0\]
		be a short exact sequence of étale $(\varphi_q,\mathcal{O}_K^{\times})$-modules over $A_{\mathrm{mv},E}$. Assume that $M_2$ is $\pi$-torsion free.
		If $M$ is overconvergent, then so are $M_1$ and $M_2$.
	\end{lem}
	\begin{proof}
		Since $M_2$ is $\pi$-torsion free, by \cite[Proposition 4.10]{du2025multivariablevarphiqmathcaloktimesmodulesassociatedpadic}, we see that $M_{\mathrm{tor}}\subseteq M_1$ and $M_1/M_{\mathrm{tor}}$ is free, hence we obtain an exact sequence of free $A_{\mathrm{mv},E}$-modules:
		\begin{align}
			0\to M_1/M_{\mathrm{tor}}\to M/M_{\mathrm{tor}}\to M_2\to 0.
		\end{align}
		Hence
		\[\mathrm{rank}_{A_{\mathrm{mv},E}}M/M_{\mathrm{tor}}= \mathrm{rank}_{A_{\mathrm{mv},E}}M_1/M_{\mathrm{tor}}+\mathrm{rank}_{A_{\mathrm{mv},E}}M_2.\]
		By lemma \ref{equivalent condition on overconvergence}, $M/M_{\mathrm{tor}}$ is overconvergent and $\mathrm{rank}_{A_{\mathrm{mv}^{\dagger}}}j_*^{\dagger}\left(M/M_{\mathrm{tor}}\right)=\mathrm{rank}_{A_{\mathrm{mv},E}}M/M_{\mathrm{tor}}$. Since $j_*^{\dagger}$ is left exact by Theorem \ref{part 5}, it induces a short exact sequence
		\[0\to  j_*^{\dagger}\left(M_1/M_{\mathrm{tor}}\right)\to j_*^{\dagger}\left(M/M_{\mathrm{tor}}\right)\to j_*^{\dagger}M_2,\]
		thus
		\[\mathrm{rank}_{A_{\mathrm{mv}^{\dagger}}}j_*^{\dagger}\left(M/M_{\mathrm{tor}}\right)\leq \mathrm{rank}_{A_{\mathrm{mv}^{\dagger}}}j_*^{\dagger}\left(M_1/M_{\mathrm{tor}}\right)+\mathrm{rank}_{A_{\mathrm{mv}^{\dagger}}}j_*^{\dagger}M_2.\]
		On the other hand, 
		\[\mathrm{rank}_{A_{\mathrm{mv}^{\dagger}}}j_*^{\dagger}\left(M_1/M_{\mathrm{tor}}\right)\leq \mathrm{rank}_{A_{\mathrm{mv},E}}M_1/M_{\mathrm{tor}},\quad \mathrm{rank}_{A_{\mathrm{mv}^{\dagger}}}j_*^{\dagger}M_2\leq \mathrm{rank}_{A_{\mathrm{mv},E}}M_2.\]
		Lemma \ref{equivalent condition on overconvergence} implies $\mathrm{rank}_{A_{\mathrm{mv}^{\dagger}}}j_*^{\dagger}\left(M/M_{\mathrm{tor}}\right)=\mathrm{rank}_{A_{\mathrm{mv}}}M/M_{\mathrm{tor}}$, hence
		\[\mathrm{rank}_{A_{\mathrm{mv}^{\dagger}}}j_*^{\dagger}\left(M_1/M_{\mathrm{tor}}\right)= \mathrm{rank}_{A_{\mathrm{mv},E}}M_1/M_{\mathrm{tor}},\quad \mathrm{rank}_{A_{\mathrm{mv}^{\dagger}}}j_*^{\dagger}M_2= \mathrm{rank}_{A_{\mathrm{mv},E}}M_2.\]
		Thus $M_1/M_{\mathrm{tor}}$ and $M_2$ are overconvergent by Lemma \ref{equivalent condition on overconvergence}.
	\end{proof}
	\begin{remark}
		It is unknown whether the converse of Lemma~\ref{subquo of oc is oc} holds. That is, if $M$ is an extension of $M_2$ by $M_1$, where $M_1$ and $M_2$ are overconvergent, it is not clear whether $M$ itself is overconvergent.
	\end{remark}
	Recall that in \cite{du2025multivariablevarphiqmathcaloktimesmodulesassociatedpadic} there is a fully faithful functor $D_{A_{\mathrm{mv},E}}^{(i)}$ sending finite type continuous $\mathcal{O}_E$-representation of $G_K\coloneq \gal(\overline{K}/K)$ to the category of finite type étale $(\varphi_q,\mathcal{O}_K^{\times})$-modules over $A_{\mathrm{mv},E}$ for $0\leq i\leq f-1$. There is a natural question:
	
	
	\begin{question}\label{question about overconvergence}
		Is it true that $D_{A_{\mathrm{mv},E}}^{(0)}(\rho)$ is overconvergent for any finite free continuous $\mathcal{O}_E$-representation $\rho$ of $G_K$? 
	\end{question}
	
	For example, if $\rho$ is an unramified character sending the geometric Frobenius to $\lambda\in\mathcal{O}_E^{\times}$, then an easy computation gives that $D_{A_{\mathrm{mv},E}}^{(0)}(\rho)=A_{\mathrm{mv},E}e$, where $\varphi_q(e)=\lambda e$ and $a(e)=e$ for $a\in\mathcal{O}_K^{\times}$, hence is overconvergent.
	Note that by \cite[Lemma 7.1]{du2025multivariablevarphiqmathcaloktimesmodulesassociatedpadic}, the overconvergence of $D_{A_{\mathrm{mv},E}}^{(0)}(\rho)$ is equivalent to the overconvergence of $D_{A_{\mathrm{mv},E}}^{(i)}(\rho)$ for any $0\leq i\leq f-1$. 
	
	If we replace $\varphi_q$ everywhere by $\varphi$ in this section, all these results remain valid. In particular, we can also define overconvergent $(\varphi,\mathcal{O}_K^{\times})$-modules over $A_{\mathrm{mv},E}$. Let $D_{\mathrm{mv},E}^{\otimes}(\rho)\coloneq \bigotimes_{i=0}^{f-1}D_{A_{\mathrm{mv},E}}^{(0)}(\rho)$ which is an étale $(\varphi,\mathcal{O}_K^{\times})$-module by \cite[Proposition 7.7]{du2025multivariablevarphiqmathcaloktimesmodulesassociatedpadic}. 
	\begin{question}
		Is it true that $D_{\mathrm{mv},E}^{\otimes}(\rho)$ is overconvergent for any finite free continuous $\mathcal{O}_E$-representation $\rho$ of $G_K$?
	\end{question}
	The overconvergence of $D_{A_{\mathrm{mv},E}}^{(0)}(\rho)$ trivially implies the overconvergence of $D_{\mathrm{mv},E}^{\otimes}(\rho)$, and the reverse direction is not clear. We have the following example of overconvergent $(\varphi,\mathcal{O}_K^{\times})$-modules:
	
	\begin{lem}
		If $\rho: G_K\to\mathcal{O}_E^{\times}$ is a character, then $D_{\mathrm{mv},E}^{\otimes}(\rho)$ is overconvergent.
	\end{lem}
	\begin{proof}
		If $\rho$ is unramified, then the conclusion follows from the overconvergence of $D_{A_{\mathrm{mv},E}}^{(0)}(\rho)$. Therefore we can assume that $\rho $ factors as $\rho: G_K\twoheadrightarrow \mathcal{O}_K^{\times}\xrightarrow{\chi} \mathcal{O}_E^{\times}$, where $G_K\to \mathcal{O}_K^{\times}$ is the Lubin-Tate character associated to the uniformizer $p$. In this case, the Lubin-Tate $(\varphi_q,\mathcal{O}_K^{\times})$-module $D_{\mathrm{LT}}(\rho)$ is given by $A_{\mathrm{LT},E}\cdot e$, where $A_{\mathrm{LT},E}$ is the Lubin-Tate coefficient ring, $\varphi_q(e)=e$, $a(e)=\chi(a)e$ for $a\in\mathcal{O}_K^{\times}$. By \cite[Theorem 5.11]{du2025multivariablevarphiqmathcaloktimesmodulesassociatedpadic}, we have
		\begin{align}
			W_{\mathcal{O}_E}(A_{\infty})\otimes_{A_{\mathrm{mv},E}}D_{\mathrm{mv},E}^{\otimes}(\rho)\cong \left(\bigotimes_{i=0,W_{\mathcal{O}_E}(A_{\infty}')}^{f-1}\left(W_{\mathcal{O}_E}(A_{\infty}')\otimes_{\mathrm{pr}_i,A_{\mathrm{LT},E}}D_{\mathrm{LT}}(\rho)\right)\right)^{\Delta_1},
		\end{align}
		for the definition of $A_{\infty}$, $A_{\infty}'$ and $\mathrm{pr}_i$, see (\ref{A_{infty}}), (\ref{A_{infty}'}) and (\ref{pr_i}), and 
		\[\Delta_1=\{(a_0,\dots,a_{f-1})\in(\mathcal{O}_K^{\times})^{f}: a_0a_1\cdots a_{f-1}=1\}.\]
		Note that $\bigotimes_{i=0,W_{\mathcal{O}_E}(A_{\infty}')}^{f-1}\left(W_{\mathcal{O}_E}(A_{\infty}')\otimes_{\mathrm{pr}_i,A_{\mathrm{LT},E}}D_{\mathrm{LT}}(\rho)\right)=W_{\mathcal{O}_E}(A_{\infty}') e_0\otimes\cdots\otimes e_{f-1}$, and for $a=(a_0,\dots,a_{f-1})\in\Delta_1$, we have
		\begin{align*}
			a(e_0\otimes\cdots\otimes e_{f-1})&=a_0(e_0)\otimes \cdots\otimes a_{f-1}(e_{f-1})=\chi(a_0\cdot\cdots\cdot a_{f-1})e_0\otimes \cdots \otimes e_{f-1}\\
			&=e_0\otimes \cdots \otimes e_{f-1},
		\end{align*}
		where the last equality uses the definition of $\Delta_1$. Hence, 
		\begin{align}
			W_{\mathcal{O}_E}(A_{\infty})\otimes_{A_{\mathrm{mv},E}}D_{\mathrm{mv},E}^{\otimes}(\rho)\cong W_{\mathcal{O}_E}(A_{\infty}) e_0\otimes \cdots \otimes e_{f-1},
		\end{align}
		which implies that $D_{\mathrm{mv},E}^{\otimes}(\rho)\cong A_{\mathrm{mv},E}e$ with $\varphi(e)=e$, $a(e)=\chi(a)e$ for $a\in\mathcal{O}_K^{\times}$. Thus $D_{\mathrm{mv},E}^{\otimes}(\rho)$ is overconvergent.
	\end{proof}

	\section{Overconvergence at the perfectoid level}\label{section 3}
	
	We do not have a complete solution to Question \ref{question about overconvergence} currently. At present, we can only prove the overconvergence at the perfectoid level. In this section, we give a detailed proof of this perfectoid overconvergence.
	\subsection{A reminder on the relative Fargues-Fontaine curve}
	We recall some constructions in the theory of relative Fargues-Fontaine curves and prove that any finite projective module over the integral Robba ring $\widetilde{\mathcal{R}}_{A_{\infty}}^{\mathrm{int}}$ is free.
	
	Throughout this section, we fix the following notation. Let $(R_0,R_0^+)$ be a perfectoid Huber pair over $\mathbb{F}$ with a pseudo-uniformizer $\varpi$ (\cite[Definition 6.1.1]{WS2020berkeley}), and let $S\coloneq\mathrm{Spa}(R_0,R_0^{+})$ be the corresponding affinoid perfectoid space. Let $\mathbf{Perf}_{S}$ be the site of perfectoid spaces over $S$ equipped with the $v$-topology (\cite[Definition 17.1.1]{WS2020berkeley}). If $R$ is a Huber ring, we simply write $\mathrm{Spa}R$ for $\mathrm{Spa}(R,R^{\circ})$, where $R^{\circ}$ is the subring of $R$ consisting of power-bounded elements.
	We say that $X$ is an affinoid perfectoid space, if $X=\mathrm{Spa}(R,R^{+})$ with $R$ perfectoid (\cite[Remark 7.1.3]{WS2020berkeley}).
	
	For any affinoid perfectoid space $X=\mathrm{Spa}(R,R^{+})$ over $S$, we also view $\varpi$ as a pseudo-uniformizer of $R$. We endow $W_{\mathcal{O}_E}(R^{+})$ with the $(\pi,[\varpi])$-adic topology, and consider the analytic adic space $\mathcal{Y}_{X,\closeopen{0}{\infty}}\coloneq \mathrm{Spa}W_{\mathcal{O}_E}(R^{+})\setminus V([\varpi])$. It is a quasi-Stein space, with a covering by affinoid opens given by $\left\{\mathcal{Y}_{X,[0,r]}: r\in\mathbb{Q}_{>0}\right\}$, where
	\begin{equation}
		\begin{aligned}
			\mathcal{Y}_{X,[0,r]}&\coloneq\{x\in \mathcal{Y}_{X,\closeopen{0}{\infty}}: |\pi|_x\leq |[\varpi]^{1/r}|_x\neq 0\}\\
			&=\mathrm{Spa}\left(\mathcal{O}(\mathcal{Y}_{X,[0,r]}),\mathcal{O}^{+}(\mathcal{Y}_{X,[0,r]})\right).
		\end{aligned}
	\end{equation}
	If $0<1/r\in\mathbb{Z}[1/p]$, we have
	\begin{align}\label{defn of O(X_[0,r])}
		\mathcal{O}(\mathcal{Y}_{X,[0,r]})= W_{\mathcal{O}_E}(R^+)\left\langle\frac{\pi}{[\varpi]^{1/r}}\right\rangle\left[\dfrac{1}{[\varpi]}\right]\subset W_{\mathcal{O}_E}(R),
	\end{align}
	where $W_{\mathcal{O}_E}(R^+)\left\langle\frac{\pi}{[\varpi]^{1/r}}\right\rangle$ is the $\pi$-adic completion of $W_{\mathcal{O}_E}(R^+)\left[\frac{\pi}{[\varpi]^{1/r}}\right]$, and $\mathcal{O}^{+}(\mathcal{Y}_{X,[0,r]})$ is the integral closure of $W_{\mathcal{O}_E}(R^+)\left\langle\frac{\pi}{[\varpi]^{1/r}}\right\rangle$ in $\mathcal{O}(\mathcal{Y}_{X,[0,r]})$. Equivalently, $W_{\mathcal{O}_E}(R^+)\left\langle\frac{\pi}{[\varpi]^{1/r}}\right\rangle$ is also the $[\varpi]$-adic completion of $W_{\mathcal{O}_E}(R^+)\left[\frac{\pi}{[\varpi]^{1/r}}\right]$. Assume that the topology on the Banach algebra $R$ is defined by a norm
	$|\cdot|$ such that $|\varpi| = p^{-1}$.
	Such a norm always exists (although it is not necessarily multiplicative;
	see for example the discussion before
	\cite[Remark 2.2.7]{WS2020berkeley}).
	For the definition of Banach rings, see for example
	\cite[§0.1]{fontaine2012perfectoides}. Then
	\begin{align}
		\mathcal{O}(\mathcal{Y}_{X,[0,r]})=\left\{\sum_{n\geq 0}\pi^n[x_n]\in W_{\mathcal{O}_E}(R): x_n\in R, \lim\limits_{n\to\infty}|x_n|p^{-n/r}=0\right\}.
	\end{align}
	We endow $\mathcal{O}(\mathcal{Y}_{X,[0,r]})$ with the $[\varpi]$-adic topology, then
	\begin{equation}\label{description of [0,r]}
		\begin{aligned}
			\left(\mathcal{O}(\mathcal{Y}_{X,[0,r]})\right)^{\circ}&=\left\{\sum_{n\geq 0}\pi^n[x_n]\in \mathcal{O}(\mathcal{Y}_{X,[0,r]}):|x_n|p^{-n/r}\leq 1 \right\},\\
			\left(\mathcal{O}(\mathcal{Y}_{X,[0,r]})\right)^{\circ\circ}&=\left\{\sum_{n\geq 0}\pi^n[x_n]\in \mathcal{O}(\mathcal{Y}_{X,[0,r]}):|x_n|p^{-n/r}< 1 \right\}.
		\end{aligned}
	\end{equation}
	Define a norm $|\cdot|_r$ on $\mathcal{O}(\mathcal{Y}_{X,[0,r]})$ by 
	\begin{align}\label{defn of ||_r}
		\left\lvert\sum_{n\geq 0}\pi^n[x_n]\right\rvert_r\coloneq \sup_n |x_n|p^{-n/r}.
	\end{align}
	We can check that the $[\varpi]$-adic topology of $\mathcal{O}(\mathcal{Y}_{X,[0,r]})$ agrees with the topology induced by $|\cdot|_r$, and $\mathcal{O}(\mathcal{Y}_{X,[0,r]})$ is a Banach algebra over $\mathcal{O}_E$. If the norm $|\cdot|$ of $R$ is multiplicative, then $|\cdot|_r$ is also multiplicative (\cite[Proposition 1.4.9]{fargues2018courbes}).

	For a closed interval $I=[s,r]\subseteq (0,\infty)$ with $r,s\in\mathbb{Q}_{>0}$, we can also consider the affinoid open $\mathcal{Y}_{X,I}$ contained in $\mathcal{Y}_{X,\closeopen{0}{\infty}}$:
	\begin{equation}
		\begin{aligned}
			\mathcal{Y}_{X,I}&\coloneq\{x\in \mathcal{Y}_{X,\closeopen{0}{\infty}}: |[\varpi]^{1/s}|_x\leq|\pi|_x\leq |[\varpi]^{1/r}|_x\neq 0\}\\
			&=\mathrm{Spa}\left(\mathcal{O}(\mathcal{Y}_{X,I}),\mathcal{O}^{+}(\mathcal{Y}_{X,I})\right).
		\end{aligned}
	\end{equation}
	If $1/r,1/s\in \mathbb{Z}[1/p]$, similar to (\ref{defn of O(X_[0,r])}), we have
	\begin{align}\label{defn of O(X_[s,r])}
		\mathcal{O}(\mathcal{Y}_{X,I})\coloneq W_{\mathcal{O}_E}(R^+)\left\langle\frac{\pi}{[\varpi]^{1/r}},\frac{[\varpi]^{1/s}}{\pi}\right\rangle\left[\dfrac{1}{[\varpi]}\right]
	\end{align}
	where $W_{\mathcal{O}_E}(R^+)\left\langle\frac{\pi}{[\varpi]^{1/r}},\frac{[\varpi]^{1/s}}{\pi}\right\rangle$ is the $\pi$-adic completion of $W_{\mathcal{O}_E}(R^+)\left[\frac{\pi}{[\varpi]^{1/r}},\frac{[\varpi]^{1/s}}{\pi}\right]$, and $\mathcal{O}^{+}(\mathcal{Y}_{X,I})$ is the integral closure of $W_{\mathcal{O}_E}(R^+)\left\langle\frac{\pi}{[\varpi]^{1/r}},\frac{[\varpi]^{1/s}}{\pi}\right\rangle$ in $\mathcal{O}(\mathcal{Y}_{X,I})$. We endow $\mathcal{O}(\mathcal{Y}_{X,I})$ with the $[\varpi]$-adic topology, which agrees with the $\pi$-adic topology. By \cite[Exemple 1.6.3]{fargues2018courbes}, $\mathcal{O}(\mathcal{Y}_{X,I})$ is a Banach algebra over $E$. 
	There are also other open subspaces of $\mathcal{Y}_{X,\closeopen{0}{\infty}}$ which play an important role in the theory of the relative Fargues-Fontaine curve:
	\begin{equation}
		\begin{aligned}
			\mathcal{Y}_{X,(0,\infty)}\coloneq \mathcal{Y}_{X,\closeopen{0}{\infty}}\setminus V(\pi)=\bigcup_{s,r\in\mathbb{Q}_{>0}}\mathcal{Y}_{X,[s,r]}
		\end{aligned}
	\end{equation}
	and
	\begin{equation}
		\begin{aligned}
			\mathcal{Y}_{X,\openclose{0}{r}}\coloneq \mathcal{Y}_{X,[0,r]}\setminus V(\pi)=\bigcup_{s\in\mathbb{Q}_{>0}}\mathcal{Y}_{X,[s,r]},
		\end{aligned}
	\end{equation}
	Both $\mathcal{Y}_{X,(0,\infty)}$ and $\mathcal{Y}_{X,\openclose{0}{r}}$ are quasi-Stein spaces.

	In the rest of this section, when $I\subseteq\closeopen{0}{\infty}$ is an interval, we always assume that $I$ is of the form $[0,r]$, $[s,r]$, $\closeopen{0}{\infty}$, $(0,\infty)$ or $\openclose{0}{r}$, where $r,s\in\mathbb{Q}_{>0}$. 
	For a general perfectoid space $X/S$ and an interval $I\subseteq\closeopen{0}{\infty}$ as above, we can glue the $\mathcal{Y}_{U,I}$ along affinoid opens $U$ of $X$ (\cite[Proposition II.1.3]{fargues2024geometrizationlocallanglandscorrespondence}) and denote by $\mathcal{Y}_{X,I}$ the resulting analytic adic space. One can easily check that
	\begin{align}\label{Y_{X,I} as fibre prod}
		\mathcal{Y}_{X,I}\cong \mathcal{Y}_{X,\closeopen{0}{\infty}}\times_{\mathcal{Y}_{S,\closeopen{0}{\infty}}}\mathcal{Y}_{S,I}.
	\end{align}
	
	\begin{remark}
		\begin{enumerate}
			\item [(i)]
			In \cite[Chapter II]{fargues2024geometrizationlocallanglandscorrespondence}, $\mathcal{Y}_{X,\closeopen{0}{\infty}}$ is denoted by $\mathcal{Y}_{X}$ and $\mathcal{Y}_{X,(0,\infty)}$ is denoted by $Y_{X}$. In \cite{WS2020berkeley}, $\mathcal{Y}_{X,\closeopen{0}{\infty}}$ is also denoted by $``X\times \mathrm{Spa}\mathbb{Z}_p"$.
			\item [(ii)]
			For a general perfectoid space $X/S$ and an interval $I$ as above, $\mathcal{Y}_{X,I}$ is a sousperfectoid space (\cite[$\S$6.3]{WS2020berkeley}). In fact, suppose that $I$ is compact, let $E(\pi^{1/p^{\infty}})$ be the $p$-adic completion of $\bigcup_{n\geq1}E(\pi^{1/p^{n}})$ which is a perfectoid field with ring of integers $\mathcal{O}_{E(\pi^{1/p^{\infty}})}$, then it is straightforward to verify that
			\[\mathcal{O}(\mathcal{Y}_{X,I})\widehat{\otimes}_{\mathcal{O}_E}\mathcal{O}_{E(\pi^{1/p^{\infty}})}\coloneq \left(\mathcal{O}(\mathcal{Y}_{X,I})^{\circ}\widehat{\otimes}_{\mathcal{O}_E}\mathcal{O}_{E(\pi^{1/p^{\infty}})}\right)^{\wedge}\left[\frac{1}{[\varpi]}\right],\] where $(\cdot)^{\wedge}$ denotes the $[\varpi]$-adic completion, is a perfectoid algebra. Now we fix a continuous section of $\mathcal{O}_E\hookrightarrow\mathcal{O}_{E(\pi^{1/p^{\infty}})}$ of $\mathcal{O}_E$-modules, it induces a continuous section $\mathcal{O}(\mathcal{Y}_{X,I})\widehat{\otimes}_{\mathcal{O}_E}\mathcal{O}_{E(\pi^{1/p^{\infty}})}\twoheadrightarrow \mathcal{O}(\mathcal{Y}_{X,I})$ of $\mathcal{O}(\mathcal{Y}_{X,I})$-modules, hence $\mathcal{Y}_{X,I}$ is affinoid sousperfectoid. 
			\item [(iii)]
			The analytic adic space $\mathcal{Y}_{X,I}$ does not depend on the choice of the uniformizer $\pi\in \mathcal{O}_E$, and does not depend on the choice of the pseudo-uniformizer $\varpi$ if $(0,\infty)\subseteq I$. If $I=[0,r]$ or $I=[s,r]$, then $\mathcal{Y}_{X,I}$ depends on the choice of $\varpi$.
		\end{enumerate}
	\end{remark}
	
	Let $X_1\to X \leftarrow X_2$ be a diagram in $\mathbf{Perf}_{S}$. There is a natural isomorphism of sousperfectoid spaces:
	\begin{align}\label{compatible with fibre prod}
		\mathcal{Y}_{X_1\times_X X_2,I}\xrightarrow{\sim} \mathcal{Y}_{X_1,I}\times_{\mathcal{Y}_{X,I}}\mathcal{Y}_{X_2,I}.
	\end{align}
	To see this, we may assume that $X_1,X,X_2$ are affinoid perfectoid spaces. Since $\mathcal{O}\left(\mathcal{Y}_{X,I}\right)\hookrightarrow \mathcal{O}\left(\mathcal{Y}_{X,I}\right)\widehat{\otimes}_{\mathcal{O}_E}\mathcal{O}_{E(\pi^{1/p^{\infty}})}$ naturally splits as topological $\mathcal{O}\left(\mathcal{Y}_{X,I}\right)$-modules, let
	\[\mathcal{Y}_{X,I}^{\mathrm{perf}}\coloneq\mathcal{Y}_{X,I}\times_{\mathrm{Spa}\mathcal{O}_E}\mathrm{Spa}\mathcal{O}_{E(\pi^{1/p^{\infty}})},\] 
	it suffices to check that the natural map
	\[\mathcal{Y}_{X_1\times_X X_2,I}^{\mathrm{perf}}\longrightarrow \mathcal{Y}_{X_1,I}^{\mathrm{perf}}\times_{\mathcal{Y}_{X,I}^{\mathrm{perf}}}\mathcal{Y}_{X_2,I}^{\mathrm{perf}}\]
	is an isomorphism. Then for $I=\closeopen{0}{\infty}$, we can pass to the corresponding diamonds and apply \cite[Proposition II.1.2]{fargues2024geometrizationlocallanglandscorrespondence} to conclude. The general case can be easily deduced from this case and (\ref{Y_{X,I} as fibre prod}). We can also use the tilting functor to prove this isomorphism (\cite[Definition 6.2.1]{WS2020berkeley}). Since
	\[\mathcal{Y}_{X,I}=\bigcup_{J\subseteq I\ \text{is compact}} \mathcal{Y}_{X,J},\]
	we can reduce to the case where $I$ is compact. Then a straightforward computation shows that there is an isomorphism
	\begin{align}\label{isom of tilt for [0,r]}
		\left(\mathcal{Y}_{X,[0,r]}^{\mathrm{perf}}\right)^{\flat}&\cong X\times_{S} \mathrm{Spa} \left(R_0\left\langle T^{1/p^{\infty}}\right\rangle,R_1^{+}\right),
	\end{align}
	and
	\begin{align}\label{isom of tilt for [s,r]}
		\left(\mathcal{Y}_{X,[s,r]}^{\mathrm{perf}}\right)^{\flat}&\cong X\times_S \mathrm{Spa}\left(R_0\left\langle T^{1/p^{\infty}},\left(\varpi^{\frac{1}{s}-\frac{1}{r}}T^{-1}\right)^{1/p^{\infty}}\right\rangle,R_2^{+}\right),
	\end{align}
	where $T=\left(\frac{\pi}{[\varpi]^{1/r}},\frac{\pi^{1/p}}{[\varpi]^{1/pr}}\dots\right)\in \left(\mathcal{O}(\mathcal{Y}^{\mathrm{perf}}_{X,I})\right)^{\flat}$, $R_0\langle\cdot\rangle$ is the $\varpi$-adic completion, $R_1^{+}$ is the integral closure of $R_0^{+}\left\langle T^{1/p^{\infty}}\right\rangle$ in $R_0\left\langle T^{1/p^{\infty}}\right\rangle$, and $R_2^{+}$ is the integral closure of $R_0^{+}\left\langle T^{1/p^{\infty}},\left(\varpi^{\frac{1}{s}-\frac{1}{r}}T^{-1}\right)^{1/p^{\infty}}\right\rangle$ in $R_0\left\langle T^{1/p^{\infty}},\left(\varpi^{\frac{1}{s}-\frac{1}{r}}T^{-1}\right)^{1/p^{\infty}}\right\rangle$. See for example \cite[Lemma 4.1.12, Proposition 4.1.13]{kedlaya2019relative}. Then the claim follows.
	
	\begin{prop}\label{v-sheaf and v-stack}
		Let $I\subseteq\closeopen{0}{\infty}$ be an interval as above.
		\begin{enumerate}
			\item [(i)] 
			The functor $\mathcal{O}_I: X\mapsto \mathcal{O}\left(\mathcal{Y}_{X,I}\right)$ is a $v$-sheaf on $\mathbf{Perf}_{S}$, and $H_v^k(X,\mathcal{O}_I)=0$ for $k\geq 1$ if $X$ is affinoid.
			\item [(ii)] The functor sending $X$ to the groupoid of vector bundles over $\mathcal{Y}_{X,I}$ is a $v$-stack on $\mathbf{Perf}_{S}$.
		\end{enumerate}
	\end{prop}
	\begin{proof}
		For (i), see \cite[Proposition II.2.1]{fargues2024geometrizationlocallanglandscorrespondence}, and (ii) is a special case of \cite[Proposition 19.5.3]{WS2020berkeley}.
	\end{proof}
	
	
	
	If $X=\mathrm{Spa}(R,R^{+})$ is a perfectoid space over $S$ where $R$ is perfectoid and $I\subseteq\closeopen{0}{\infty}$ is a compact interval with rational end points, recall that the topology of $\mathcal{O}(\mathcal{Y}_{X,I})$ is the $[\varpi]$-adic topology. If $I$ is an interval as before, we endow $\mathcal{O}(\mathcal{Y}_{X,I})\cong \varprojlim\limits_{J\subseteq I,\ J\ \text{is compact}}\mathcal{O}(\mathcal{Y}_{X,J})$ with the inverse limit topology.
	
	\begin{cor}\label{G-torsor}
		Let $G$ be a profinite group, and suppose that $X'=\mathrm{Spa}(R',R'^{+})\to X=\mathrm{Spa}(R,R^{+})$ is a pro-étale $G$-torsor in $\mathbf{Perf}_S$ where $(R,R^{+}),(R',R'^{+})$ are Huber pairs with $R,R'$ perfectoid. If $I\subseteq\closeopen{0}{\infty}$ is an interval as above, then we have
		\begin{align}
			H^k_{\mathrm{cont}}\left(G,\mathcal{O}\left(\mathcal{Y}_{X',I}\right)\right)=\begin{cases}
				\mathcal{O}\left(\mathcal{Y}_{X,I}\right),\quad &k=0,\\
				0,\quad &k\geq 1,
			\end{cases}
		\end{align}
		where $H^k_{\mathrm{cont}}\left(G,\mathcal{O}\left(\mathcal{Y}_{X',I}\right)\right)$ is the $k$-th continuous cohomology group.
	\end{cor}
	\begin{proof}
		Let $C^{\bullet}$ be the complex 
		\[0\to \mathcal{O}\left(\mathcal{Y}_{X,I}\right)\xrightarrow{d^{-1}} \mathcal{O}\left(\mathcal{Y}_{X',I}\right)\xrightarrow{d^0} C(G,\mathcal{O}\left(\mathcal{Y}_{X',I}\right))\to\cdots,\]
		where $C^{k}=C(G^{k},\mathcal{O}_{X',I})$ is the ring of continuous maps from $G^k$ to $\mathcal{O}\left(\mathcal{Y}_{X',I}\right)$ for $k\geq 0$, $C^{-1}\coloneq \mathcal{O}(\mathcal{Y}_{X,I}) $, and $d^k$ is the usual differential which computes the continuous cohomology groups for $k\geq 0$.
		We need to show the exactness of $C^{\bullet}$. Since 
		\[\mathcal{O}(\mathcal{Y}_{X',I})\cong \varprojlim_{J\subseteq I,\ J\ \text{is compact}}\mathcal{O}(\mathcal{Y}_{X',J}),\]
		and for any $J_1\subseteq J_2$, the restriction map $\mathcal{O}(\mathcal{Y}_{X',J_2})\to \mathcal{O}(\mathcal{Y}_{X',J_1})$ has dense image, by \cite[Chap. 0, Remarques (13.2.4)]{EGA3P1}, we can reduce to the case $I=[s,r]$ or $I=[0,r]$. 
		
		We still denote by $[\varpi]$ the constant function sending $g\in G^k$ to $[\varpi]\in\mathcal{O}(\mathcal{Y}_{X',I})$. In the rest of the proof, we write $-\widehat{\otimes}_{\mathcal{O}_E}\mathcal{O}_{E(\pi^{1/p^{\infty}})}$ for the $[\varpi]$-adic completion of $-\otimes_{\mathcal{O}_E}\mathcal{O}_{E(\pi^{1/p^{\infty}})}$. 
		As $C^k\hookrightarrow C^k\widehat{\otimes}_{\mathcal{O}_E}\mathcal{O}_{E(\pi^{1/p^{\infty}})}$ naturally splits as topological $C^k$-modules, it suffices to show the exactness of $C^{\bullet}\widehat{\otimes}_{\mathcal{O}_E}\mathcal{O}_{E(\pi^{1/p^{\infty}})}$. Moreover, we can check that the natural map
		\begin{align}\label{function tensor agree with tensor function}
			C(G^k,\mathcal{O}(\mathcal{Y}_{X',I}))\widehat{\otimes}_{\mathcal{O}_E}\mathcal{O}_{E(\pi^{1/p^{\infty}})}\xrightarrow{\sim} C\left(G^k,\mathcal{O}(\mathcal{Y}_{X',I})\widehat{\otimes}_{\mathcal{O}_E}\mathcal{O}_{E(\pi^{1/p^{\infty}})}\right)
		\end{align}
		sending $f\otimes x$ to the map $g\mapsto f(g)\otimes x$ is an isomorphism for $k\geq 0$. To see this, we first note that $C(G^k,\mathcal{O}(\mathcal{Y}_{X',I}))^{\circ}=C(G^k,\mathcal{O}(\mathcal{Y}_{X',I})^{\circ})$, and $C\left(G^k,\mathcal{O}(\mathcal{Y}_{X',I})\widehat{\otimes}_{\mathcal{O}_E}\mathcal{O}_{E(\pi^{1/p^{\infty}})}\right)=C\left(G^k,\mathcal{O}(\mathcal{Y}_{X',I})^{\circ}\widehat{\otimes}_{\mathcal{O}_E}\mathcal{O}_{E(\pi^{1/p^{\infty}})}\right)\left[\frac{1}{[\varpi]}\right]$.
		Thus it suffices to check that
		\begin{align}\label{function tensor agree with tensor function, bounded ver}
			C(G^k,\mathcal{O}(\mathcal{Y}_{X',I})^{\circ})\widehat{\otimes}_{\mathcal{O}_E}\mathcal{O}_{E(\pi^{1/p^{\infty}})}\to C\left(G^k,\mathcal{O}(\mathcal{Y}_{X',I})^{\circ}\widehat{\otimes}_{\mathcal{O}_E}\mathcal{O}_{E(\pi^{1/p^{\infty}})}\right)
		\end{align}
		is an isomorphism. 
		Let us write $G^k=\varprojlim_{\alpha}H_\alpha$ as a inverse limit of finite quotient groups of $G^k$, then we can check that the left hand side of (\ref{function tensor agree with tensor function, bounded ver}) is the $[\varpi]$-adic completion of $\left(\varinjlim_{\alpha}C(H_\alpha,\mathcal{O}(\mathcal{Y}_{X',I})^{\circ})\right){\otimes}_{\mathcal{O}_E}\mathcal{O}_{E(\pi^{1/p^\infty})}$, and the right hand side of (\ref{function tensor agree with tensor function, bounded ver}) is the $[\varpi]$-adic completion of $\varinjlim_{\alpha}C\left(H_\alpha,\mathcal{O}(\mathcal{Y}_{X',I})^{\circ}{\otimes}_{\mathcal{O}_E}\mathcal{O}_{E(\pi^{1/p^\infty})}\right)$.
		Moreover, we have isomorphisms of topological rings
		\begin{align}\label{isom on dense subrings}
			\begin{aligned}
				\left(\varinjlim_{\alpha}C(H_\alpha,\mathcal{O}(\mathcal{Y}_{X',I})^{\circ})\right){\otimes}_{\mathcal{O}_E}\mathcal{O}_{E(\pi^{1/p^\infty})}&\cong \varinjlim_{\alpha}\left(C(H_\alpha,\mathcal{O}(\mathcal{Y}_{X',I})^{\circ}){\otimes}_{\mathcal{O}_E}\mathcal{O}_{E(\pi^{1/p^\infty})}\right)\\
				&\cong \varinjlim_{\alpha}C\left(H_\alpha,\mathcal{O}(\mathcal{Y}_{X',I})^{\circ}{\otimes}_{\mathcal{O}_E}\mathcal{O}_{E(\pi^{1/p^\infty})}\right),
			\end{aligned}
		\end{align}
		where every ring is equipped with the $[\varpi]$-adic topology. Then (\ref{isom on dense subrings}) implies (\ref{function tensor agree with tensor function, bounded ver}) by taking $[\varpi]$-adic completion.

		Let $\mathcal{Y}_{X,I}^{\mathrm{perf}}\coloneq\mathcal{Y}_{X,I}\times_{\mathrm{Spa}\mathcal{O}_E}\mathrm{Spa}\mathcal{O}_{E(\pi^{1/p^{\infty}})}$.
		Since $X'\to X$ is a pro-étale $G$-torsor and $X,X'$ are affinoid, there exists a  pro-étale covering $U\to X$ of affinoid perfectoid spaces such that $U'\coloneq U\times_{X}X'\cong U\times G$, where $U\times G$ is the representable pro-étale sheaf defined in \cite[$\S$9.3]{WS2020berkeley}. Using (\ref{compatible with fibre prod}), (\ref{isom of tilt for [0,r]}) and (\ref{isom of tilt for [s,r]}), we deduce that
		\begin{align}\label{tilt, torsor}
			\left(\mathcal{Y}_{U,I}^{\mathrm{perf}}\times_{\mathcal{Y}_{X,I}^{\mathrm{perf}}}\mathcal{Y}_{X',I}^{\mathrm{perf}}\right)^{\flat}\cong\left(\mathcal{Y}_{U',I}^{\mathrm{perf}}\right)^{\flat}\cong \left(\mathcal{Y}_{U,I}^{\mathrm{perf}}\right)^{\flat}\times G.
		\end{align}
		For any perfectoid space $T$ over $\left(\mathcal{Y}_{U,I}^{\mathrm{perf}}\right)^{\flat}$, by \cite[Theorem 7.1.4]{WS2020berkeley}, there exist a unique (up to isomorphism) untilt $T^{\sharp}$ over $\mathcal{Y}_{U,I}^{\mathrm{perf}}$. In particular, $\mathcal{Y}_{U,I}^{\mathrm{perf}}\times G$ is the untilt of $\left(\mathcal{Y}_{U,I}^{\mathrm{perf}}\times G\right)^{\flat}$, we then have
		\begin{align*}
			\mathrm{Hom}_{\left(\mathcal{Y}_{U,I}^{\mathrm{perf}}\right)^{\flat}}\left(T,\left(\mathcal{Y}_{U,I}^{\mathrm{perf}}\times G\right)^{\flat}\right)&=\mathrm{Hom}_{\mathcal{Y}_{U,I}^{\mathrm{perf}}}\left(T^{\sharp},\mathcal{Y}_{U,I}^{\mathrm{perf}}\times G\right)\\
			&=\mathrm{Hom}_{\mathcal{Y}_{U,I}^{\mathrm{perf}}}\left(T^{\sharp},\mathcal{Y}_{U,I}^{\mathrm{perf}}\right)\times \mathrm{Hom}_{\mathrm{cont}}(|T^{\sharp}|,G)\\
			&=\mathrm{Hom}_{\left(\mathcal{Y}_{U,I}^{\mathrm{perf}}\right)^{\flat}}\left(T,\left(\mathcal{Y}_{U,I}^{\mathrm{perf}}\right)^{\flat}\right)\times \mathrm{Hom}_{\mathrm{cont}}(|T|,G)\\
			&=\mathrm{Hom}_{\left(\mathcal{Y}_{U,I}^{\mathrm{perf}}\right)^{\flat}}\left(T,\left(\mathcal{Y}_{U,I}^{\mathrm{perf}}\right)^{\flat}\times G\right).
		\end{align*}
		Here the first isomorphism uses \cite[Theorem 7.1.4]{WS2020berkeley}, the second and the last isomorphism uses \cite[$\S$9.3]{WS2020berkeley}, and the third isomorphism uses \cite[Theorem 6.2.6]{WS2020berkeley}. Then we deduce that $ \left(\mathcal{Y}_{U,I}^{\mathrm{perf}}\right)^{\flat}\times G\cong \left(\mathcal{Y}_{U,I}^{\mathrm{perf}}\times G\right)^{\flat}$.
		Therefore, by \cite[Theorem 7.1.4]{WS2020berkeley} and (\ref{tilt, torsor}),
		\[\mathcal{Y}_{U,I}^{\mathrm{perf}}\times_{\mathcal{Y}_{X,I}^{\mathrm{perf}}}\mathcal{Y}_{X',I}^{\mathrm{perf}}\cong \mathcal{Y}_{U,I}^{\mathrm{perf}}\times G.\]
		Since $U\to X$ is a pro-étale covering, (\ref{isom of tilt for [0,r]}) and (\ref{isom of tilt for [s,r]}) imply that $\left(\mathcal{Y}_{U,I}^{\mathrm{perf}}\right)^{\flat}\to \left(\mathcal{Y}_{X,I}^{\mathrm{perf}}\right)^{\flat}$ is a pro-étale covering (\cite[Proposition 8.2.5.(3)]{WS2020berkeley}), hence $\mathcal{Y}_{U,I}^{\mathrm{perf}}\to \mathcal{Y}_{X,I}^{\mathrm{perf}}$ is also a pro-étale covering by \cite[Corollary 7.5.3]{WS2020berkeley}. Hence $\mathcal{Y}_{X',I}^{\mathrm{perf}}\to \mathcal{Y}_{X,I}^{\mathrm{perf}}$ is a pro-étale $G$-torsor. Let $X'^{\times n}\coloneq X'\times_{X}\cdots\times_XX'$ be the fibre product of $n$ copies of $X'$ over $X$ and let $\left(\mathcal{Y}_{X',I}^{\mathrm{perf}}\right)^{\times n}\coloneq \mathcal{Y}_{X',I}^{\mathrm{perf}}\times_{\mathcal{Y}_{X,I}^{\mathrm{perf}}}\cdots\times_{\mathcal{Y}_{X,I}^{\mathrm{perf}}} \mathcal{Y}_{X',I}^{\mathrm{perf}}$ be the fibre product of $n$ copies of $\mathcal{Y}_{X',I}^{\mathrm{perf}}$ over $\mathcal{Y}_{X,I}^{\mathrm{perf}}$ for $n\geq 1$. By induction on $n\geq 1$ and (\ref{compatible with fibre prod}), we have
		\[\mathcal{Y}_{X',I}^{\mathrm{perf}}\times G^{n-1}\cong \left(\mathcal{Y}_{X',I}^{\mathrm{perf}}\right)^{\times n} \cong \mathcal{Y}_{X'^{\times n},I}^{\mathrm{perf}},\quad n\geq 1.\]
		This induces an isomorphism
		\begin{align}\label{torsor induce isom on tensor}
			\mathcal{O}\left(\mathcal{Y}_{X'^{\times n},I}^{\mathrm{perf}}\right)\cong C\left(G^{n-1},\mathcal{O}\left(\mathcal{Y}_{X',I}\right)\widehat{\otimes}_{\mathcal{O}_E}\mathcal{O}_{E(\pi^{1/p^{\infty}})}\right), \quad n\geq 1.
		\end{align}
		Combining (\ref{function tensor agree with tensor function}) and (\ref{torsor induce isom on tensor}), we see that the complex $C^{\bullet}\widehat{\otimes}_{\mathcal{O}_E}\mathcal{O}_{E(\pi^{1/p^{\infty}})}$ is isomorphic to the \v{C}ech complex 
		\[0\to \mathcal{O}\left(\mathcal{Y}_{X,I}^{\mathrm{perf}}\right)\to \mathcal{O}\left(\mathcal{Y}_{X',I}^{\mathrm{perf}}\right)\to \mathcal{O}\left(\mathcal{Y}_{X'^{\times 2},I}^{\mathrm{perf}}\right)\to\cdots.\]
		The verification of compatibility of differentials is analogous to the proof of \cite[Theorem 5.3]{du2025multivariablevarphiqmathcaloktimesmodulesassociatedpadic}, hence we omit details.
		Then the conclusion follows from Proposition \ref{v-sheaf and v-stack}.
	\end{proof}
	
	\begin{defn}
		If $I\subseteq\closeopen{0}{\infty}$ is an interval as above and $R$ is a perfectoid algebra over $\mathbb{F}$ with a pseudo-uniformizer $\varpi$, we define three rings as follows:
		\begin{align}\label{defn of perfd robba rings}
			B_{R,I}\coloneq \mathcal{O}\left(\mathcal{Y}_{\mathrm{Spa}(R,R^{\circ}),I}\right),\quad \widetilde{\mathcal{R}}_{R}\coloneq \varinjlim_{r\to 0^{+}}B_{R,\openclose{0}{r}},\quad \widetilde{\mathcal{R}}^{\mathrm{int}}_{R}\coloneq \varinjlim_{r\to 0^{+}}B_{R,[0,r]}.
		\end{align}
		The ring $\widetilde{\mathcal{R}}_{R}$ is called the perfectoid Robba ring and $\widetilde{\mathcal{R}}^{\mathrm{int}}_{R}$ is called the integral perfectoid Robba ring. Note that $\pi$ is invertible in $B_{R,I}$ if and only if $0\notin I$, hence $\pi$ is not invertible in $\widetilde{\mathcal{R}}^{\mathrm{int}}_{R}$. Also note that although $B_{R,[0,r]}$ and $B_{R,\openclose{0}{r}}$ depend on the choice of $\varpi$, $\widetilde{\mathcal{R}}_{R}$ and $\widetilde{\mathcal{R}}^{\mathrm{int}}_{R}$ do not depend on the choice of $\varpi$.
	\end{defn}

	If $L$ is a perfectoid field, Kedlaya proved that $ B_{L,I}$ has very strong ring-theoretic properties (\cite[Corollary 2.10]{kedlaya2016noetherian}, \cite[Lemma 4.2.6, Remark 4.2.7]{kedlaya2015relative}):
	
	\begin{prop}\label{Bezout property of Robba ring over fields}
		The ring $B_{L,I}$ is a principal ideal domain, if $I=[0,r]$ or $[s,r]$ with $0<1/r,1/s\in \mathbb{Z}[1/p]$, and is a Bézout domain if $I=\openclose{0}{r}$ for $1/r\in \mathbb{Z}[1/p]$. Hence as colimits of Bézout domains, $\widetilde{\mathcal{R}}_{L}$ and $\widetilde{\mathcal{R}}^{\mathrm{int}}_{L}$ are also Bézout domains.
	\end{prop}
	
	In particular, any finite projective module over $\widetilde{\mathcal{R}}^{\mathrm{int}}_{L}$ is free. We can prove an analogue:
	
	\begin{prop}\label{fin proj over int Robba is free}
		If $R$ is a perfectoid $\mathbb{F}$-algebra such that any finite projective $R$-module is free, then any finite projective $\widetilde{\mathcal{R}}_{R}^{\mathrm{int}}$-module is free.
	\end{prop}
	
	\begin{proof}
		First we check that $\pi$ is contained in the Jacobson radical of $\widetilde{\mathcal{R}}_{R}^{\mathrm{int}}$: let $x\in \widetilde{\mathcal{R}}_{R}^{\mathrm{int}}$, there exists $k\geq 0$ and $0<1/r\in\mathbb{Z}[1/p]$ such that $[\varpi]^kx\in \left(B_{R,[0,r]}\right)^{\circ}$by (\ref{description of [0,r]}). Without any loss of generality, we may assume that $k\geq 1/r$, then $1+\pi x$ is invertible in $\left(B_{R,[0,1/(k+1)]}\right)^{\circ}$, hence in $\widetilde{\mathcal{R}}_{R}^{\mathrm{int}}$. Therefore $\pi$ is contained in the Jacobson radical of $\widetilde{\mathcal{R}}_{R}^{\mathrm{int}}$ by \cite[\href{https://stacks.math.columbia.edu/tag/0AME}{Tag 0AME}]{stacks-project}. Now suppose that $M$ is a finite projective $\widetilde{\mathcal{R}}_{R}^{\mathrm{int}}$-module, since $\widetilde{\mathcal{R}}_{R}^{\mathrm{int}}/\varpi\cong R$, we see that $M/\pi M$ is a free $R$-module of finite rank, hence there exists $e_1,\dots,e_r\in M$ such that $M=\sum_{i=1}^r\widetilde{\mathcal{R}}_{R}^{\mathrm{int}}e_i+\pi M$. Then Nakayama's lemma (see for example \cite[\href{https://stacks.math.columbia.edu/tag/00DV}{Tag 00DV}]{stacks-project}) implies that $M=\sum_{i=1}^r\widetilde{\mathcal{R}}_{R}^{\mathrm{int}}e_i$. Now suppose that $\sum_{i=1}^ra_ie_i=0$ for some $a_i\in \widetilde{\mathcal{R}}_{R}^{\mathrm{int}}$, then $M/\pi M=\bigoplus_{i=1}^r Re_i$ implies that $\pi\mid a_i$ for each $i$, hence $\pi\cdot \sum_{i=1}^r\frac{a_i}{\pi} e_i=0$. Since $M$ is a projective $\widetilde{\mathcal{R}}_{R}^{\mathrm{int}}$-module, it follows that $\sum_{i=1}^r\frac{a_i}{\pi} e_i=0$. Repeating this argument, we deduce that $\pi^n\mid a_i$ for any $n\geq 1$, hence $a_i=0$, hence $M$ is a free $\widetilde{\mathcal{R}}_{R}^{\mathrm{int}}$-module with a basis given by $e_1,\dots,e_r$.
	\end{proof}
	
	\begin{remark}
		In fact, we can show that $(\widetilde{\mathcal{R}}_{R}^{\mathrm{int}},\pi\widetilde{\mathcal{R}}_{R}^{\mathrm{int}})$ is always a henselian pair (for the definition, see for example \cite[\href{https://stacks.math.columbia.edu/tag/09XD}{Tag 09XD}]{stacks-project}). We do not need this fact.
	\end{remark}

	\subsection{Overconvergent perfectoid $(\varphi_q,\mathcal{O}_K^{\times})$-modules}
	In this section, we show that to every finite free continuous $\mathcal{O}_E$-representation $\rho$ of $G_K$ and $r\in\mathbb{Q}_{>0}$, we can associate a finite projective module over $B_{A_{\infty},[0,r]}$. Moreover, for $r$ sufficiently small, it will be a free $B_{A_{\infty},[0,r]}$-module.
	
	\subsubsection{Lubin-Tate case}
	In the Lubin–Tate case, although the usual overconvergence does not hold in general (\cite[Remark 5.21]{fourquaux2014triangulable}), perfectoid overconvergence remains valid. Results of this kind are well known in the literature, and there are different proofs. For example, in the proof of \cite[Proposition III.3.1]{cherb1996}, Cherbonnier–Colmez apply a variant of Sen’s method to perform Galois descent. In \cite[Theorem 2.4.5]{kedlaya2015new} and \cite[Proposition 4.8]{de2019induction}, the argument proceeds by descending $\varphi$-modules over the ring of Witt vectors to $\varphi$-modules over the integral perfectoid Robba rings. Here, we present a geometric proof using Proposition \ref{v-sheaf and v-stack} (ii), which is somewhat also a Galois descent but is conceptually simpler and less computational.
	
	Fix an embedding $\sigma_{0}: \mathcal{O}_K\hookrightarrow\mathcal{O}_E$, which also induces an embedding $\sigma_{0}:\mathbb{F}_q\hookrightarrow \mathbb{F}$. We fix a Frobenius power series $p_{\mathrm{LT}}\in \mathcal{O}_K[\negthinspace[T_{\mathrm{LT}}]\negthinspace]$ associated to the uniformizer $p$ which defines a Lubin-Tate formal group, and let $K_{\infty}$ be the corresponding Lubin-Tate extension of $K$, and let $H_K\coloneq \gal(\overline{K}/K_{\infty})$. Since $K_{\infty}^{\flat}=\mathbb{F}_q(\negthinspace(T_{\mathrm{LT}}^{1/p^{\infty}})\negthinspace)$ (see for example \cite[$\S$1.4]{schneider2017galois}), we choose $T_{\mathrm{LT}}$ as the fixed pseudo-uniformizer of $\mathbb{F}\otimes_{\sigma_{0},\mathbb{F}_q}K_{\infty}^{\flat}=\mathbb{F}(\negthinspace(T_{\mathrm{LT}}^{1/p^{\infty}})\negthinspace)$ to define $B_{\mathbb{F}(\negthinspace(T_{\mathrm{LT}}^{1/p^{\infty}})\negthinspace),[0,r]}$ for $r\in\mathbb{Q}_{>0}$ (see (\ref{defn of perfd robba rings})). 
	For $0<1/r\in\mathbb{Z}[1/p]$, we also define
	\begin{align*}
		B_{K_{\infty}^{\flat},[0,r],K}&\coloneq W\left(\mathbb{F}_q[\negthinspace[T_{\mathrm{LT}}^{1/p^{\infty}}]\negthinspace]\right)\left\langle\frac{p}{[T_{\mathrm{LT}}]^{1/r}}\right\rangle\left[\frac{1}{[T_{\mathrm{LT}}]}\right],\\
		B_{\mathbb{C}_p^{\flat},[0,r],K}&\coloneq W\left(\mathcal{O}_{\mathbb{C}_p}^{\flat}\right)\left\langle\frac{p}{[T_{\mathrm{LT}}]^{1/r}}\right\rangle\left[\frac{1}{[T_{\mathrm{LT}}]}\right],
	\end{align*}
	then $B_{\mathbb{F}(\negthinspace(T_{\mathrm{LT}}^{1/p^{\infty}})\negthinspace),[0,r]}=B_{K_{\infty}^{\flat},[0,r],K}\otimes_{\mathcal{O}_K,\sigma_{0}}\mathcal{O}_E$. The $q$-th Frobenius endomorphism of $K_{\infty}^{\flat}$ and $\mathbb{C}_p^{\flat}$ induce $\mathcal{O}_K$-linear isomorphisms 
	\begin{align}\label{phi_q: (LT,K) and (C_p, K)}
		\varphi_q:B_{K_{\infty}^{\flat},[0,r],K}\to B_{K_{\infty}^{\flat},[0,r/q],K},\quad \varphi_q:B_{\mathbb{C}_p^{\flat},[0,r],K}\to B_{\mathbb{C}_p^{\flat},[0,r/q],K}
	\end{align}
	and an $\mathcal{O}_E$-linear isomorphism
	\begin{align}\label{phi_q: (LT,E)}
		\varphi_q: B_{\mathbb{F}(\negthinspace(T_{\mathrm{LT}}^{1/p^{\infty}})\negthinspace),[0,r]}\to B_{\mathbb{F}(\negthinspace(T_{\mathrm{LT}}^{1/p^{\infty}})\negthinspace),[0,r/q]}.
	\end{align}
	
	Let $\rho$ be a finite type continuous $\mathcal{O}_E$-representation of $G_K$. By Corollary \ref{G-torsor},
	\[H^0_{\mathrm{cont}}(H_K,B_{\mathbb{C}_p^{\flat},[0,r],K})=B_{K_{\infty}^{\flat},[0,r],K}\] 
	since $\mathrm{Spa}\mathbb{C}_p^{\flat}\to \mathrm{Spa}K_{\infty}^{\flat}$ is a pro-étale $H_K$-torsor (\cite[Theorem 7.3.1]{WS2020berkeley}). Thus for $r\in\mathbb{Q}_{>0}$,
	\[D_{\mathrm{LT}}^{[0,r]}(\rho)\coloneq\left(B_{\mathbb{C}_p^{\flat},[0,r],K}\otimes_{\mathcal{O}_K,\sigma_{0}}\rho\right)^{H_K}\]
	is a module over $B_{K_{\infty}^{\flat},[0,r],K}\otimes_{\mathcal{O}_K,\sigma_{0}}\mathcal{O}_E=B_{\mathbb{F}(\negthinspace(T_{\mathrm{LT}}^{1/p^{\infty}})\negthinspace),[0,r]}$ with a semi-linear $\mathcal{O}_K^{\times}$-action. The map (\ref{phi_q: (LT,K) and (C_p, K)}) induces a bijective map
	\begin{align}\label{phi_q: (LT) module}
		\varphi_q: D_{\mathrm{LT}}^{[0,r]}(\rho)\to D_{\mathrm{LT}}^{[0,r/q]}(\rho).
	\end{align}
	
	\begin{lem}
		If $\rho$ is a free $\mathcal{O}_E$-representation of $G_K$ of rank $d\geq 1$, then $D_{\mathrm{LT}}^{[0,r]}(\rho)$
		is a free $B_{\mathbb{F}(\negthinspace(T_{\mathrm{LT}}^{1/p^{\infty}})\negthinspace),[0,r]}$-module of rank $d$. Moreover, there are natural isomorphisms
		\[B_{\mathbb{C}_p^{\flat},[0,r],K}\otimes_{B_{K_{\infty}^{\flat},[0,r],K}}D_{\mathrm{LT}}^{[0,r]}(\rho)\cong B_{\mathbb{C}_p^{\flat},[0,r],K}\otimes_{\mathcal{O}_K,\sigma_{0}}\rho\]
		and
		\[B_{\mathbb{F}(\negthinspace(T_{\mathrm{LT}}^{1/p^{\infty}})\negthinspace),[0,r/q]}\otimes_{\varphi_q,B_{\mathbb{F}(\negthinspace(T_{\mathrm{LT}}^{1/p^{\infty}})\negthinspace),[0,r]}}D_{\mathrm{LT}}^{[0,r]}(\rho)\xrightarrow[\sim]{\mathrm{id}\otimes\varphi_q}D_{\mathrm{LT}}^{[0,r/q]}(\rho).\]
	\end{lem}
	\begin{proof}
		Note that $\mathrm{Spa}\mathbb{C}_p^{\flat}\to \mathrm{Spa}K_{\infty}^{\flat}$ is a pro-étale $H_K$-torsor (\cite[Theorem 7.3.1]{WS2020berkeley}). The first isomorphism and the fact that $D_{\mathrm{LT}}^{[0,r]}(\rho)$ is a finite projective module over $B_{K_{\infty}^{\flat},[0,r],K}$ locally free of rank $d\cdot[E:K]$ follow from Proposition \ref{v-sheaf and v-stack} (ii). By Proposition \ref{v-sheaf and v-stack} (ii),  the second isomorphism can be checked after applying $B_{\mathbb{C}_p^{\flat},[0,r/q],K}\otimes_{B_{K_{\infty}^{\flat},[0,r/q],K}}-$. Hence it remains to show that $D_{\mathrm{LT}}^{[0,r]}(\rho)$ is free of rank $d$ over $B_{\mathbb{F}(\negthinspace(T_{\mathrm{LT}}^{1/p^{\infty}})\negthinspace),[0,r]}$. Since $B_{\mathbb{F}(\negthinspace(T_{\mathrm{LT}}^{1/p^{\infty}})\negthinspace),[0,r]}$ is a PID (Proposition \ref{Bezout property of Robba ring over fields}), it suffices to show that $D_{\mathrm{LT}}^{[0,r]}(\rho)$ is torsion-free over $B_{\mathbb{F}(\negthinspace(T_{\mathrm{LT}}^{1/p^{\infty}})\negthinspace),[0,r]}$. Suppose that there exists  $0\neq x\in B_{\mathbb{F}(\negthinspace(T_{\mathrm{LT}}^{1/p^{\infty}})\negthinspace),[0,r]}$ and $0\neq t\in D_{\mathrm{LT}}^{[0,r]}(\rho)$ such that $xt=0$. Since every PID is a UFD, without any loss of generality, we may assume that $xB_{\mathbb{F}(\negthinspace(T_{\mathrm{LT}}^{1/p^{\infty}})\negthinspace),[0,r]}$ is a prime ideal, hence $\mathfrak{p}\coloneq xB_{\mathbb{F}(\negthinspace(T_{\mathrm{LT}}^{1/p^{\infty}})\negthinspace),[0,r]}\cap B_{\mathbb{F}(\negthinspace(T_{\mathrm{LT}}^{1/p^{\infty}})\negthinspace),[0,r],K}$ is a non-zero prime ideal of $B_{\mathbb{F}(\negthinspace(T_{\mathrm{LT}}^{1/p^{\infty}})\negthinspace),[0,r],K}$ by applying \cite[Theorem 9.3.(ii)]{matsumura1989commutative} to the zero ideal. Let $y\in \mathfrak{p}$ be a non-zero element, hence $yt=0$. This leads to a contradiction, since $D_{\mathrm{LT}}^{[0,r]}(\rho)$ is finite projective over $B_{\mathbb{F}(\negthinspace(T_{\mathrm{LT}}^{1/p^{\infty}})\negthinspace),[0,r],K}$. Thus $D_{\mathrm{LT}}^{[0,r]}(\rho)$ is torsion-free over $B_{\mathbb{F}(\negthinspace(T_{\mathrm{LT}}^{1/p^{\infty}})\negthinspace),[0,r]}$.
	\end{proof}
	\subsubsection{Multivariable case}
	
	Let 
	\begin{align}\label{A_{infty}}
		A_{\infty}\coloneq \mathbb{F}\left(\negthinspace\left(Y_{\sigma_0}^{1/p^{\infty}}\right)\negthinspace\right)\left\langle\left(\dfrac{Y_{\sigma_{i}}}{Y_{\sigma_0}}\right)^{\pm 1/p^{\infty}}: 1\leq i\leq f-1\right\rangle
	\end{align}
	be the completed perfection of the ring $A$, and let
	\begin{align}\label{A_{infty}'}
		A_{\infty}'\coloneq \mathbb{F}\left(\negthinspace\left(T_{\mathrm{LT},0}^{1/p^{\infty}}\right)\negthinspace\right)\left\langle\left(\dfrac{T_{\mathrm{LT},{i}}}{T_{\mathrm{LT},0}^{p^{i}}}\right)^{\pm 1/p^{\infty}}: 1\leq i\leq f-1\right\rangle.
	\end{align}
	There is a continuous $\mathbb{F}$-linear $(\varphi_q,(\mathcal{O}_K^{\times})^{f})$-action on $A_{\infty}'$, see \cite[(28), (47)]{breuil2023multivariable}. Recall that there is a map $m:A_{\infty}\hookrightarrow A_{\infty}'$ induced by functoriality such that $m: \mathrm{Spa}(A_{\infty}',\left(A_{\infty}'\right)^{\circ})\to \mathrm{Spa}(A_{\infty},A_{\infty}^{\circ})$ is a pro-étale $\Delta_1$-torsor, where $\Delta_1=\{(a_0,\dots,a_{f-1})\in(\mathcal{O}_K^{\times})^{f}: a_0a_1\cdots a_{f-1}=1\}$ (\cite[Proposition 2.4.4]{breuil2023multivariable}). We choose $Y_{\sigma_{0}}\in A_{\infty}\subset A_{\infty}'$ as the fixed uniformizer to define $B_{A_{\infty},[0,r]}$ and $B_{A_{\infty}',[0,r]}$ for $r\in\mathbb{Q}_{>0}$. The $\mathbb{F}$-linear automorphism $\varphi_q$ of $A_{\infty}$ and $A_{\infty}'$ induces $\mathcal{O}_E$-linear isomorphisms
	\begin{align}\label{phi_q: (A_{infty})}
		\varphi_q:B_{A_{\infty},[0,r]}\to B_{A_{\infty},[0,r/q]},\quad \varphi_q:B_{A_{\infty}',[0,r]}\to B_{A_{\infty}',[0,r/q]}.
	\end{align}
	The map $\mathrm{pr}_i: \mathbb{F}(\negthinspace(T_{\mathrm{LT}}^{1/p^{\infty}})\negthinspace)\to A_{\infty}', T_{\mathrm{LT}}\mapsto T_{\mathrm{LT},i}$ is $(\varphi_q,(\mathcal{O}_K^{\times})^f)$-equivariant for $0\leq i\leq f-1$, where $(\mathcal{O}_K^{\times})^f$ acts on $\mathbb{F}_q(\negthinspace(T_{\mathrm{LT}}^{1/p^{\infty}})\negthinspace)$ via the $i$-th projection. Using \cite[p.50, (56)]{breuil2023multivariable}, it induces a map
	\begin{align}\label{pr_i}
		\mathrm{pr}_i: B_{\mathbb{F}(\negthinspace(T_{\mathrm{LT}}^{1/p^{\infty}})\negthinspace),[0,\frac{p-1}{(q-1)p^{i}}r]}\hookrightarrow B_{A_{\infty}',[0,r]}.
	\end{align}
	Let $\rho$ be a finite type continuous $\mathcal{O}_E$-representation of $G_K$. For $0\leq i\leq f-1$, we define
	\[D_{A_{\infty},i}^{[0,r]}(\rho)\coloneq\left(B_{A_{\infty}',[0,r]}\otimes_{B_{\mathbb{F}(\negthinspace(T_{\mathrm{LT}}^{1/p^{\infty}})\negthinspace),[0,\frac{p-1}{(q-1)p^{i}}r]},\mathrm{pr}_i}D_{\mathrm{LT}}^{[0,\frac{p-1}{(q-1)p^{i}}r]}(\rho)\right)^{\Delta_1}.\]
	This is a $B_{A_{\infty},[0,r]}$-module with a semi-linear $\mathcal{O}_K^{\times}$-action since 
	\[H^0_{\mathrm{cont}}(\Delta_1,B_{A_{\infty}',[0,r]})=B_{A_{\infty},[0,r]}\]
	by Corollary \ref{G-torsor}. Moreover, (\ref{phi_q: (LT) module}) and (\ref{phi_q: (A_{infty})}) induce a bijective map
	\begin{align}\label{phi_q: mv module}
		\varphi_q: D_{A_{\infty},i}^{[0,r]}(\rho)\to D_{A_{\infty},i}^{[0,r/q]}(\rho).
	\end{align}
	
	\begin{thm}\label{overconvergence of mv module at the perfectoid level}
		If $\rho$ is an $\mathcal{O}_E$-representation of $G_K$ free of rank $d\geq 1$, then $D_{A_{\infty},i}^{[0,r]}(\rho)$
		is a $B_{A_{\infty},[0,r]}$-module locally free of rank $d$. Moreover, there are natural isomorphisms
		\[B_{A_{\infty}',[0,r]}\otimes_{B_{A_{\infty},[0,r]}}D_{A_{\infty},i}^{[0,r]}(\rho)\cong B_{A_{\infty}',[0,r]}\otimes_{B_{\mathbb{F}(\negthinspace(T_{\mathrm{LT}}^{1/p^{\infty}})\negthinspace),[0,\frac{p-1}{(q-1)p^{i}}r]},\mathrm{pr}_i}D_{\mathrm{LT}}^{[0,\frac{p-1}{(q-1)p^{i}}r]}(\rho)\]
		and
		\[B_{A_{\infty},[0,r/q]}\otimes_{\varphi_q,B_{A_{\infty},[0,r]}}D_{A_{\infty},i}^{[0,r]}(\rho)\xrightarrow[\sim]{\mathrm{id}\otimes\varphi_q}D_{A_{\infty},i}^{[0,r/q]}(\rho).\]
	\end{thm}
	
	\begin{proof}
		Since $m: \mathrm{Spa}(A_{\infty}',(A_{\infty}')^{\circ})\to \mathrm{Spa}(A_{\infty},A_{\infty}^{\circ})$ is a pro-étale $\Delta_1$-torsor, this is a direct consequence of Proposition \ref{v-sheaf and v-stack} (ii).
	\end{proof}
	
	It is unknown whether every finite projective module over $B_{A_{\infty},[0,r]}$ is necessarily free, although this is known to be the case for $A_{\infty}$ by \cite[Theorem 2.19]{DH21Projective}. However, using Corollary \ref{fin proj over int Robba is free}, we can show that $D_{A_{\infty},i}^{[0,r]}(\rho)$ is free for $r$ sufficiently small.
	\begin{prop}\label{for r large is free}
		Let $\rho$ be a free continuous $\mathcal{O}_E$-representation of $G_K$ of rank $d$, there exists $r(\rho)\in\mathbb{Q}_{>0}$ such that for any $0<r<r(\rho)$ with $r\in\mathbb{Q}$, $D_{A_{\infty},i}^{[0,r]}(\rho)$
		is a free $B_{A_{\infty},[0,r]}$-module of rank $d$.
	\end{prop}
	\begin{proof}
		First we fix $r_0\in \mathbb{Q}_{>0}$. By Corollary \ref{fin proj over int Robba is free}, $\widetilde{D}_{A_{\infty},i}^{\mathrm{int}}(\rho)\coloneq\widetilde{\mathcal{R}}_{A_{\infty}}^{\mathrm{int}}\otimes_{B_{A_{\infty},[0,r_0]}}D_{A_{\infty},i}^{[0,r_0]}(\rho)$ is a free $\widetilde{\mathcal{R}}_{A_{\infty}}^{\mathrm{int}}$-module of rank $d$. Thus there exists $e_1,\dots,e_d\in \widetilde{D}_{A_{\infty},i}^{\mathrm{int}}(\rho)$ which form a basis over $\widetilde{\mathcal{R}}_{A_{\infty}}^{\mathrm{int}}$. Suppose that $x_1,\dots,x_n\in D_{A_{\infty},i}^{[0,r_0]}(\rho)$ generate $D_{A_{\infty},i}^{[0,r_0]}(\rho)$ over $B_{A_{\infty},[0,r_0]}$, then $x_i=\sum_{j=1}^da_{ij}e_j$ for some $a_{ij}\in \widetilde{\mathcal{R}}_{A_{\infty}}^{\mathrm{int}}$. Since $\widetilde{\mathcal{R}}_{A_{\infty}}^{\mathrm{int}}=\varinjlim\limits_{r\to 0^{+}} B_{A_{\infty},[0,r]}$, there exists $0<r(\rho)\leq r_0$ such that $a_{ij}\in B_{A_{\infty},[0,r(\rho)]}$ and $e_j\in D_{A_{\infty},i}^{[0,r(\rho)]}(\rho)= B_{A_{\infty},[0,r(\rho)]}\otimes_{B_{A_{\infty},[0,r_0]}}D_{A_{\infty},i}^{[0,r_0]}(\rho)$ for any $i,j$. Then $D_{A_{\infty},i}^{[0,r(\rho)]}(\rho)$ is a free $B_{A_{\infty},[0,r(\rho)]}$-module with a basis $e_1,\dots,e_d$, and the conclusion follows. 
	\end{proof}
	
	\begin{remark}
		Our first proof of Theorem \ref{overconvergence of mv module at the perfectoid level} and Proposition \ref{for r large is free} used the generalized Colmez-Sen-Tate method developed in \cite{colmez1998theorie}, \cite{cherbonnier1998representations} and \cite{berger2008familles}. This method only ensured that Theorem \ref{overconvergence of mv module at the perfectoid level} holds for sufficiently small $r$, whereas the above method shows that the overconvergent radius at the perfectoid level can be arbitrary. 
	\end{remark}
	
	\begin{remark}
		Using the same method as in \cite[Theorem 2.4.5]{kedlaya2015new} and \cite[Proposition 4.8]{de2019induction}, one can show that the base change functor from the category of finite projective étale $\varphi_q$-modules over $\widetilde{\mathcal{R}}^{\mathrm{int}}_{A_{\infty}}$ to the category of finite projective étale $\varphi_q$-modules over $W_{\mathcal{O}_E}(A_{\infty})$ is an equivalence of categories, which is a special case of \cite[Theorem 4.5.7]{kedlaya2019relative}. This gives a different proof of the perfectoid overconvergence.
	\end{remark}
	
	
	\subsection{The relation between $\widetilde{\mathcal{R}}^{\mathrm{int}}_{A_{\infty}}$ and $A_{\mathrm{mv},E}^{\dagger}$}
	Recall that in \cite[(3.4)]{du2025multivariablevarphiqmathcaloktimesmodulesassociatedpadic} there is a $(\varphi_q,\mathcal{O}_K^{\times})$-equivariant map
	\begin{align}
		\iota: A_{\mathrm{mv},E}\hookrightarrow W_{\mathcal{O}_E}(A_{\infty}).
	\end{align}
	More precisely,  $\iota(Y_{\sigma_{0}}),\dots,\iota(Y_{\sigma_{f-1}})$ are characterized by the following lemma:
	\begin{lem}\label{decription of iota}
		There exist unique elements $y_0,\dots,y_{f-1}\in W(A_{\infty})$ such that for $0\leq i\leq f-1$,
		\begin{align}
			\begin{cases}
				\varphi (y_i)=F_i(y_0,\dots,y_{f-1}), &\\
				\overline{y_i}= Y_{\sigma_i}\in A_{\infty},&
			\end{cases}
		\end{align}
		where $\varphi: W(A_{\infty})\to W(A_{\infty})$ is induced by the $\mathbb{F}$-linear endomorphism $\varphi$ of $A_{\infty}$, see \cite[p.33, (43)]{breuil2023multivariable}, and $F_0,\dots,F_{f-1}$ are formal power series with $\mathcal{O}_K$-coefficients such that in $\mathcal{O}_K[\negthinspace[\mathcal{O}_K]\negthinspace]$ we have $\varphi(Y_{\sigma_{i}})=F_i(Y_{\sigma_{0}},\dots,Y_{\sigma_{f-1}})$.
		Moreover, $y_0,\dots,y_{f-1}\in W(A_{\infty}^{\circ\circ})$.
	\end{lem}
	\begin{proof}
		For the existence, let $y_i^{(0)}\coloneq [Y_{\sigma_{i}}]$, and $y_i^{(n)}\coloneq \varphi^{-1}\left(F_i\left(y_0^{(n-1)},\dots,y_{f-1}^{(n-1)}\right)\right)$ for $n\geq 1$. We claim that $y_i^{(n)}\equiv y_i^{(n-1)} \mod p^{n}$ for $n\geq 1$, then we can take $y_i\coloneq \lim_{n\to\infty}y_i^{(n)}$, which lies in $W(A_{\infty}^{\circ\circ})$ by the construction, and the sequence $y_0,\dots,y_{f-1}$ works. Now we prove the claim by induction on $n\geq 1$. First using (\ref{endomorphism phi_q}), we see the claim holds for $n=1$. Then suppose that the claim holds for $n= m$. Hence we can write $y_i^{(m)}=y_i^{(m-1)}+p^mt_i$ for some $t_i\in W(A_{\infty})$. By the definition of $\varphi$, it suffices to check that
		\begin{align}\label{need to check mod p^m}
			F_i\left(y_0^{(m)},\dots,y_{f-1}^{(m)}\right)\equiv F_i\left(y_0^{(m-1)},\dots,y_{f-1}^{(m-1)}\right) \mod p^{m+1}.
		\end{align}
		Using (\ref{endomorphism phi_q}) and
		\[F_i\left(y_0^{(m)},\dots,y_{f-1}^{(m)}\right)=F_i\left(y_0^{(m-1)}+p^mt_0,\dots,y_{f-1}^{(m-1)}+p^mt_{f-1}\right),\]
		the equality (\ref{need to check mod p^m}) is easy to verify, and we omit the details. For the uniqueness, suppose that $z_0,\dots,z_{f-1}$ is another sequence satisfying the same conditions, then $z_i\equiv y_i^{(0)}\mod p$, hence we can define $z_i^{(n)}\coloneq \varphi^{-1}\left(F_i\left(z_0^{(n-1)},\dots,z_{f-1}^{(n-1)}\right)\right)$ for $n\geq 1$ and $z_i^{(0)}\coloneq z_i$. Our assumption on $z_i$ implies that $z_i^{(n)}=z_{i}$ for any $n\geq 0$. By induction on $n\geq 0$, as in the proof of the claim above, we can show that $z_i^{(n)}\equiv y_i^{(n)} \mod p^{n+1}$, hence after taking limits, we have $z_i=y_i$.
	\end{proof}
	
	Recall that $A_{\infty}$ is a Banach algebra with a normalized multiplicative norm such that $|Y_{\sigma_{0}}|=p^{-1}$ (\cite[Lemma 2.4.2]{breuil2023multivariable}), hence by (\ref{defn of ||_r}), there is a multiplicative norm $|\cdot|_r$ on $B_{A_{\infty},[0,r]}$. Recall also that $A_{\mathrm{mv},E}^{\dagger,s^{-}}$ and $\lVert\cdot \rVert_{s}$ are defined in (\ref{defn of s^{-1}}) and (\ref{defn of || ||_s}).
	
	\begin{prop}
		Let $s$ be a positive integer. The embedding $\iota$ satisfies 
		\begin{enumerate}
			\item 
			$\iota\left(A_{\mathrm{mv},E}^{\dagger,s^{-}}\right)\subseteq \left(B_{A_{\infty},[0,1/s]}\right)^{\circ}=W_{\mathcal{O}_E}(A_{\infty}^{\circ})\left\langle\frac{\pi}{[Y_{\sigma_{0}}]^{s}}\right\rangle$,
			\item 
			$\iota\left(A_{\mathrm{mv},E}^{\dagger,s^{-}}\left[\frac{1}{Y_{\sigma_{0}}}\right]\right)\subseteq B_{A_{\infty},[0,1/s]}$,
			\item 
			For $f\in A_{\mathrm{mv},E}^{\dagger,s^{-}}$, we have
			\begin{align}\label{equality of two norms}
				\lVert f\rVert_{s}^s=\left|\iota(f)\right|_{1/s}.
			\end{align}
		\end{enumerate}
	\end{prop}
	
	\begin{proof}
		By Lemma \ref{decription of iota}, for $0\leq i\leq f-1$, we have
		\[\iota(Y_{\sigma_{i}})\in [Y_{\sigma_{i}}]+p W(A_{\infty}^{\circ\circ}) \subseteq [Y_{\sigma_{i}}]+\pi W_{\mathcal{O}_E}(A_{\infty}^{\circ\circ}),\]
		this implies that $\frac{\iota(Y_{\sigma_{i}})}{[Y_{\sigma_{i}}]}\in 1+\frac{\pi}{[Y_{\sigma_{i}}]}W_{\mathcal{O}_E}(A_{\infty}^{\circ\circ})$ hence is invertible in $\left(B_{A_{\infty},[0,1]}\right)^{\circ}$. We deduce that
		\[\iota\left(\frac{\pi}{Y_{\sigma_{0}}^s}\right)=\frac{\pi}{[Y_{\sigma_{0}}]^s}\left(\dfrac{[Y_{\sigma_{0}}]}{\iota(Y_{\sigma_{0}})}\right)^s\in \left(B_{A_{\infty},[0,1/s]}\right)^{\circ},\]
		\[\iota\left(\dfrac{Y_{\sigma_{i}}}{Y_{\sigma_j}}\right)=\dfrac{[Y_{\sigma_{i}}]}{[Y_{\sigma_{j}}]}\cdot \dfrac{\iota(Y_{\sigma_{i}})}{[Y_{\sigma_{i}}]}\cdot\dfrac{[Y_{\sigma_{i}}]}{\iota(Y_{\sigma_{j}})}\in \left(B_{A_{\infty},[0,1]}\right)^{\circ},\]
		and
		\[\iota\left(\dfrac{1}{Y_{\sigma_{0}}}\right)=\dfrac{1}{[Y_{\sigma_{0}}]}\cdot\dfrac{[Y_{\sigma_{0}}]}{\iota(Y_{\sigma_{i}})}\in B_{A_{\infty},[0,1]}.\]
		Thus $\iota\left(A_{\mathrm{mv},E}^{\dagger,s^{-}}\right)\subseteq \left(B_{A_{\infty},[0,1/s]}\right)^{\circ}$ and $\iota\left(A_{\mathrm{mv},E}^{\dagger,s^{-}}\left[\frac{1}{Y_{\sigma_{0}}}\right]\right)\subseteq B_{A_{\infty},[0,1/s]}$. Now consider the commutative diagram
		\begin{equation}
			\begin{tikzcd}
				A_{\mathrm{mv},E}^{\dagger,s^{-}}\arrow[d,two heads,"\mathrm{mod}\ Y_{\sigma_{0}}=\mathrm{mod}\ \left(A_{\mathrm{mv},E}^{\dagger,s^{-}}\right)^{\circ\circ}"] \arrow[rr,hook,"\iota"]& & \left(B_{A_{\infty},[0,1/s]}\right)^{\circ}=W_{\mathcal{O}_E}(A_{\infty}^{\circ})\left\langle\frac{\pi}{[Y_{\sigma_{0}}]^{s}}\right\rangle\arrow[d,two heads,"\mathrm{mod}\, \left(B_{A_{\infty},[0,1/s]}\right)^{\circ\circ}"]\\
				\mathbb{F}\left[\dfrac{\pi}{Y_{\sigma_{0}}^s},\left(\dfrac{Y_{\sigma_{i}}}{Y_{\sigma_{0}}}\right)^{\pm1}\right]\arrow[rr,hook]& & \mathbb{F}\left[\dfrac{\pi}{Y_{\sigma_{0}}^s},\left(\dfrac{Y_{\sigma_{i}}}{Y_{\sigma_{0}}}\right)^{\pm1/p^{\infty}}\right].
			\end{tikzcd}
		\end{equation}
		Let $f\in A_{\mathrm{mv},E}^{\dagger,s^{-}}$. If $\lVert f\rVert_{s}=1$, $f\not\equiv 0\mod Y_{\sigma_{0}}$, thus $\iota(f)\notin \left(B_{A_{\infty},[0,1/s]}\right)^{\circ\circ}$, hence $|\iota(f)|_{1/s}=1$. In general, $f=Y_{\sigma_{0}}^k\cdot f_0$ for some $f_0\in A_{\mathrm{mv},E}^{\dagger,s^{-}}$, $\lVert f_0\rVert_{s}=1$, then $\lVert f\rVert_{s}=\lVert Y_{\sigma_{0}}^k\rVert_{s}=p^{-k/s}$ and $|\iota(f)|_{1/s}=|\iota(Y_{\sigma_{0}})|_{1/s}^k=p^{-k}$. This finishes the proof of (\ref{equality of two norms}).
	\end{proof}
	
		Using \cite[Lemma 2.14]{berger2020rigid}, we finally check that the $\mathcal{O}_K^{\times}$-action on $A_{\mathrm{mv},E}^{\dagger,s^{-}}\left[\frac{1}{Y_{\sigma_{0}}}\right]$ is locally $\qp$-analytic (for the definition, see for example \cite[$\S$2]{berger2016multivariable}).
		\begin{lem}\label{subring of qp-locally analytic vectors}
			Let $s$ be a positive integer. The $\mathcal{O}_K^{\times}$-action on $A_{\mathrm{mv},E}^{\dagger,s^{-}}\left[\frac{1}{Y_{\sigma_{0}}}\right]$ is locally $\qp$-analytic.
		\end{lem}
		
		\begin{proof}
			Note that $\left(A_{\mathrm{mv},E}^{\dagger,s^{-}}\left[\frac{1}{Y_{\sigma_{0}}}\right]\right)^{\circ}=A_{\mathrm{mv},E}^{\dagger,s^{-}}$, where the topology on $A_{\mathrm{mv},E}^{\dagger,s^{-}}\left[\frac{1}{Y_{\sigma_{0}}}\right]$ is induced by $\lVert\cdot\rVert_s$, hence by \cite[Lemma 2.14]{berger2020rigid}, we need to check that for sufficiently large $m$, for any $\gamma\in 1+p^m\mathcal{O}_K\subseteq \mathcal{O}_K^{\times}$ and any $x\in A_{\mathrm{mv},E}^{\dagger,s^{-}}$, we have
			\begin{align}\label{BSX condition}
				\lVert\gamma(x)-x\rVert_{s}\leq p^{-\frac{1}{p-1}}.
			\end{align}
			Since $\lVert\cdot\rVert_s$ is a multiplicative norm satisfying the strong triangle inequality, for $x,y\in A_{\mathrm{mv},E}^{\dagger,s^{-}}$ and $A_{\mathrm{mv},E}^{\dagger,s^{-}}$ is stable under the $\mathcal{O}_K^{\times}$-action, we have
			\begin{align*}
				\lVert\gamma(xy)-xy\rVert_{s}&=\lVert\gamma(x)\gamma(y)-\gamma(x)y+\gamma(x)y-xy\rVert_{s}\\
				&\leq\max\{\lVert\gamma(x)\gamma(y)-\gamma(x)y\rVert_{s},\lVert\gamma(x)y-xy\rVert_{s}\}\\
				&=\max\{\lVert\gamma(x)\rVert_{s}\lVert\gamma(y)-y\rVert_{s},\lVert\gamma(x)-x\rVert_{s}\lVert\gamma(y)\rVert_{s}\}\\
				&\leq\max\{\lVert\gamma(y)-y\rVert_{s},\lVert\gamma(x)-x\rVert_{s}\}
			\end{align*}
			and
			\begin{align*}
				\lVert\gamma(x+y)-(x+y)\rVert_{s}=\lVert \gamma(x)-x+\gamma(y)-y\rVert_{s}\leq \max\{\lVert \gamma(x)-x\rVert_{s},\lVert\gamma(y)-y\rVert_{s}\}.
			\end{align*}
			Moreover, if $x\in A_{\mathrm{mv},E}^{\dagger,s^{-}}$ is invertible, then $\lVert x\rVert_{s}=1$, and by last assertion of Lemma \ref{||_s and phi,Gamma},
			\begin{align*}
				\lVert\gamma(x^{-1})-x^{-1}\rVert_s=	\lVert \gamma(x)^{-1}x^{-1}(x-\gamma(x))\rVert_s=	\lVert (x-\gamma(x))\rVert_s.
			\end{align*}
			Therefore, using the explicit expression (\ref{defn of s^{-1}}), it suffices to show that for $x\in\left\{Y_{\sigma_{0}},\frac{\pi}{Y_{\sigma_{0}}^s},\frac{Y_{\sigma_i}}{Y_{\sigma_{0}}}\right\}$, there exists $m\geq 1$ such that (\ref{BSX condition}) holds for any $\gamma\in 1+p^m\mathcal{O}_K$. We claim that we can take $m=s$. For $\gamma\in 1+p^{s}\mathcal{O}_K$, by (\ref{refined O_k^*-action}), we can write $\gamma(Y_{\sigma_{i}})=Y_{\sigma_{i}}+pQ_i+P_i$ with $Q_i\in \mathfrak{m}, P_i\in \mathfrak{m}^{p^{s}}$ for $0\leq i\leq f-1$. Then
			for $x=Y_{\sigma_{0}}$, $\gamma(Y_{\sigma_{0}})-Y_{\sigma_{0}}=pQ_0+P_0\in p\mathfrak{m}+\mathfrak{m}^{p^s}$, hence
			\[\lVert\gamma(Y_{\sigma_{0}})-Y_{\sigma_{0}}\rVert_{s}\leq \max\left\{\lVert p\rVert_{s},\lVert Y_{\sigma_{0}}\rVert_{s}^{p^{s}}\right\}=\max\{p^{-e},p^{-p^{s}/s}\}\leq  p^{-\frac{1}{p-1}} ,\]
			where $e$ is the ramification index of $E/\qp$. For $x=\frac{\pi}{Y_{\sigma_{0}}^s}$, $\gamma(x)-x=\gamma(x)(1-\frac{x}{\gamma(x)})$, hence
			\begin{align*}
				\lVert\gamma(x)-x\rVert_{s}&=\left\lVert1-\left(\frac{\gamma(Y_{\sigma_{0}})}{Y_{\sigma_{0}}}\right)^s \right\rVert_s=\left\lVert1-\left(1+p\frac{Q_0}{Y_{\sigma_{0}}}+\frac{P_0}{Y_0}\right)^s \right\rVert_s\\
				&=\left\lVert\sum_{j=1}^{s}\binom{s}{i}\left(p\frac{Q_0}{Y_{\sigma_{0}}}+\frac{P_0}{Y_0}\right)^{i}\right\rVert_s\leq \max_{1\leq i\leq s}\left\lVert p\frac{Q_0}{Y_{\sigma_{0}}}+\frac{P_0}{Y_0}\right\rVert_{s}^i\\
				&=\left\lVert p\frac{Q_0}{Y_{\sigma_{0}}}+\frac{P_0}{Y_0}\right\rVert_{s}\leq\max\{\lVert p \rVert_s, \lVert Y_{\sigma_{0}}\rVert_s^{p^s-1}\}\leq p^{-\frac{1}{p-1}}.
			\end{align*}
			For $x=\frac{Y_{\sigma_{i}}}{Y_{\sigma_{0}}}$, $\gamma(x)-x=\gamma(x)(1-\frac{x}{\gamma(x)})$, thus 
			\begin{align*}
				\lVert\gamma(x)-x\rVert_{s}&=	\left\lVert 1-\frac{x}{\gamma(x)}\right\rVert_{s}=\left\lVert1-\frac{\gamma(Y_{\sigma_{0}})}{Y_{\sigma_{0}}}\frac{Y_{\sigma_{i}}}{\gamma(Y_{\sigma_{i}})} \right\rVert_s\\
				&=\left\lVert1-\frac{1+p\frac{Q_0}{Y_{\sigma_{0}}}+\frac{P_0}{Y_{\sigma_{0}}}}{1+p\frac{Q_i}{Y_{\sigma_{i}}}+\frac{P_i}{Y_{\sigma_{i}}}} \right\rVert_s=\left\lVert\frac{p\frac{Q_i}{Y_{\sigma_{i}}}+\frac{P_i}{Y_{\sigma_{i}}}-\left(p\frac{Q_0}{Y_{\sigma_{0}}}+\frac{P_0}{Y_{\sigma_{0}}}\right)}{1+p\frac{Q_i}{Y_{\sigma_{i}}}+\frac{P_i}{Y_{\sigma_{i}}}} \right\rVert_s.
			\end{align*}
			Since $p\frac{Q_i}{Y_{\sigma_{i}}}+\frac{P_i}{Y_{\sigma_{i}}}$ is an element of $ A_{\mathrm{mv},E}^{\dagger,s^{-}} $ with $\lVert p\frac{Q_i}{Y_{\sigma_{i}}}+\frac{P_i}{Y_{\sigma_{i}}} \rVert_{s} \leq p^{1/(p-1)}$, we see that $ 1+p\frac{Q_i}{Y_{\sigma_{i}}}+\frac{P_i}{Y_{\sigma_{i}}}$ is invertible in $ A_{\mathrm{mv},E}^{\dagger,s^{-}} $ with norm equal to $1$, thus
			\[\lVert\gamma(x)-x\rVert_{s}=\left \lVert p\frac{Q_i}{Y_{\sigma_{i}}}+\frac{P_i}{Y_{\sigma_{i}}}-\left(p\frac{Q_0}{Y_{\sigma_{0}}}+\frac{P_0}{Y_{\sigma_{0}}}\right) \right\rVert_{s}\leq  p^{-\frac{1}{p-1}}.\]
			Therefore, (\ref{BSX condition}) holds for any $\gamma\in 1+p^s\mathcal{O}_K$.
		\end{proof}
		
		As a consequence, $\bigcup_{n\geq0}\varphi_q^{-n}\left(A_{\mathrm{mv},E}^{\dagger,q^ns^{-}}\left[\frac{1}{Y_{\sigma_{0}}}\right]\right)$ is contained in the subring of locally $\mathbb{Q}_p$-analytic vectors of $B_{A_{\infty},[0,1/s]}$ for the $\mathcal{O}_K^{\times}$-action. It is not clear to us whether this is an equality or a strict inclusion.

	\appendix
	\numberwithin{equation}{subsection}
	\section{Some technical lemmas}\label{appendix}

	\subsection{Finiteness of $A^{\circ}$-submodules of étale $(\varphi_q,\mathcal{O}_K^{\times})$-modules over $A$}
	
	Let $A\coloneq A_{\mathrm{mv},E}/p$, then $A^{\circ}$ and $A$ are Noetherian domains, and the ideal of topological nilpotent elements of $A^{\circ}$ is $A^{\circ\circ}=Y_{\sigma_0}\cdot A^{\circ}$. Let $v: A\to \mathbb{Z}\cup\{-\infty\}$ be the $Y_{\sigma_0}$-adic valuation of $A$ with $v(Y_{\sigma_0})=1$. If $M$ is an étale $\varphi_q$-module over $A$, by \cite[Corollary 4.9]{du2025multivariablevarphiqmathcaloktimesmodulesassociatedpadic}, $M$ is a free $A$-module of finite rank. We define
	\begin{align}\label{defn of M^circ}
		M^{\circ}\coloneq\left\{x\in M: \sum_{i=0}^{\infty}A^{\circ}\varphi_q^{i}(x)\ \text{is a finite}\ A^{\circ}\text{-module}\right\}
	\end{align}
	which is a $A^{\circ}$-submodule of $M$. In other words, if we take
	\[F(M)\coloneq\left\{N\subseteq M: N\ \text{is a finite}\ A^{\circ}\text{-module}, \varphi_q(N)\subseteq N\right\},\]
	then
	\label{equivalent description}
	\begin{align}
		M^{\circ}=\bigcup_{N\in F(M)}N.
	\end{align}
	
	\begin{lem}\label{A^{circ}-submodule is finite}
		The $A^{\circ}$-module $M^{\circ}$ is finitely generated and $M=M^{\circ}\left[\frac{1}{Y_{\sigma_{0}}}\right]$.
	\end{lem}
	\begin{proof}
		Let $d\coloneq \mathrm{rank}_AM$.
		If $d=1$, then $M=Ae$, $\varphi_q(e)=\lambda e$ for some $\lambda\in A^{\times}$.  Replacing $e$ by $Y_{\sigma_0}^ne$ for some $n\geq 0$, we may assume that $\lambda \in A^{\circ}$, and let $r\coloneq v(\lambda)\geq 0$. We shall prove that
		\begin{align}\label{inclusion}
			A^{\circ}e\subseteq M^{\circ}\subseteq \dfrac{1}{Y_{\sigma_0}^{r}}A^{\circ}e
		\end{align}
		which implies the finiteness of $M^{\circ}$ since $A^{\circ}$ is Noetherian. The first inclusion is obvious. For the second inclusion, we shall prove by contradiction. Suppose that $M^{\circ}\nsubseteq \frac{1}{Y_{\sigma_0}^{r}}A^{\circ}e$, then there exists $x\in M^{\circ}$ with $x\notin \frac{1}{Y_{\sigma_0}^{r}}A^{\circ}e$. Since $M=Ae$, we have $x=te$ for some $t\in A$ with $v(t)\leq -r-1$. Then
		\[\varphi_q^{i}(x)=\varphi_q^{i}(t)\prod_{j=0}^{i-1}\varphi_q^j(\lambda) e,\quad i\geq 0.\]
		By the definition of $M^{\circ}$, $\sum_{i=0}^{\infty}A^{\circ}\varphi_q^{i}(x)=\bigcup_{i=1}^{\infty}\sum_{j=0}^{i}A^{\circ}\varphi_q^{j}(x)$ is a finite $A^{\circ}$-module. Since $A^{\circ}$ is Noetherian, there exists $i\geq 1$ such that $\varphi_q^{i}(x)\in \sum_{k=0}^{i-1}A^{\circ}\varphi_q^{k}(x)$. But this is impossible: 
		If
		\[\varphi_q^{i}(x)=\sum_{k=0}^{i-1}a_k\varphi_q^{k}(x),\quad a_0,\dots,a_{i-1}\in A^{\circ},\]
		then
		\begin{align}\label{impossible equality}
			\varphi_q^{i}(t)\prod_{j=0}^{i-1}\varphi_q^j(\lambda)=\sum_{k=0}^{i-1}a_k\varphi_q^k(t)\prod_{j=0}^{k-1}\varphi_q^j(\lambda).
		\end{align}
		But for $k=0,\dots,i-1$,
		\begin{align*}
			v\left(\varphi_q^{i}(t)\prod_{j=0}^{i-1}\varphi_q^j(\lambda)\right)&=q^{i}v(t)+\dfrac{q^{i}-1}{q-1}v(\lambda)\\
			&<q^{k}v(t)+\dfrac{q^{k}-1}{q-1}v(\lambda)=v\left(\varphi_q^{k}(t)\prod_{j=0}^{k-1}\varphi_q^j(\lambda)\right)\\
			&\leq v\left(\varphi_q^{k}(t)\prod_{j=0}^{k-1}\varphi_q^j(\lambda)\right)+v(a_k)=v\left(a_k\varphi_q^{k}(t)\prod_{j=0}^{k-1}\varphi_q^j(\lambda)\right),
		\end{align*}
		hence there is no $a_0,\dots,a_{i-1}\in A^{\circ}$ such that (\ref{impossible equality}) holds. Therefore, (\ref{inclusion}) is true.
		
		For general $d\geq 1$, let $e_1,\dots,e_d$ be an $A$-basis of $M$, and let $P=(p_{ij})_{i,j}$ be the matrix of $\varphi_q\in \mathrm{End}(M)$ with respect to $e_1,\dots,e_d$. Replacing $e_1,\dots,e_d$ by $Y_{\sigma_0}^ne_1,\dots,Y_{\sigma_0}^ne_d$, we may assume that $P\in\mathrm{M}_d(A^{\circ})$, hence $\lambda\coloneq \det P \in A^{\circ}$. Let $r\coloneq v(\lambda)\geq 1$. Then $\wedge^d_{A}M=A e_1\wedge\cdots\wedge e_d$ is an étale $\varphi_q$-module over $A$, since $\varphi_q(e_1\wedge\cdots\wedge e_d)=\lambda e_1\wedge\cdots\wedge e_d$. Thus the proof for $d=1$ implies that
		\begin{align}\label{wedge is bounded}
			A^{\circ}e_1\wedge\cdots\wedge e_d\subseteq \left(\wedge^d_{A}M\right)^{\circ}\subseteq \dfrac{1}{Y_{\sigma_0}^r}A^{\circ}e_1\wedge\cdots\wedge e_d.
		\end{align}
		Let $N\in F(M)$ be a finite $A^{\circ}$-submodule of $M$ with $\varphi_q(N)\subseteq N$. Replacing $N$ by $N+\oplus_{i=1}^dA^{\circ}e_i$, we may assume that $\oplus_{i=1}^dA^{\circ}e_i\subseteq N$. We define
		\[\wedge^dN\coloneq\sum_{x_1,\dots,x_d\in N}A^{\circ}x_1\wedge\cdots\wedge x_d\subseteq \wedge^d_{A}M,\]
		then
		$\varphi_q(\wedge^dN)\subseteq\wedge^dN$ and $\wedge^dN$ is finitely generated over $A^{\circ}$. This implies that $\wedge^dN\subseteq \left(\wedge^d_{A}M\right)^{\circ}$.
		Let
		$x=\sum_{i=1}^{d}a_i e_i\in N$, $a_i\in A$. For $1\leq i\leq d$, we have
		\[e_1\wedge\cdots\wedge \widehat{e}_i\wedge\cdots \wedge e_d\wedge x=(-1)^{d-i}a_ie_1\wedge\cdots\wedge e_d\in \wedge^dN\subseteq \left(\wedge^d_{A}M\right)^{\circ}.\]
		Thus $a_i\in \frac{1}{Y_{\sigma_0}^r}A^{\circ}$ for every $i$ by (\ref{wedge is bounded}), i.e. $x\in \frac{1}{Y_{\sigma_0}^r}\left(\oplus_{i=1}^dA^{\circ}e_i\right)$, hence $N\subseteq\frac{1}{Y_{\sigma_0}^r}\oplus_{i=1}^dA^{\circ}e_i$. Using (\ref{equivalent description}), we deduce that $M^{\circ}\subseteq \frac{1}{Y_{\sigma_0}^r}\bigoplus_{i=1}^dA^{\circ}e_i$, hence $M^{\circ}$ is a finite $A^{\circ}$-module.
	\end{proof}

	\subsection{{É}taleness of submodules of étale $(\varphi_q,\mathcal{O}_K^{\times})$-modules over $A$}
	
	\begin{lem}\label{phi_q is injective implies etaleness}
		Let $M$ be a finite module over $A$ endowed with a semi-linear endomorphism $\varphi_q$ and a semi-linear $\mathcal{O}_K^{\times}$-action commuting with $\varphi_q$. If $A\otimes_{\varphi_q,A}M\to M,x\otimes m\mapsto x\varphi_q(m)$ is injective, then $M$ is an étale $(\varphi_q,\mathcal{O}_K^{\times})$-module.
	\end{lem}
	\begin{proof}
		By \cite[Proposition 3.1.1.8]{breuil2023conjectures}, $M$ is a finite projective $A$-module, hence is free of rank $r$ for some $r$ (\cite[Satz 3, p. 131]{lutkebohmert1977vektorraumbundel}). We fix a basis $e_1,\dots,e_r$ of $M$, and let $P=(p_{ij})_{i,j}$ be the matrix of $\varphi_q$ with respect to $e_1,\dots,e_r$. We need to prove that $\lambda\coloneq\det P$ is invertible in $A$. The assumption that $A\otimes_{\varphi_q,A}M\to M$ is injective implies $\lambda\neq 0$. Consider $\wedge^rM=Ae_1\wedge\dots\wedge e_r$, we have $\varphi_q(e_1\wedge\dots\wedge e_r)=\lambda e_1\wedge\dots\wedge e_r$. Let $\mathcal{O}_K^{\times}$ act on $\wedge^rM$ diagonally, for every $a\in \mathcal{O}_K^{\times}$, there exists $c_a\in A^{\times}$ such that $a(e_1\wedge\dots\wedge e_r)=c_a e_1\wedge\dots\wedge e_r$.
		Since $\varphi_q$ commutes with the $\mathcal{O}_K^{\times}$-action, we have
		\[\lambda \varphi_q(c_a)= a(\lambda) c_a,\quad c_a\in A^{\times}, a\in \mathcal{O}_K^{\times}.\]
		This implies that the ideal $\lambda A$ is stable under the $\mathcal{O}_K^{\times}$-action, hence $\lambda A=A$ (\cite[Corollary 3.1.1.7]{breuil2023conjectures}), hence $\lambda$ is invertible.
	\end{proof}
	
	\begin{cor}\label{submodule of etale module over A is etale}
		If $M$ is an étale $(\varphi_q,\mathcal{O}_K^{\times})$-module over $A$ and $N\subseteq M$ is a submodule stable under $\varphi_q$ and the $\mathcal{O}_K^{\times}$-action, then $N$ is also an étale $(\varphi_q,\mathcal{O}_K^{\times})$-module over $A$.
	\end{cor}
	\begin{proof}
		Since $M$ is finite and $A$ is Noetherian, $N$ is also a finite $A$-module. As $\varphi_q: A\to A$ is flat, we see that $A\otimes_{\varphi_q,A}N\subseteq A\otimes_{\varphi_q,A}M$. Since $M$ is étale, the map $A\otimes_{\varphi_q,A}M\to M$ is an isomorphism. This implies that $A\otimes_{\varphi_q,A}N\to N$ is injective, then we apply Lemma \ref{phi_q is injective implies etaleness} to conclude.
	\end{proof}
	
	\begin{lem}\label{equal rank implies equality}
		Let $M,N$ be étale $(\varphi_q,\mathcal{O}_K^{\times})$-modules over $A$ with $N\subseteq M$. If $\mathrm{rank}_AN=\mathrm{rank}_AM$, then $M=N$.
	\end{lem}
	\begin{proof}
		Since $\modet_{(\varphi_q,\mathcal{O}_K^{\times})}(A)$ is an abelian category, $M/N$ must be an étale $(\varphi_q,\mathcal{O}_K^{\times})$-module over $A$, hence is a free $A$-module (\cite[Proposition 3.1.1.8]{breuil2023conjectures}), thus
		\[\mathrm{rank}_AM/N=\mathrm{rank}_AM-\mathrm{rank}_AN=0,\]
		thus $M/N=0$.
	\end{proof}

	\subsection{Flatness of inclusions}
	\begin{lem}\label{inclusion for s is flat}
		Let $s\geq 1$. Then the inclusion $A_{\mathrm{mv},E}^{\dagger,s}\hookrightarrow A_{\mathrm{mv},E}$ is flat.
	\end{lem}
	\begin{proof}
		Note that
		\begin{align}
			\dfrac{A_{\mathrm{mv},E}^{\dagger,s}}{\frac{\pi}{Y_{\sigma_{0}}^s}A_{\mathrm{mv},E}^{\dagger,s}}\cong A^{\circ},\quad \dfrac{A_{\mathrm{mv},E}}{\frac{\pi}{Y_{\sigma_{0}}^s}A_{\mathrm{mv},E}}=\dfrac{A_{\mathrm{mv},E}}{\pi A_{\mathrm{mv},E}}=A=\bigcup_{n=0}^{\infty}Y_{\sigma_{0}}^{-n}A^{\circ},
		\end{align}
		where the first isomorphism follows from Lemma \ref{A dagger s is noetherian}.
		Thus $\frac{A_{\mathrm{mv},E}}{\frac{\pi}{Y_{\sigma_{0}}^s}A_{\mathrm{mv},E}}$ is a countable union of free $\frac{A_{\mathrm{mv},E}^{\dagger,s}}{\frac{\pi}{Y_{\sigma_{0}}^s}A_{\mathrm{mv},E}^{\dagger,s}}$-modules, hence is a flat $\frac{A_{\mathrm{mv},E}^{\dagger,s}}{\frac{\pi}{Y_{\sigma_{0}}^s}A_{\mathrm{mv},E}^{\dagger,s}}$-module. Consider the short exact sequence of $A_{\mathrm{mv},E}^{\dagger,s}$-modules
		\[0\to A_{\mathrm{mv},E}^{\dagger,s}\xrightarrow{\times \frac{\pi}{Y_{\sigma_{0}}^s}}A_{\mathrm{mv},E}^{\dagger,s}\to \dfrac{A_{\mathrm{mv},E}^{\dagger,s}}{\frac{\pi}{Y_{\sigma_{0}}^s}A_{\mathrm{mv},E}^{\dagger,s}}\to 0,\]
		applying $A_{\mathrm{mv},E}\otimes_{A_{\mathrm{mv},E}^{\dagger,s}}-$, we have a short exact sequence
		\begin{align*}
			\mathrm{Tor}_1^{A_{\mathrm{mv},E}^{\dagger,s}}\left(A_{\mathrm{mv},E},A_{\mathrm{mv},E}^{\dagger,s}\right)\to \mathrm{Tor}_1^{A_{\mathrm{mv},E}^{\dagger,s}}\left(A_{\mathrm{mv},E},\dfrac{A_{\mathrm{mv},E}^{\dagger,s}}{\frac{\pi}{Y_{\sigma_{0}}^s}A_{\mathrm{mv},E}^{\dagger,s}}\right)\to A_{\mathrm{mv},E}
			\xrightarrow{\times \frac{\pi}{Y_{\sigma_{0}}^s}}A_{\mathrm{mv},E},
		\end{align*}
		and the first term of the exact sequence is zero since $A_{\mathrm{mv},E}^{\dagger,s}$ is a free $A_{\mathrm{mv},E}^{\dagger,s}$-module. Thus $\mathrm{Tor}_1^{A_{\mathrm{mv},E}^{\dagger,s}}\left(A_{\mathrm{mv},E},\frac{A_{\mathrm{mv},E}^{\dagger,s}}{\frac{\pi}{Y_{\sigma_{0}}^s}A_{\mathrm{mv},E}^{\dagger,s}}\right)=0$. Then we apply \cite[\href{https://stacks.math.columbia.edu/tag/0AGW}{Tag 0AGW}]{stacks-project} to conclude.
	\end{proof}
	
	Recall the following characterization of flat modules:
	\begin{prop}\label{critterion for flatness}
		Let $R$ be a ring, $M$ be an $R$-module. Then $M$ is flat if and only if $I\otimes_{R}M\to M,r\otimes m\mapsto rm$ is injective for any finitely generated ideal $I\subset R$.
	\end{prop}
	\begin{proof}
		See for example \cite[Theorem 7.7]{matsumura1989commutative}.
	\end{proof}
	
	\begin{prop}\label{inclusion for dagger is flat}
		The inclusion $A_{\mathrm{mv},E}^{\dagger,\infty}\hookrightarrow A_{\mathrm{mv},E}$ is flat. As a consequence, the inclusion $A_{\mathrm{mv},E}^{\dagger}\hookrightarrow A_{\mathrm{mv},E}$ is faithfully flat.
	\end{prop}
	\begin{proof}
		Let $I_{\infty}=\left(x_1,\dots,x_n\right)\subseteq A_{\mathrm{mv},E}^{\dagger,\infty}$ be a finitely generated ideal. Take $s_0\geq 1$ such that $x_1,\dots,x_n\in A_{\mathrm{mv},E}^{\dagger,s_0}$, and let $I_s$ be the ideal of $A_{\mathrm{mv},E}^{\dagger,s}$ generated by $x_1,\dots,x_n$ for $s\geq s_0$. Let $f_{s}$ be the natural map $I_s\otimes_{A_{\mathrm{mv},E}^{\dagger,s}}A_{\mathrm{mv},E}\to A_{\mathrm{mv},E}$ for $s\geq s_0$ or $s=\infty$. Let $t_s$ be the natural map $I_s\otimes_{A_{\mathrm{mv},E}^{\dagger,s}}A_{\mathrm{mv},E}\to I_{\infty}\otimes_{A_{\mathrm{mv},E}^{\dagger,\infty}}A_{\mathrm{mv},E}$ for $s\geq s_0$. The fact that $I_{\infty}=\bigcup_{s\geq s_0} I_{s}$ implies that $I_{\infty}\otimes_{A_{\mathrm{mv},E}^{\dagger,\infty}}A_{\mathrm{mv},E}=\bigcup_{s\geq s_0}\mathrm{im}t_s$. Note that we have a commutative diagram of abelian groups
		\begin{equation*}
			\begin{tikzcd}
				I_s\otimes_{A_{\mathrm{mv},E}^{\dagger,s}}A_{\mathrm{mv},E}\arrow[rd,hook,"f_s"]\arrow[d,"t_s"]&\\
				I_{\infty}\otimes_{A_{\mathrm{mv},E}^{\dagger,\infty}}A_{\mathrm{mv},E}\arrow[r,"f_{\infty}"]&A_{\mathrm{mv},E}.
			\end{tikzcd}
		\end{equation*}
		By Lemma \ref{inclusion for s is flat} and Proposition \ref{critterion for flatness}, $f_{s}$ is always injective.
		Thus for $x\in \ker f_{\infty}$, there exists $s\geq s_0$ and $y\in I_s\otimes_{A_{\mathrm{mv},E}^{\dagger,s}}A_{\mathrm{mv},E}$ such that $x=t_s(y)$, hence $f_s(y)=f_{\infty}(t_s(y))=f_{\infty}(x)=0$, hence $y=0$, hence $x=0$, hence $f_{\infty}$ is injective, hence the inclusion $A_{\mathrm{mv},E}^{\dagger,\infty}\hookrightarrow A_{\mathrm{mv},E}$ is flat by Proposition \ref{critterion for flatness}, hence $A_{\mathrm{mv},E}^{\dagger}\hookrightarrow A_{\mathrm{mv},E}$ is flat. It remains to show that every closed point of $\mathrm{Spec}A_{\mathrm{mv},E}^{\dagger}$ is contained in the image of $\mathrm{Spec}A_{\mathrm{mv},E}\to \mathrm{Spec}A_{\mathrm{mv},E}^{\dagger}$ (\cite[\href{https://stacks.math.columbia.edu/tag/00HQ}{Tag 00HQ}]{stacks-project}). Note that $\pi$ is contained in the Jacobson radical of $A_{\mathrm{mv},E}^{\dagger}$ (Lemma \ref{pA dagger is topologically nilpotent}) and the Jacobson radical of $A_{\mathrm{mv},E}$ (as $A_{\mathrm{mv},E}$ is $\pi$-adically complete), thus it suffices to check that $\mathrm{Spec}A_{\mathrm{mv},E}/\pi\to \mathrm{Spec}A_{\mathrm{mv},E}^{\dagger}/\pi$ is surjective. But this is obvious, since $A_{\mathrm{mv},E}/\pi\cong A\cong A_{\mathrm{mv},E}^{\dagger}/\pi$.
	\end{proof}

	\printbibliography
\end{document}